\documentclass[a4paper,12pt,reqno]{amsart}
\usepackage{amsmath,graphicx,graphics,amssymb}
\usepackage{accents}

\newtheorem{lem}{Lemma}

\newtheorem{theo}{Theorem}

\newtheorem{prop}{Proposition}

\numberwithin{equation}{section}

\newcommand\thickbar[1]{\accentset{\rule{.4em}{.6pt}}{#1}}

\newcommand{\M}{\operatorname{M}}

\newcommand{\h}{\operatorname{H}}
\newcommand{\epf}{\hfill{$\square$}\medskip}

\newcommand{\de}{\operatorname{d}}
\newcommand{\w}{\operatorname{w}}
\newcommand{\ka}{\operatorname{k}}

\newmuskip\pFqskip
\mathchardef\pFcomma=\mathcode`, 

\begin{document}

\title{Tilings of hexagons with a removed triad of bowties}

\author[Mihai Ciucu, Tri Lai and Ranjan Rohatgi]{\box\Adr}

\newbox\Adr
\setbox\Adr\vbox{ 
\vspace{0.5cm} 
\centerline{ \large Mihai Ciucu} \vspace{0.25cm}
\centerline{Department of Mathematics, Indiana University}
\centerline{Bloomington, IN 47401, USA}
\vspace{0.5cm}
\centerline{ \large Tri Lai} \vspace{0.25cm}
\centerline{Department of Mathematics, University of Nebraska -- Lincoln}
\centerline{Lincoln, NE 68588}
\vspace{0.4cm} \centerline{and}  \vspace{0.4cm}
\centerline{ \large Ranjan Rohatgi} \vspace{0.25cm}
\centerline{Department of Mathematics and Computer Science, Saint Mary's College}
\centerline{Notre Dame, IN, 46556}
}

\thanks{M. C. was partially supported by the National Science Foundation DMS grant 1501052;
  T. L. was supported in part  by Simons Foundation Collaboration Grant \# 585923.
}

\begin{abstract} In this paper we consider arbitrary hexagons on the triangular lattice with three arbitrary bowtie-shaped holes, whose centers form an equilateral triangle. The number of lozenge tilings of such general regions is not expected --- and indeed is not --- given by a simple product formula. However, when considering a certain natural normalized counterpart $\thickbar{R}$ of any such region $R$, we prove that the ratio between the number of tilings of $R$ and the number of tilings of $\thickbar{R}$ is given by a simple, conceptual product formula. Several seemingly unrelated previous results from the literature --- including Lai's formula for hexagons with three dents and Ciucu and Krattenthaler's formula for hexagons with a removed shamrock --- follow as immediate special cases of our result.

\end{abstract}

\maketitle

\section{Introduction}


MacMahon's classical formula \cite{MacM} stating that the number of plane partitions that fit in an $x\times y\times z$ box is equal to
\begin{equation}
P(x,y,z)=\prod_{i=1}^x\prod_{j=1}^y\prod_{k=1}^z \frac{i+j+k-1}{i+j+k-2}
\label{eaa}
\end{equation}
has served as motivation and source of inspiration for a considerable amount of work in enumerative combinatorics for the past three decades. Following David and Tomei's \cite{DT} elegant observation that such boxed plane partitions are in one-to-one correspondence with lozenge tilings of a hexagon of sides $x,y,z,x,y,z$ (in cyclic order) on the triangular lattice, a lot of this research has been phrased in terms of lozenge tilings.

Generalizations of MacMahon's formula include \cite{CEKZ}\cite{pp1}\cite{sc}\cite{Vuletic}\cite{ff}\cite{fv}\cite{Lai3dent}\cite{Rohatgi2dent}\cite{2i}\cite{LRshuffle}\cite{Lshuffle}\cite{Byun}.

In this paper we consider a family of regions which generalizes several of the regions involved in the above mentioned previous work in the literature. We call our regions {\it triad hexagons} --- arbitrary hexagons on the triangular lattice with three bowtie-shaped holes arranged in a triad, so that the nodes of the bowties form a lattice triangle.

\begin{figure}[h]
  \centerline{
\hfill
{\includegraphics[width=0.44\textwidth]{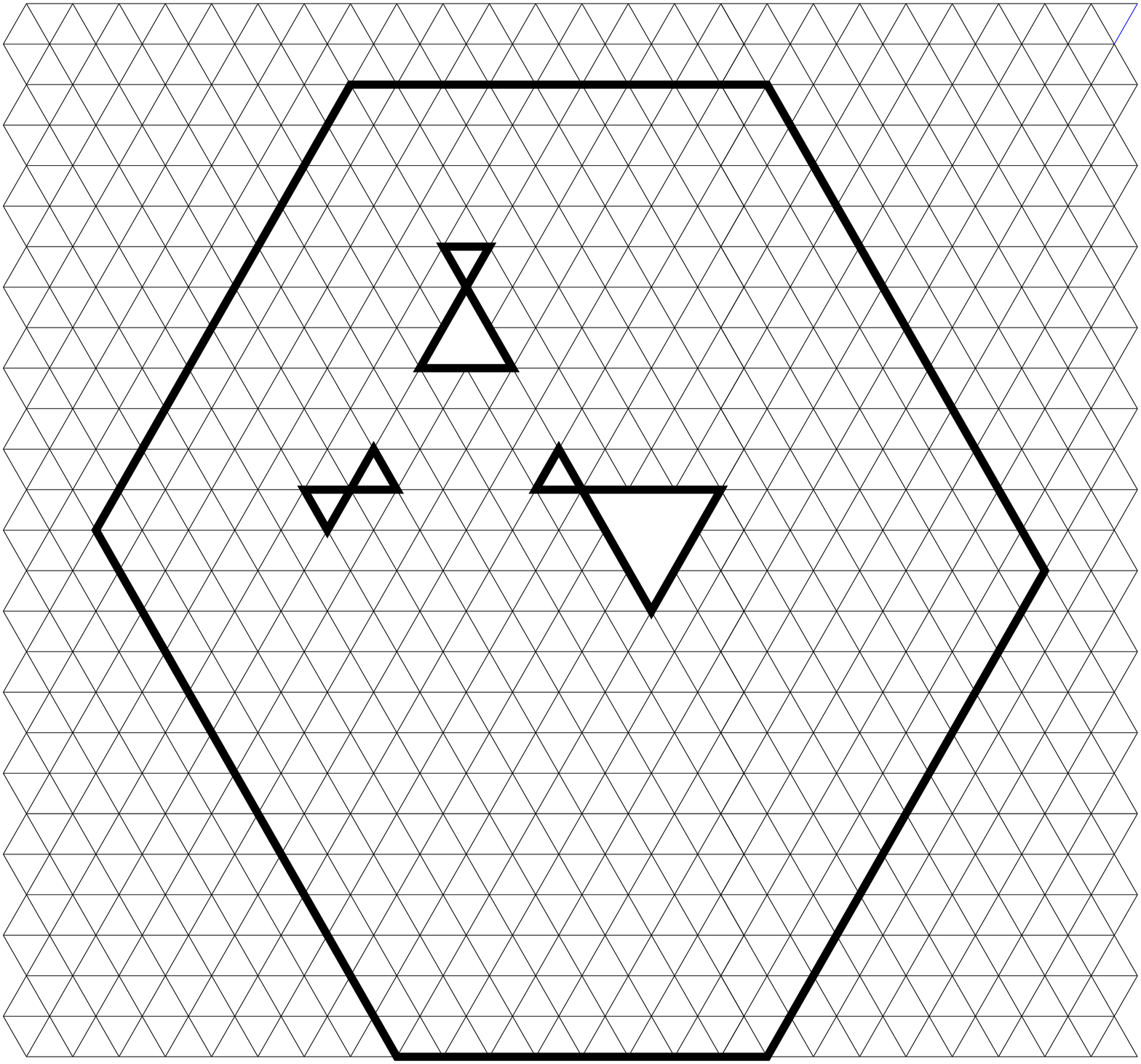}}
\hfill
{\includegraphics[width=0.44\textwidth]{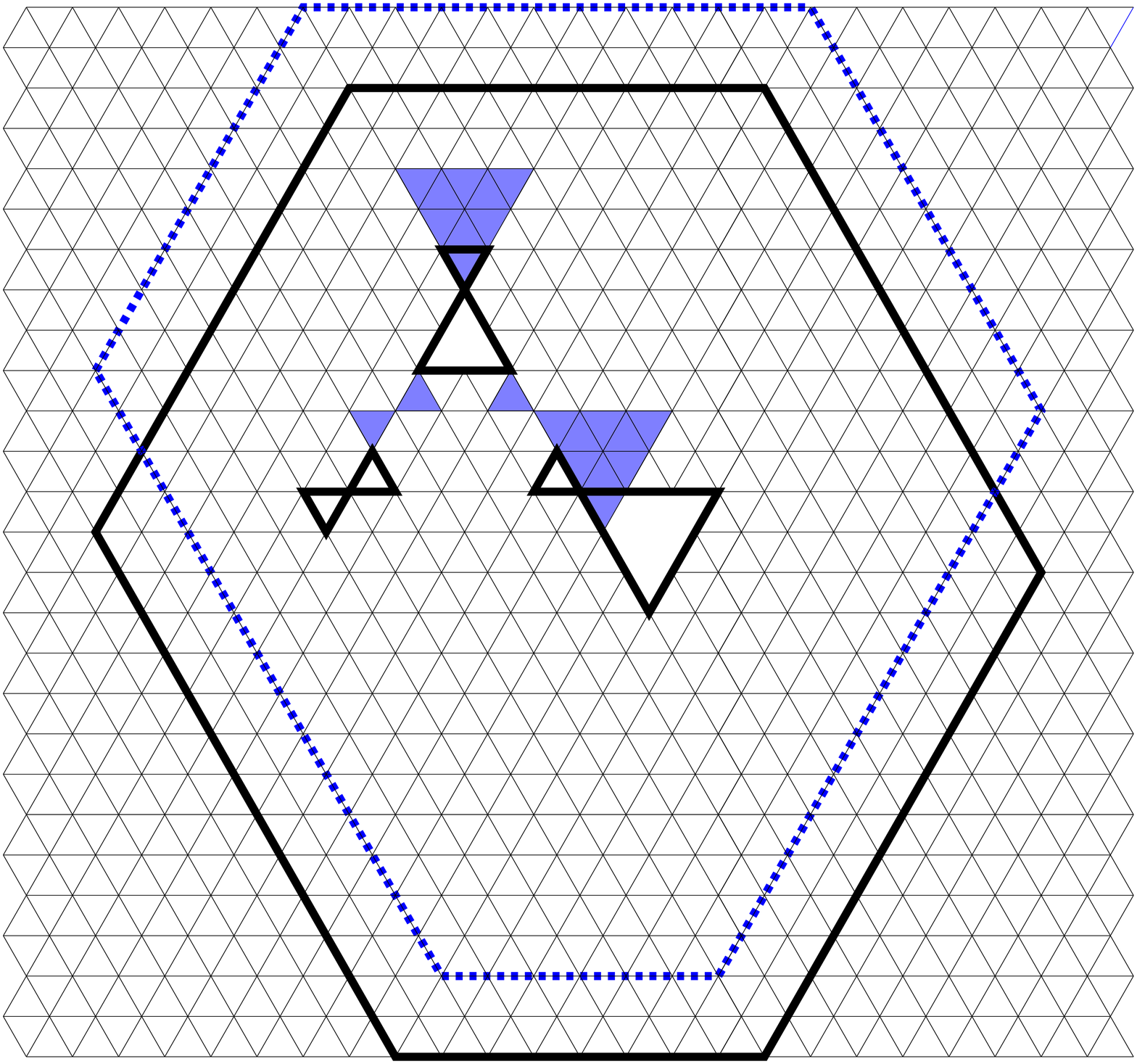}}
}
\vskip0.2in
  \centerline{
\hfill
{\includegraphics[width=0.44\textwidth]{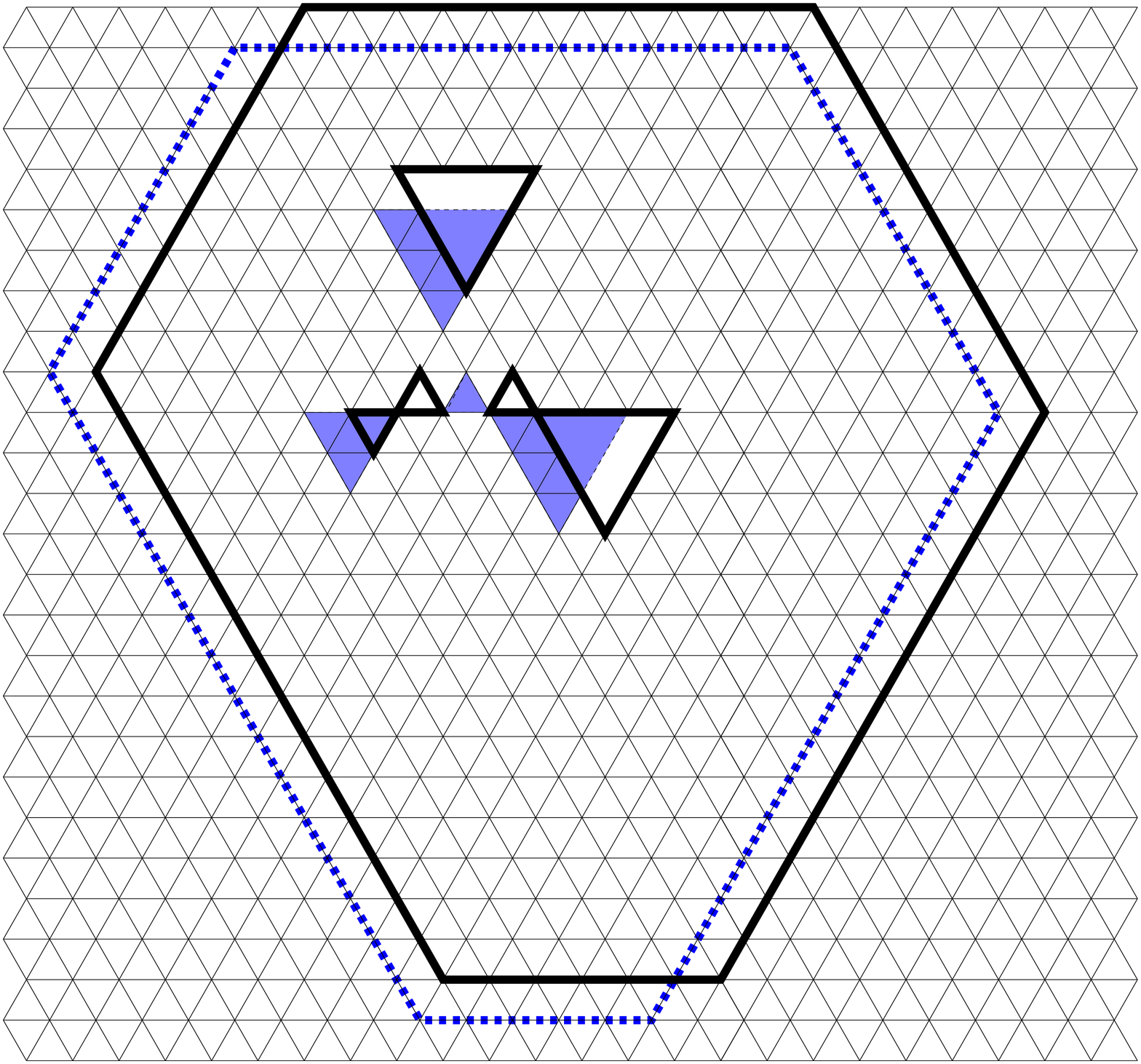}}
\hfill
{\includegraphics[width=0.44\textwidth]{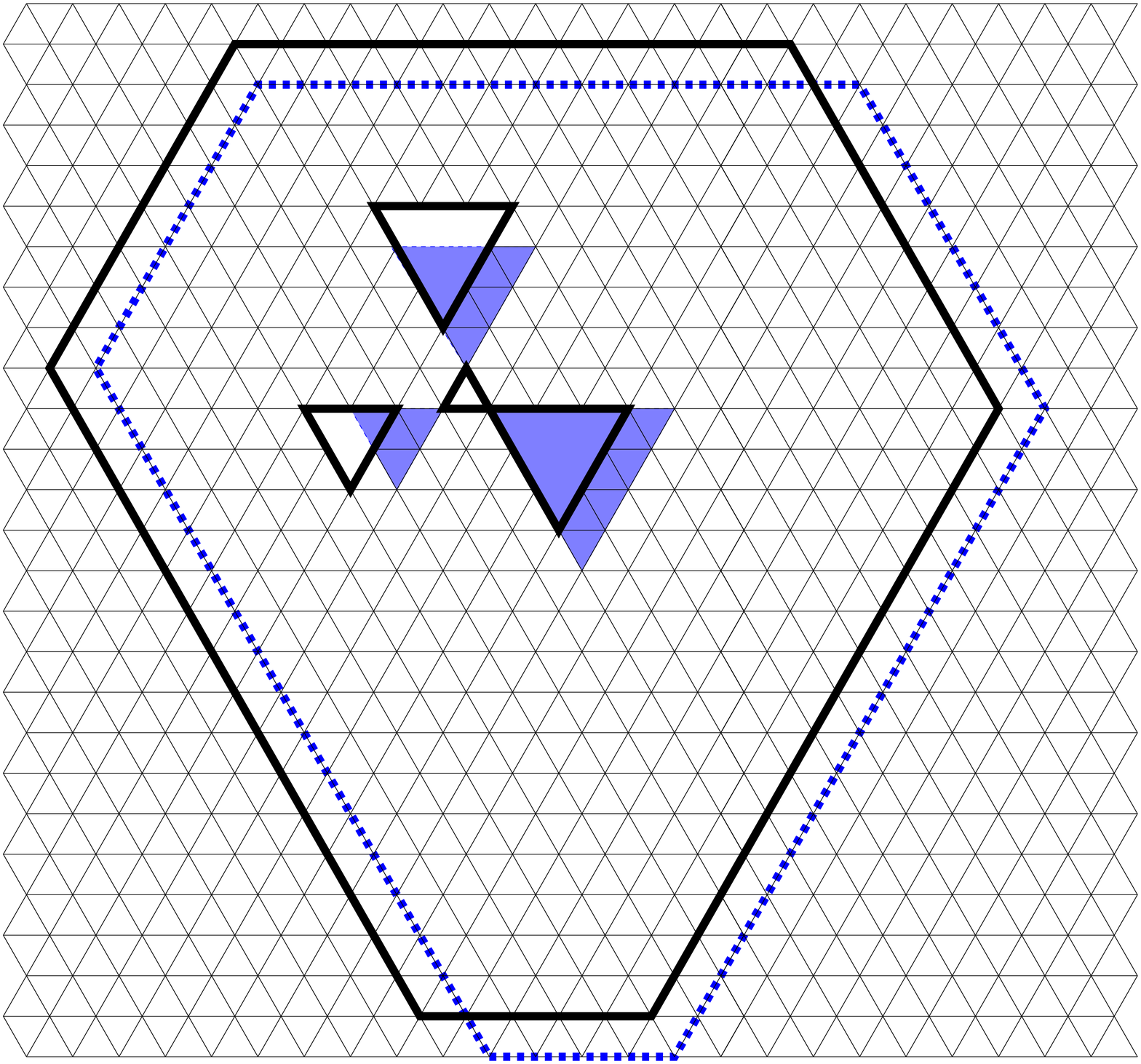}}
}
  \caption{\label{faa} A triad hexagon $R$ with $(x,y,z)=(4,8,7)$, $(a,b,c)=(1,1,3)$ and $(a',b',c')=(2,1,1)$ (top left); squeezing out the top bowtie (top right); squeezing out the left bowtie (bottom left); squeezing out the right bowtie --- the region $\bar{R}$ (bottom right).}
\end{figure}

The main result of this paper is not a product formula for the number of lozenge tilings of a single such triad hexagon (a simple product formula does not seem to exist in general). Instead, we define a natural equivalence relation on the set of triad hexagons based on an operation we call {\it bowtie squeezing}, and we prove that the ratio of the number of tilings of any two regions in the same equivalence class is given by a simple, conceptual product formula.

Several of the mentioned results from the literature, including Lai's formula \cite{Lai3dent} for the number of lozenge tilings of hexagons with three dents, and Ciucu and Krattenthaler's formula \cite{ff} concering hexagons with a removed shamrock, follow as immediate special cases of our result. Given the simple form of our current formula, this point of view helps to understand conceptually the original formulas, which were less structured and more complicated.

\section{Statement of main results}

A {\it bowtie} is a union of two oppositely oriented, not necessarily congruent lattice triangles sharing a vertex, called the {\it node}; a bowtie with down-pointing lobe of side-length $a$ and up-pointing lobe of side-length $a'$ is said to have type $(a,a')$, and is referred to as an $(a,a')$-bowtie. Three bowties form a {\it triad} if their nodes form a lattice triangle housing at each of its three angles one bowtie lobe (see the top left picture in Figure \ref{faa} for in example).

Suppose we remove a triad of bowties, say of types $(a,a')$, $(b,b')$ and $(a,a')$ (counterclockwise from top), from a hexagonal region. It is not hard to see that, provided the resulting region has the same number of up- and down-pointing unit triangles (a necessary condition for tileability by lozenges\footnote{ A lozenge is the union of two unit triangles sharing an edge.}), the side-lengths of the hexagon must be of the form $x+a+b+c$, $y+a'+b'+c'$, $z+a+b+c$, $x+a'+b'+c'$, $y+a+b+c$, $z+a'+b'+c'$, with $x,y,z$ non-negative integers.

Indeed, take a lozenge tiling of our region, and consider in it the $a+b+c$ paths of lozenges that start upward along the horizontal edges of the down-pointing lobes. These must end somewhere along the top side of the hexagon; if the number of unit segments on this side where no such path ends is $x$, then the top side has length $x+a+b+c$. An analogous argument, involving the paths of lozenges starting downward from the horizontal edges of the up-pointing lobes, shows that the bottom side of the hexagon has length $x'+a'+b'+c'$, where $x'$ is some non-negative integer. Because the paths of lozenges that start at the bottom side and do not end at the lobes can only end at places on the top side not connected by paths of lozenges to the lobes, we must have $x'=x$. Repeating this argument for the other two pairs of opposite sides of the hexagon we obtain the claim in the previous paragraph.

If the nodes of the bowties in the triad are at points $A$, $B$ and $C$ (counterclockwise from top), we denote this bowtie by $R_{x,y,z}^{A,B,C}(a,b,c,a',b',c')$. We emphasize that the upper indices denote the geometrical position of the nodes, and not numbers. This hybrid notation between integer parameters and geometric positions is best suited for bringing out the conceptual form of our formulas. We call the points $A$, $B$ and $C$ {\it focal points}, and the segments $AB$, $AC$ and $BC$ {\it focal edges}.

We call the distance $|AB|=|AC|=|BC|$ (measured in unit triangle side lengths) the {\it focal distance} of the triad hexagon, and we denote it by $f$. Note that if $R_{x,y,z}^{A,B,C}(a,b,c,a',b',c')$ can be tiled by lozenges, we must necessarily have $f\geq a'+b'+c'$. One can see this for instance by considering a tiling and following the paths of lozenges that start along the horizontal side of the lobe of size $a'$ of the top bowtie: These $a'$ paths must fit through the gap determined by the bottom two bowties, which has size $f-b'-c'$, so $a'\leq f-b'-c'$, proving our claim.

Therefore, throughout this paper we will assume that the focal distance $f$ of our triad hexagons satisfies $f\geq a'+b'+c'$.


We now define the operation of {\it bowtie squeezing}, which turns a given triad hexagon into another triad hexagon, as follows. Given a triad hexagon $R$, the triad hexagon obtained from $R$ by {\it squeezing out the $(a,a')$-bowtie $d$ units}, where $d\leq a'$, is the region obtained from $R$ by

\bigskip
$(i)$ keeping the node $A$ fixed and replacing the $(a,a')$-bowtie with an $(a+d,a'-d)$-bowtie

\medskip
$(ii)$ translating the $(b,b')$- and $(c,c')$-bowties $d$ units (measured in unit triagle sides) in the $\overrightarrow{BA}$ and $\overrightarrow{CA}$ directions, respectively

\medskip
$(iii)$ pushing out $d$ units (measured in lattice spacings) the top three sides of the hexagon, and pulling in $d$ units the bottom three sides of the hexagon.

\bigskip
The top right picture in Figure \ref{faa} illustrates the operation of squeezing out the top bowtie two units. The resulting triad hexagon has the outer boundary indicated by the thick dotted line, and its removed bowties are shaded (the inner lobe of the resulting top bowtie is empty, as that lobe was completely squeezed out).

The operation of squeezing out the other two bowties is defined by symmetry. The inverse of the described operation is called {\it squeezing in the $(a,a')$-bowtie $d$ units}; it is defined for $d\leq a$.

Note that the difference between the focal length and the sum of the sizes of the inner lobes is invariant under bowtie squeezing: both decrease (resp. increase) by $d$ units when a bowtie is squeezed out (resp., squeezed in) $d$ units. This implies in particular that, since the bowties in $R$ have disjoint interiors, so do the bowties in any triad hexagon obtained from $R$ by a sequence of bowtie squeezings.


One special triad hexagon we get from $R$ is the one obtained by squeezing out completely all three bowties. Figure \ref{faa} shows an example (the top right, bottom left and bottom right pictures illustrate the operation of squeezing out successively the top, left and right bowtie, respectively). We denote the resulting region, in which all three inner lobes have shrunk to zero, by $\thickbar{R}$.


Two triad hexagons are said to be equivalent if one can be obtained from the other by a sequence of bowtie squeezing operations. This is obviously an equivalence relation on the set of triad hexagons.


Our main result is a simple product formula for the ratio of the number of tilings of any two triad hexagons in the same equivalence class. To state it, we need to define the {\it weight} of a triad and the {\it couple} of a focal point and of a focal edge.

Recall that the hyperfactorial $\h(n)$ is defined by $H(0):=1$ and
\begin{equation}
\label{eba}
  \h(n):=0!\,1!\cdots(n-1)!, \ \ \ n\geq1.
\end{equation}

For a triad of bowties of types $(a,a')$, $(b,b')$, $(c,c')$ and focal distance $f$, we define its {\it weight} $\w$ by
\begin{equation}
\label{ebb}
\w:=\frac{\h(f)^4\h(a)\h(b)\h(c)\h(a')\h(b')\h(c')}{\h(f+a)\h(f+b)\h(f+c)\h(f-a')\h(f-b')\h(f-c')}.
\end{equation}

\medskip
For a triad hexagon $R=R_{x,y,z}^{A,B,C}(a,b,c,a',b',c')$, we define the weight $\w^{(R)}$ to be equal to the quantity $\w$ given by \eqref{ebb}. Note that $\w^{(R)}$ depends only on the triad of bowties, and not on the position of the triad inside the hexagon.

We also define the {\it couples of the focal points} $A$, $B$ and $C$,  by
\begin{align}
\label{ebca}
\ka_A^{(R)}&:=\h(\de(A,N))\h(\de(A,S))
\\[5pt]
\label{ebcb}
\ka_B^{(R)}&:=\h(\de(B,NE))\h(\de(B,SW))
\\[5pt]
\label{ebcc}
\ka_C^{(R)}&:=\h(\de(C,NW))\h(\de(C,SE)),
\end{align}
where $\de(A,N)$ denotes the distance between $A$ and the northern side of the outer boundary of $R$ (expressed in lattice spacings), $\de(B,NE)$ is the distance between $B$ and the northeastern boundary, and so on.

Similarly, the {\it couples of the focal segments} $BC$, $AC$ and $AB$ are defined by
\begin{align}
\label{ebda}
\ka_{BC}^{(R)}&:=\h(\de(BC,N))\h(\de(BC,S))
\\[5pt]
\label{ebdb}
\ka_{AC}^{(R)}&:=\h(\de(AC,NE))\h(\de(AC,SW))
\\[5pt]
\label{ebdc}
\ka_{AB}^{(R)}&:=\h(\de(AB,NW))\h(\de(AB,SE)).
\end{align}
It will be helpful in the formulation and proof of our main result to characterize the triad hexagons that are tileable (i.e., admit at least one lozenge tiling). To this end, define the {\it S-depth}, {\it NE-depth} and {\it NW-depth}, of a triad hexagon to be equal to
\begin{align}
&
\de(BC,S)-b-c
\nonumber
\\[5pt]
&
\de(AC,NE)-a-c
\nonumber
\\[5pt]
&
\de(AB,NW)-a-b, 
\nonumber
\end{align}
respectively. Then we have the following characterization.

\begin{figure}[h]
  \centerline{
{\includegraphics[width=0.54\textwidth]{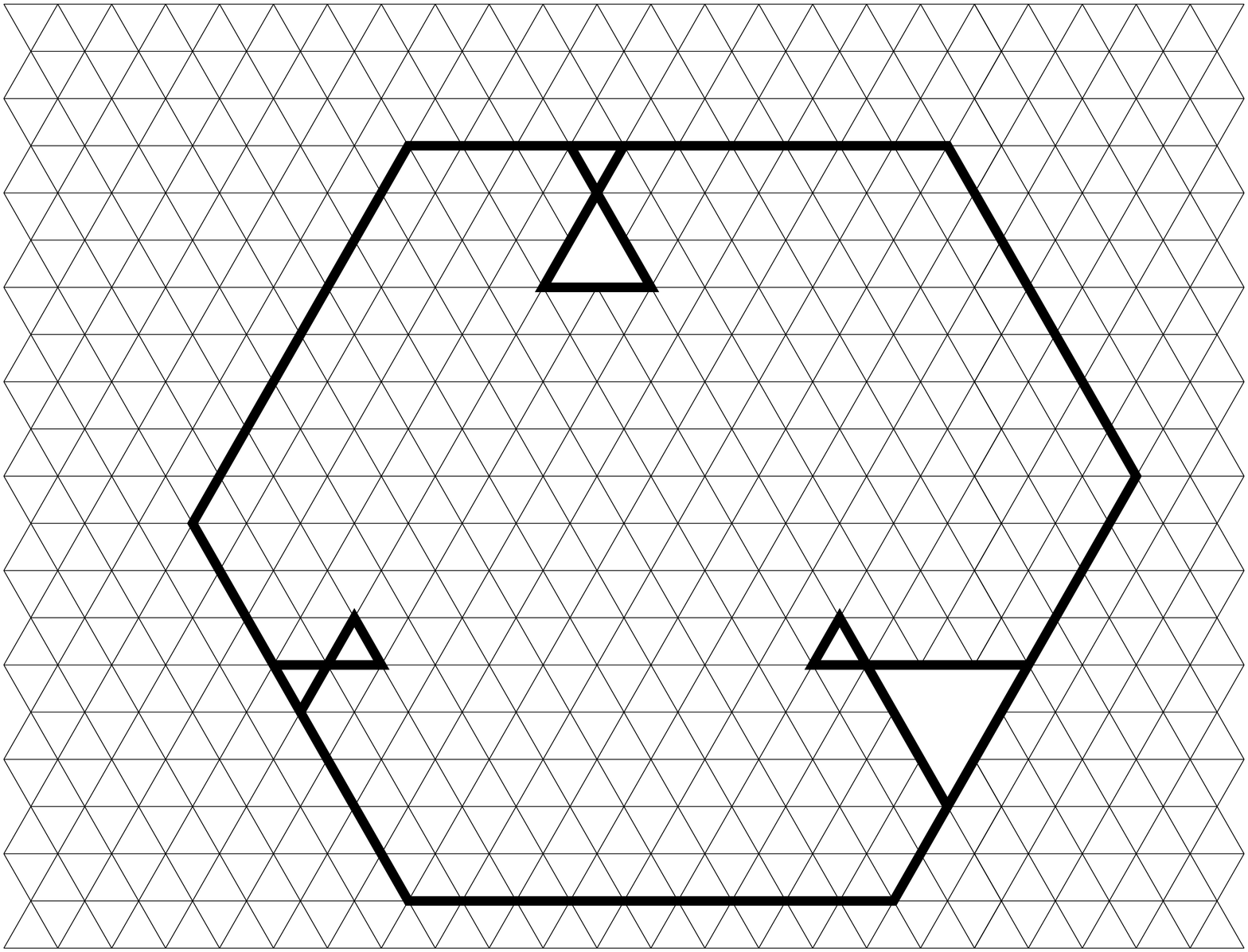}}
\hfill
{\includegraphics[width=0.42\textwidth]{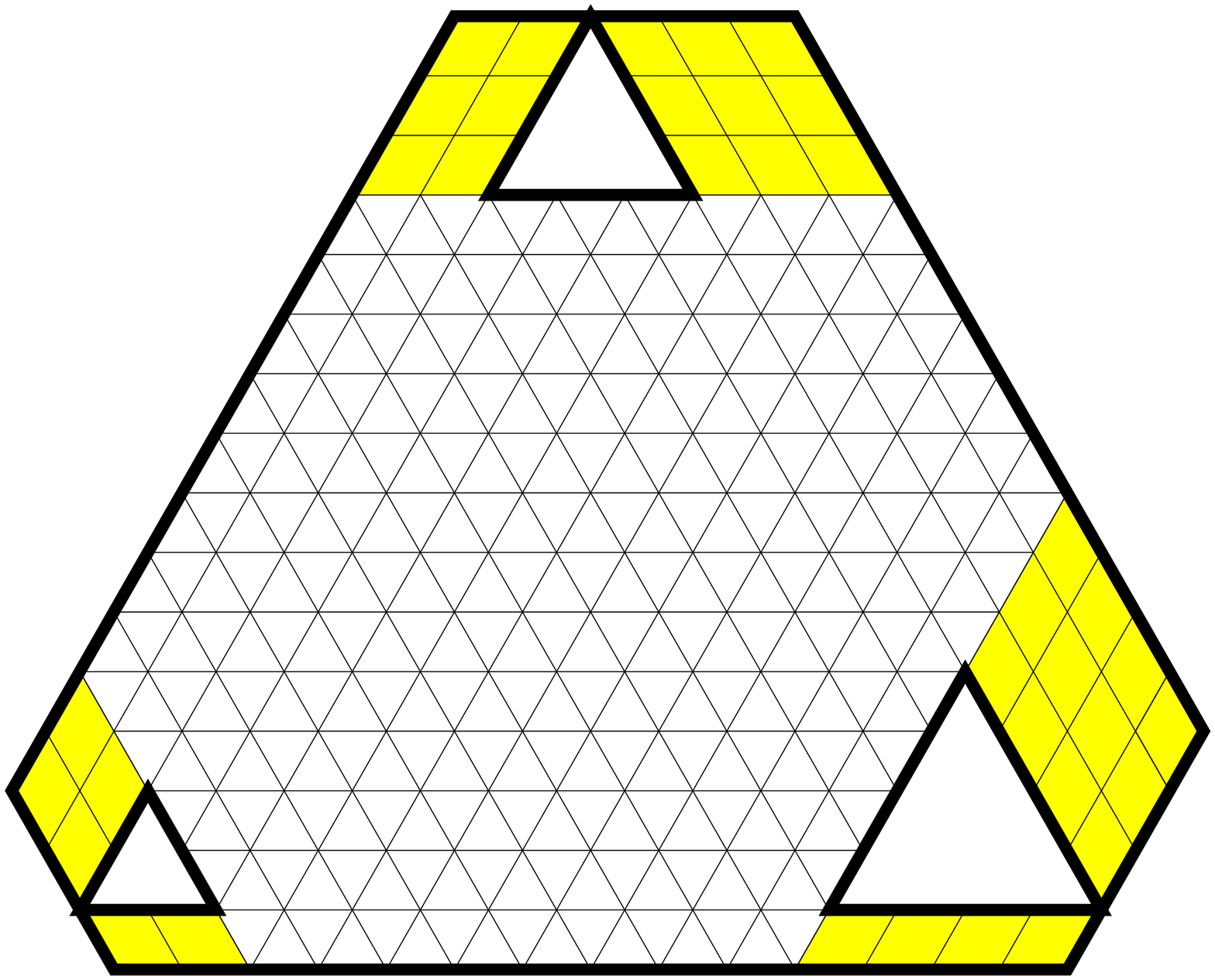}}
}
  \caption{\label{fba} A triad hexagon $R$ with the bowties touching the edges (left); the corresponding region $\bar{R}$ with all bowties completely squeezed in (right).}
\end{figure}

\begin{figure}[h]
\centerline{
{\includegraphics[width=0.33\textwidth]{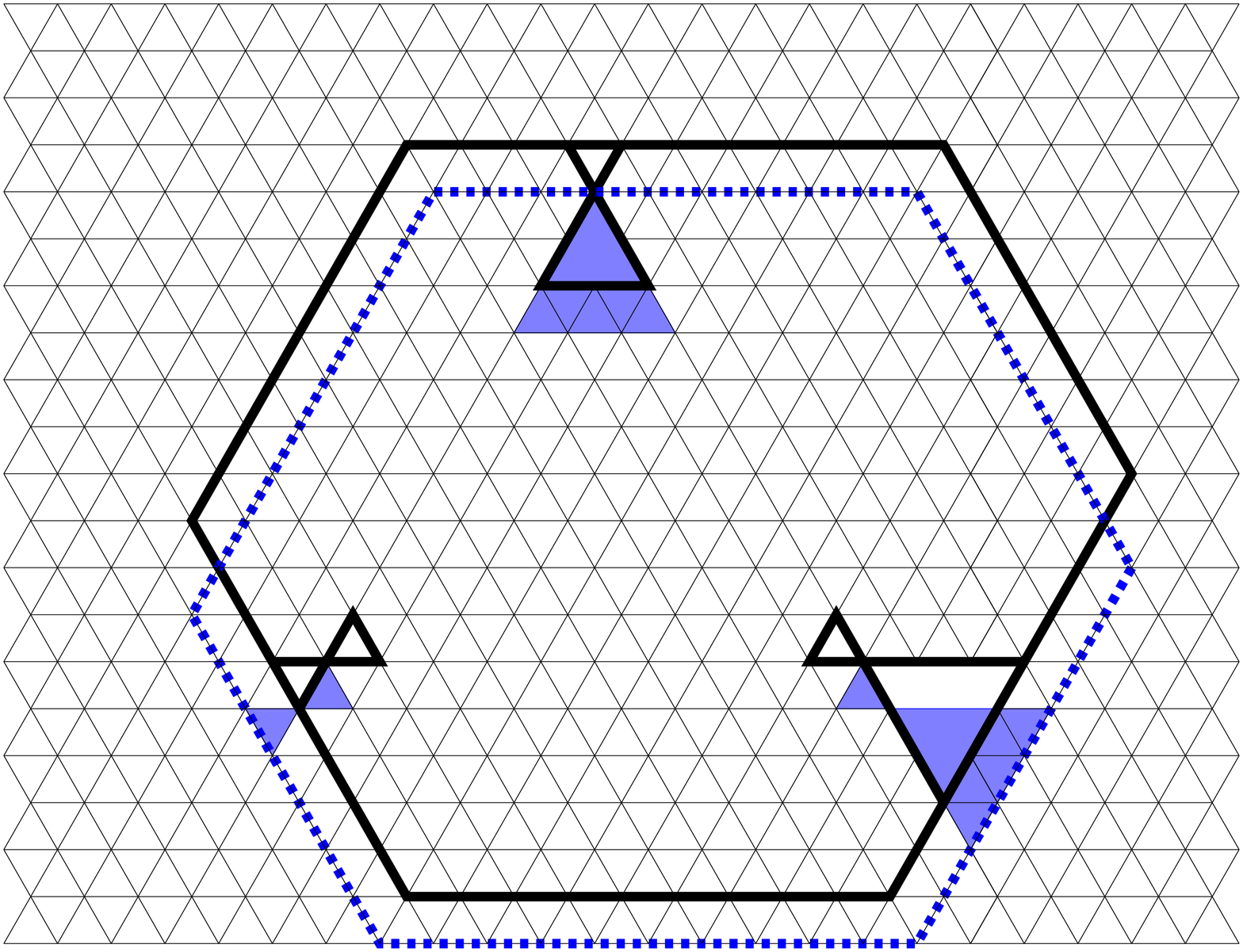}}
\hfill
{\includegraphics[width=0.33\textwidth]{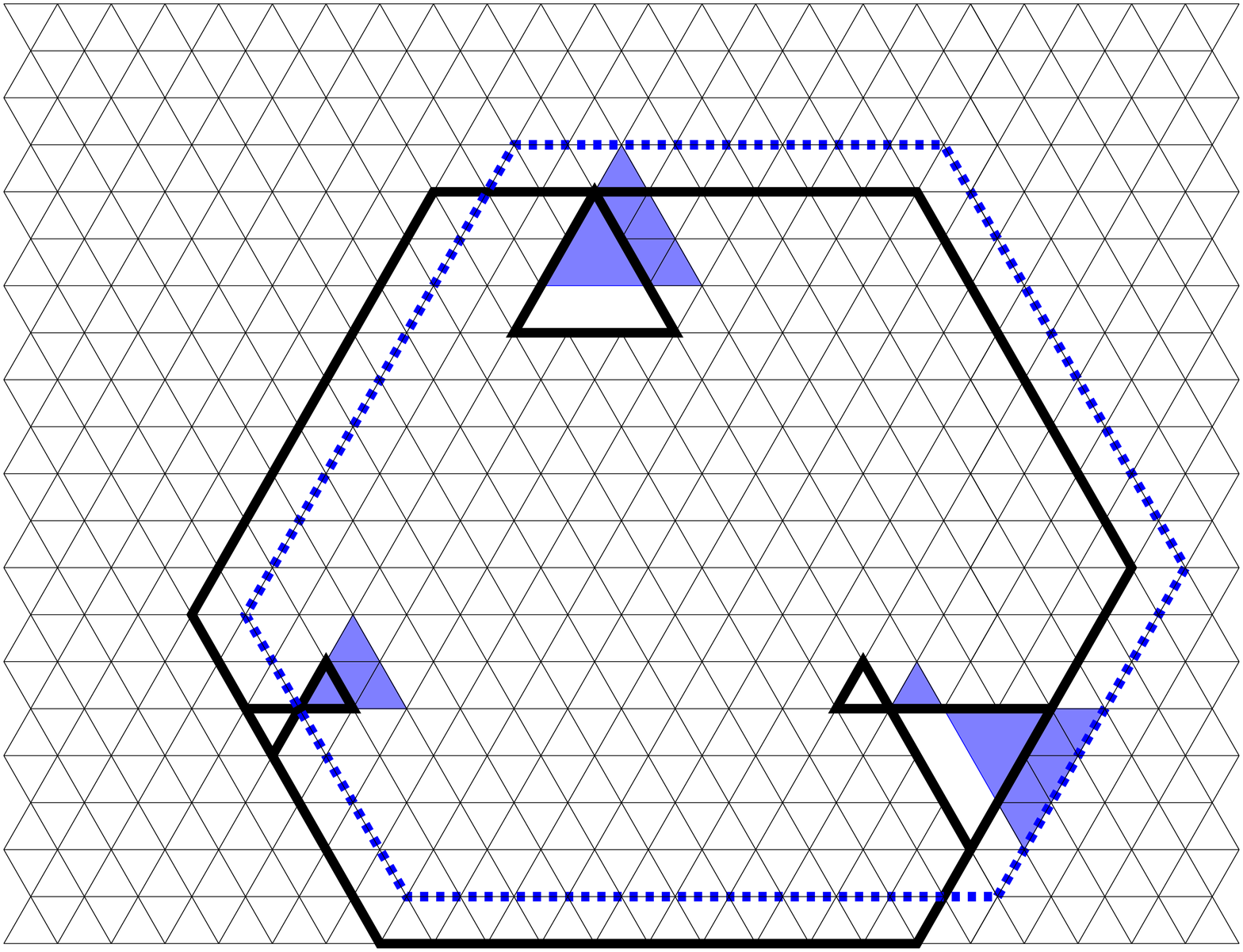}}
\hfill
{\includegraphics[width=0.33\textwidth]{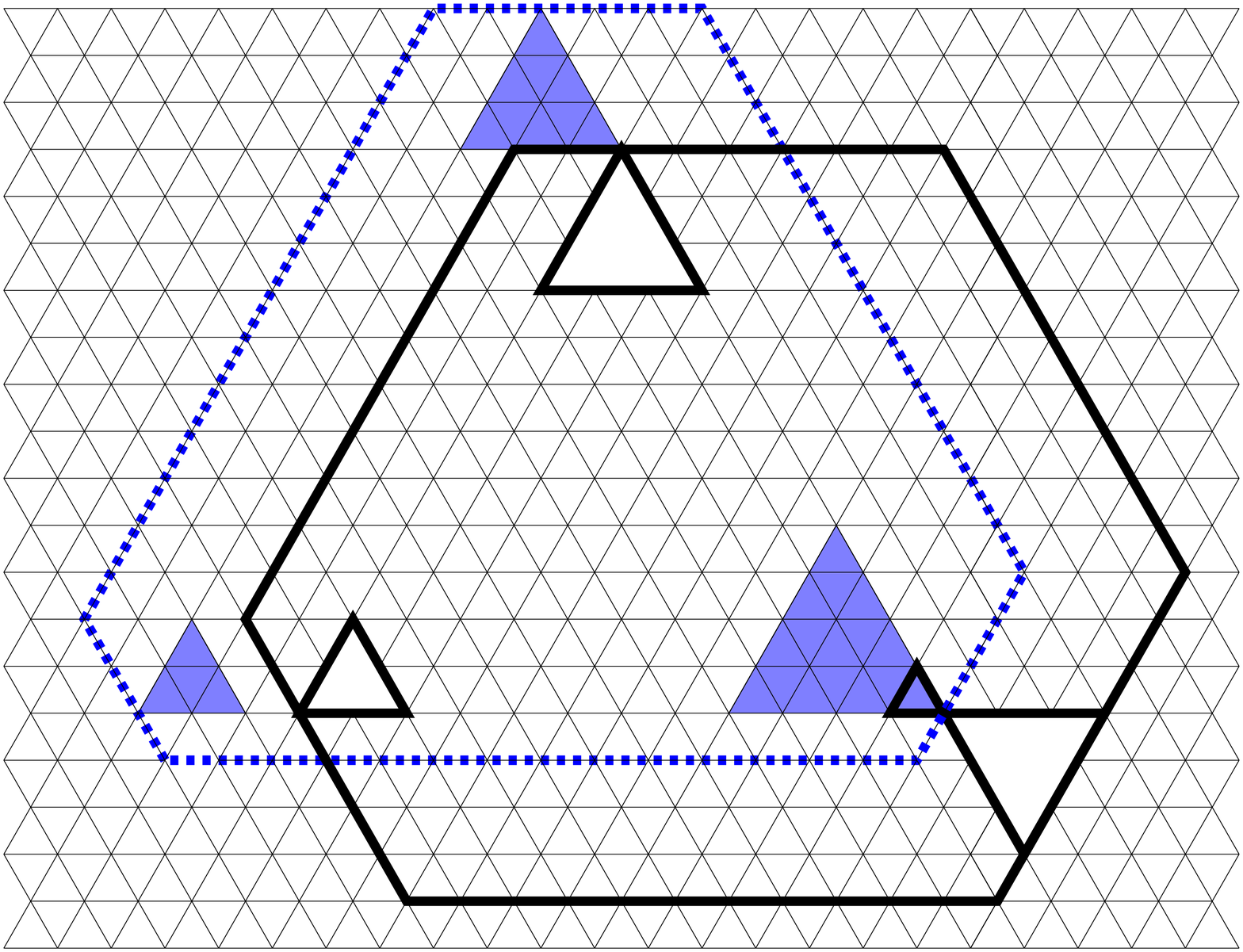}}
}
  \caption{\label{fbb} Succesive squeezing in of the three bowties in Figure \ref{fba}.}
\end{figure}

\begin{lem}
\label{tileability}
$($\emph{a}$)$. A triad hexagon is tileable if and only if its three depths are non-negative.

\ \ \ \ \ \ \ \ \ \ \ \ 
$($\emph{b}$)$. The depths of a triad hexagon are invariant under bowtie squeezing.

\ \ \ \ \ \ \ \ \ \ \ \ \ \,\!\! 
$($\emph{c}$)$. If a triad hexagon is tileable, any triad hexagon obtained from it by a sequence of bowtie squeezings is also tileable.

\end{lem}

\parindent0pt
The main result of this paper is the following.

\parindent15pt
\begin{theo} 
\label{tba}
Let $R=R_{x,y,z}^{A,B,C}(a,b,c,a',b',c')$ be an arbitrary tileable triad hexagon, and let $Q=R_{x_1,y_1,z_1}^{A_1,B_1,C_1}(a_1,b_1,c_1,a_1',b_1',c_1')$ be a triad hexagon obtained from $R$ by a sequence of bowtie squeezings. The we have
\begin{equation}
\label{ebe}
\frac{\M(R)}{\M(Q)}=\dfrac
{\w^{(R)}\dfrac{\ka_A^{(R)}\ka_B^{(R)}\ka_C^{(R)}}{\ka_{BC}^{(R)}\ka_{AC}^{(R)}\ka_{AB}^{(R)}}}
{\w^{(Q)}\dfrac{\ka_{A_1}^{(Q)}\ka_{B_1}^{(Q)}\ka_{C_1}^{(Q)}}{\ka_{B_1C_1}^{(Q)}\ka_{A_1C_1}^{(Q)}\ka_{A_1B_1}^{(Q)}}},
\end{equation}
where the weights $\w$ and the couples $\ka$ are defined by equations \eqref{ebb}--\eqref{ebdc}.

\end{theo}

%

\parindent0pt
\textsc{Remark 1.} Consider the special case when $R$ is a triad hexagon in which the bowties touch the northern, southwestern and southeastern sides of the boundary (see the left picture in Figure~\ref{fba} for an illustration). Let $\thickbar{R}$ be the region obtained from $R$ by completey squeezing in all three bowties. Our definition of the bowtie squeezing operation implies that $\thickbar{R}$ is a hexagon with three triangular holes touching with one of their vertices alternate sides of the boundary (if $R$ is the region on the left in Figure \ref{fba}, $\thickbar{R}$ is pictured on the right in the same figure). The details of the construction are shown in Figure \ref{fbb}. After removing from $\thickbar{R}$ all the lozenges that are forced to be part of each of its tilings, the leftover region is a centrally symmetric hexagon, whose number of tilings is given by MacMahon's formula \eqref{eaa}. Therefore equation \eqref{ebe} yields a product formula for $\M(R)$. This gives the following equivalent form of Lai's earlier result \cite{Lai3dent}.

\begin{theo} \cite[Theorem 1.1]{Lai3dent}
\label{3btouching}
Let $T=T_{x,y,z}^{A,B,C}(a,b,c,a',b',c')$ be the region obtained from the hexagon $H$ of side-lengths $x+a+b+c$, $y+a'+b'+c'$, $z+a+b+c$, $x+a'+b'+c'$, $y+a+b+c$, $z+a'+b'+c'$ $($clockwise from top$)$ by removing bowties of types $(a,a')$, $(b,b')$ and $(c,c')$ from along the northern, southwestern and southeastern sides of $H$, with focal points at $A$, $B$ and $C$, respectively $($see the picture on the left in Figure $\ref{fba}$$)$. Then
\begin{equation}
\M(T_{x,y,z}^{A,B,C}(a,b,c,a',b',c'))=P(x+a+a',y+b+b',z+c+c')
\dfrac
{\w^{(T)}\dfrac{\ka_A^{(T)}\ka_B^{(T)}\ka_C^{(T)}}{\ka_{BC}^{(T)}\ka_{AC}^{(T)}\ka_{AB}^{(T)}}}
{\w^{(T_1)}\dfrac{\ka_{A_1}^{(T_1)}\ka_{B_1}^{(T_1)}\ka_{C_1}^{(T_1)}}{\ka_{B_1C_1}^{(T_1)}\ka_{A_1C_1}^{(T_1)}\ka_{A_1B_1}^{(T_1)}}},
\end{equation}
where $T_1$ is the triad hexagon obtained from $T$ by completely squeezing in the three bowties $($see the picture on the right in Figure $\ref{fbc}$$)$, and $A_1$, $B_1$ and $C_1$ are its top, left and right focal points, respectively.

\end{theo}

\textsc{Remark 2.} Another interesting special case is when the bottom two bowties consist just of their outer lobes (i.e. their inner lobes are empty), and they touch the corners of the inner lobe of the top bowtie (see Figure \ref{fbc} for an example). Let $R$ be such a region, and let $Q$ be the region obtained from $R$ by completely squeezing out the top bowtie (if $R$ is as pictured on the left in Figure \ref{fbc}, the resulting region $Q$ is illustrated in the same figure on the right). Then the bowties in $Q$ consist of single down-pointing lobes, sharing a common vertex. The shaded lozenges indicated in Figure \ref{fbc} are forced. Upon their removal, the leftover region is a hexagon with an equilateral triangle removed from its center (see \cite{CEKZ} for the precise definition of what this central position means). Since  equation \eqref{ebe} holds, and the lozenge tilings of $Q$ (being a hexagon with an equilateral triangle removed from its center) are enumerated by Theorems 1 
%
and 2 in \cite{CEKZ}, we obtain a simple product formula for $\M(R)$. This yields Ciucu and Krattenthaler's earlier result~\cite{ff}.

\parindent15pt

\section{Two known special cases}

In this section we present formulas that give the number of lozenge tilings of two families of regions, both special cases of triad hexagons. We will use these formulas in our proof of Theorem \ref{tba}. Both results are known from the literature. However, the form of the formulas is new --- it is tailored to make our calculations in the proof of Theorem \ref{tba} easier.

The first family of regions, called {\it magnet bar regions}, was introduced in \cite{ff}. The picture on the left in Figure \ref{fea} describes the magnet bar region $I_{x,y}(a,b,c,m)$.

Note that $I_{x,y}(a,b,c,m)$ is a special case of a triad hexagon, with the focal points $A$, $B$ and $C$ being the top, left and right vertices of the triangular dent of side $m$ along the base, and bowties of types $(c,m)$, $(0,0)$ and $(0,0)$, respectively.

The following result was proved in \cite{ff}\footnote{In \cite{ff} we denoted these regions by the letter $B$; to avoid confusion with the focal point $B$, we use here the letter $I$ instead.}.

\begin{figure}[h]
  \centerline{
{\includegraphics[width=0.44\textwidth]{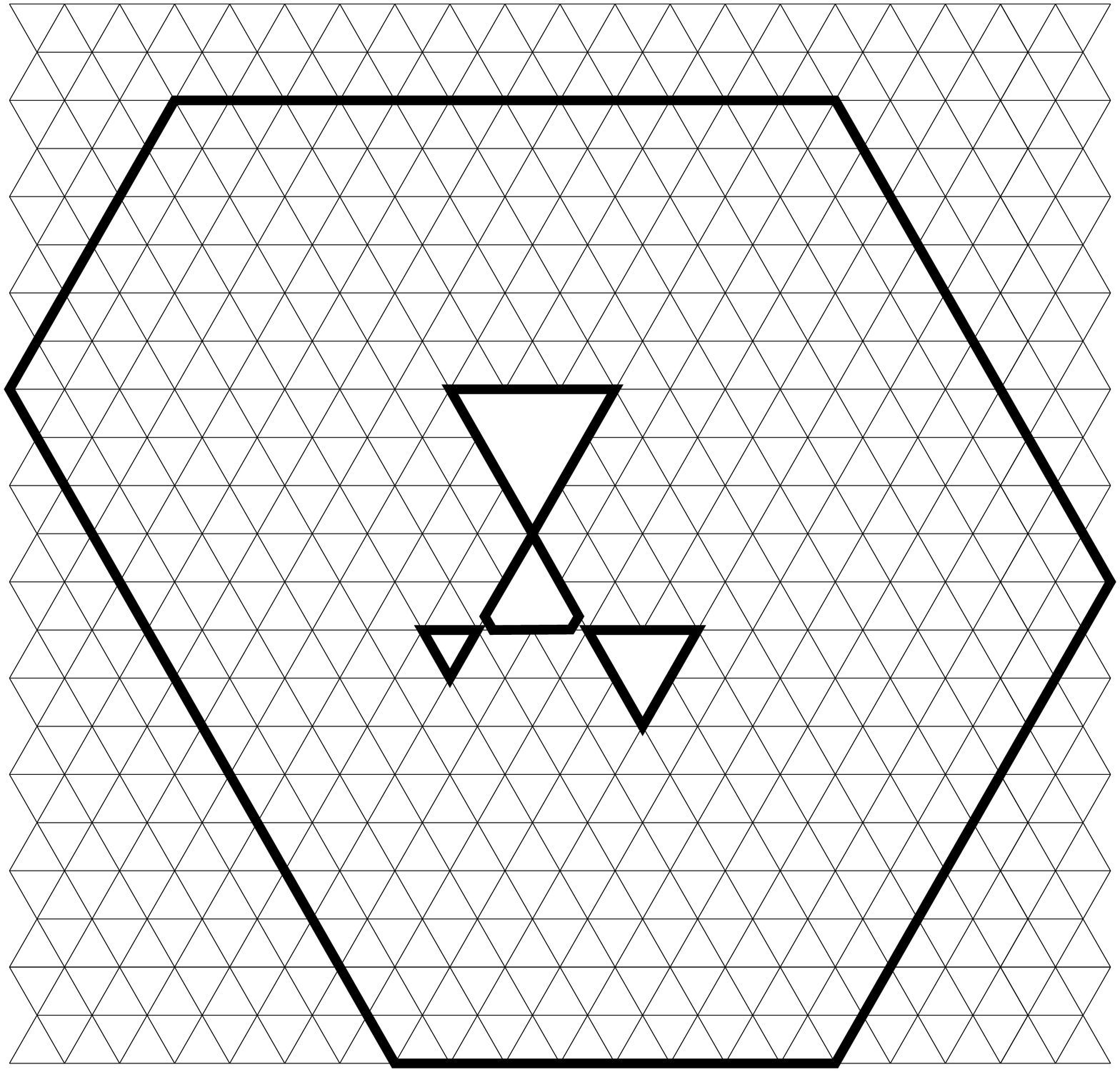}}
\hfill
{\includegraphics[width=0.44\textwidth]{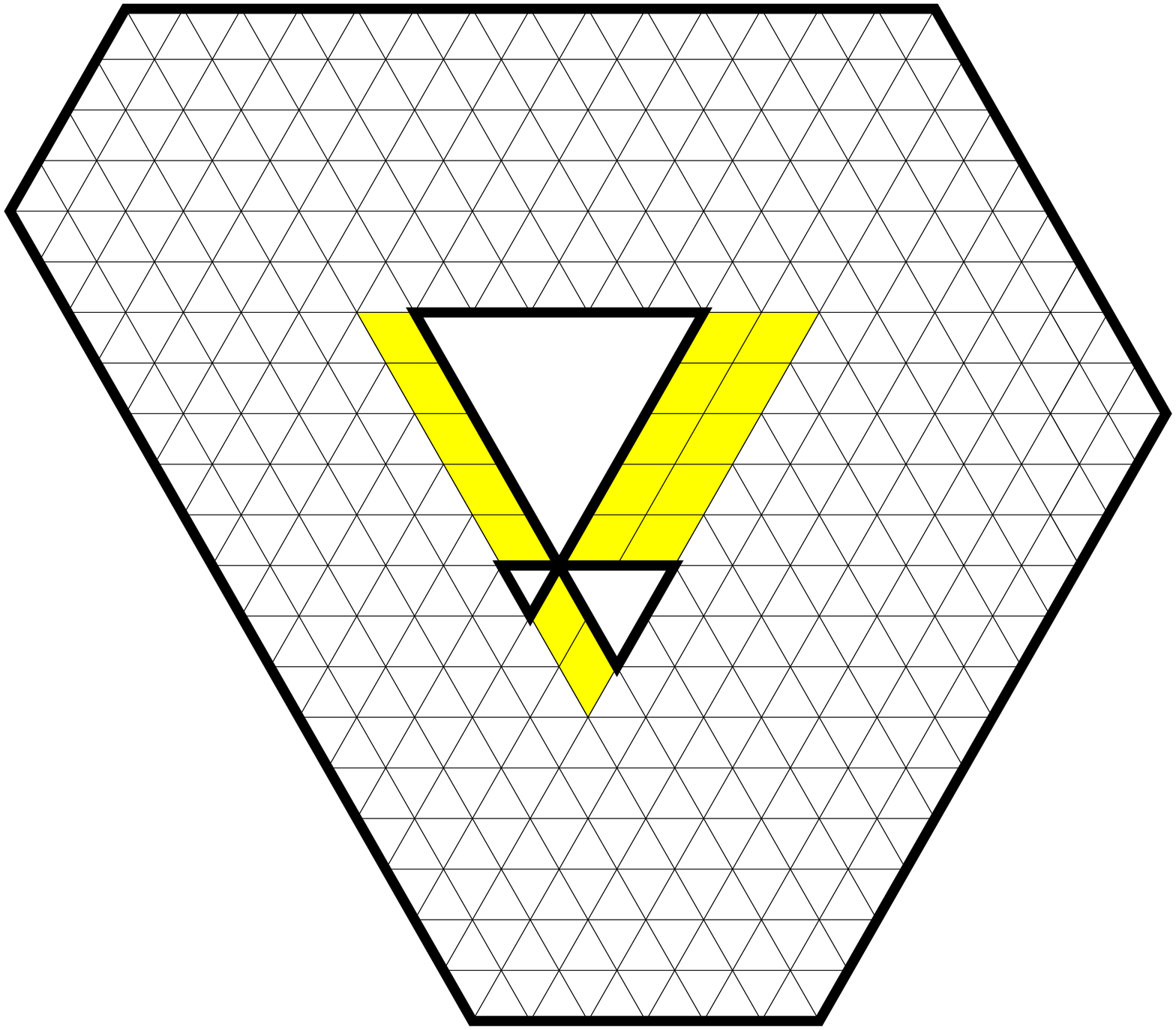}}
}
  \caption{\label{fbc} A triad hexagon $R$ with three bowties forming a shamrock (left); the corresponding region $Q$ with the top bowtie completely squeezed out (right).}
\end{figure}

\begin{figure}[h]
  \centerline{
\hfill
{\includegraphics[width=0.44\textwidth]{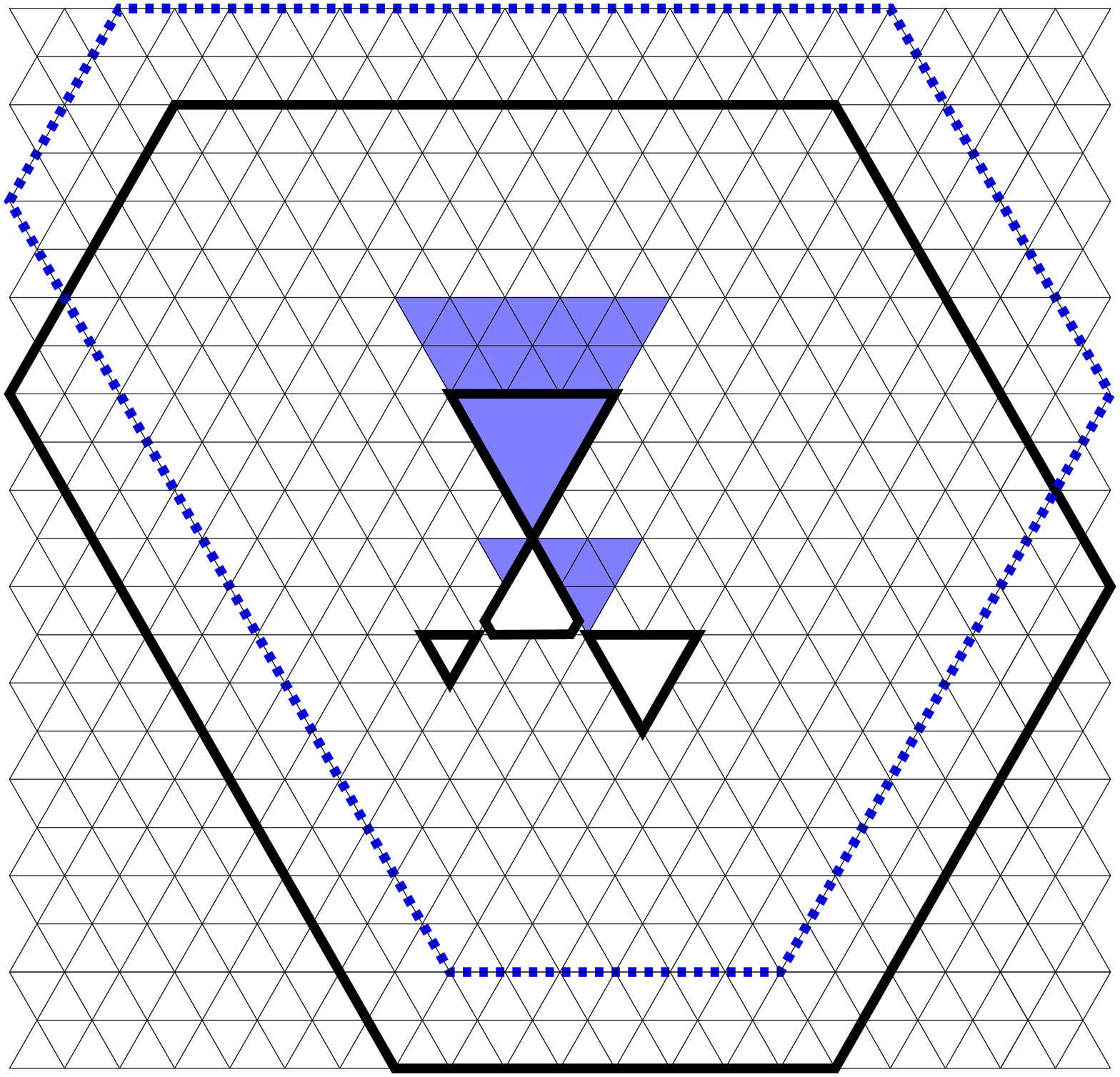}}
\hfill
}
  \caption{\label{fbd} Illustration of the bowtie squeezing that turns the left region into the right region in Figure \ref{fbc}.}
\end{figure}

\begin{figure}[h]
  \centerline{
\hfill
{\includegraphics[width=0.51\textwidth]{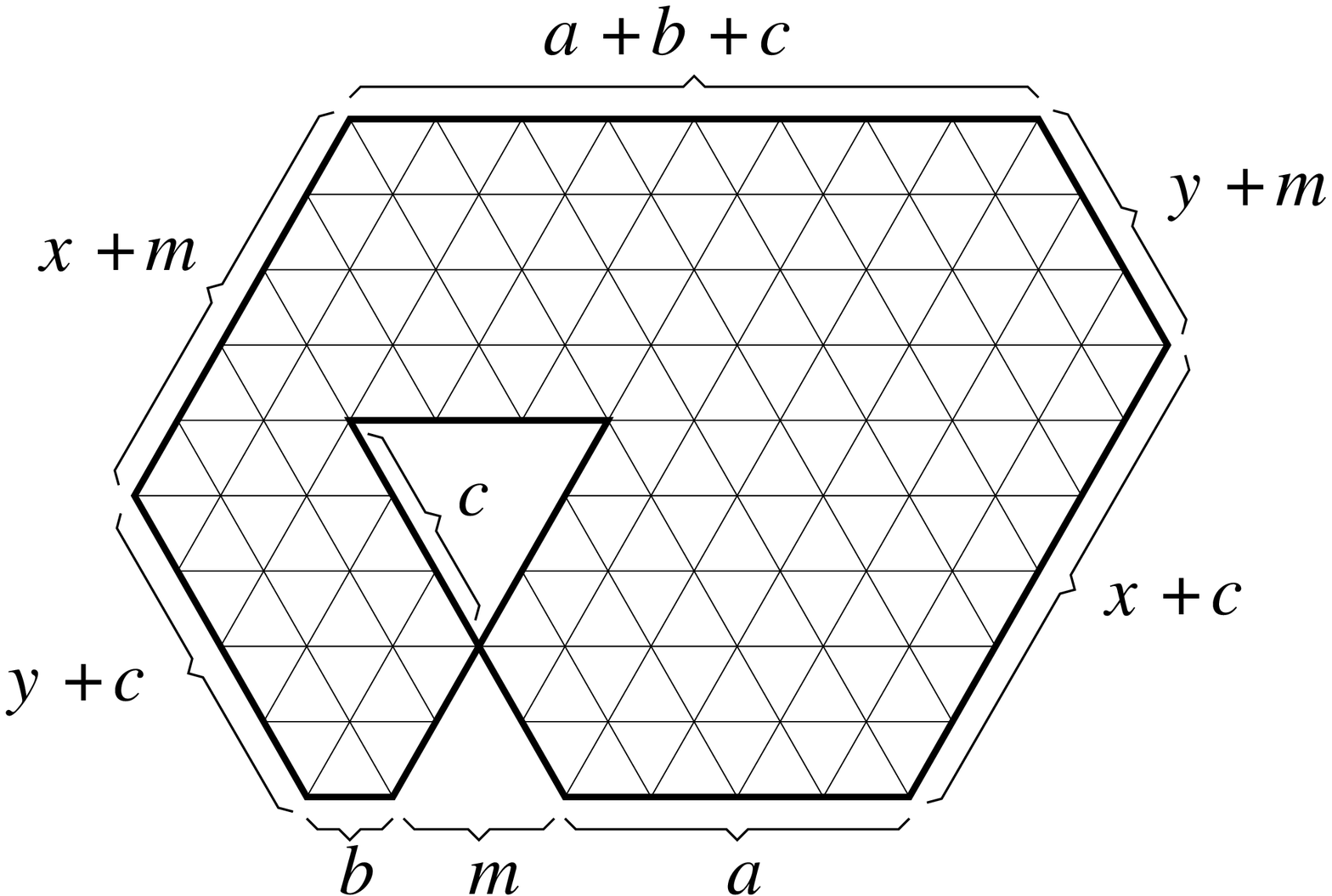}}
\hfill
{\includegraphics[width=0.42\textwidth]{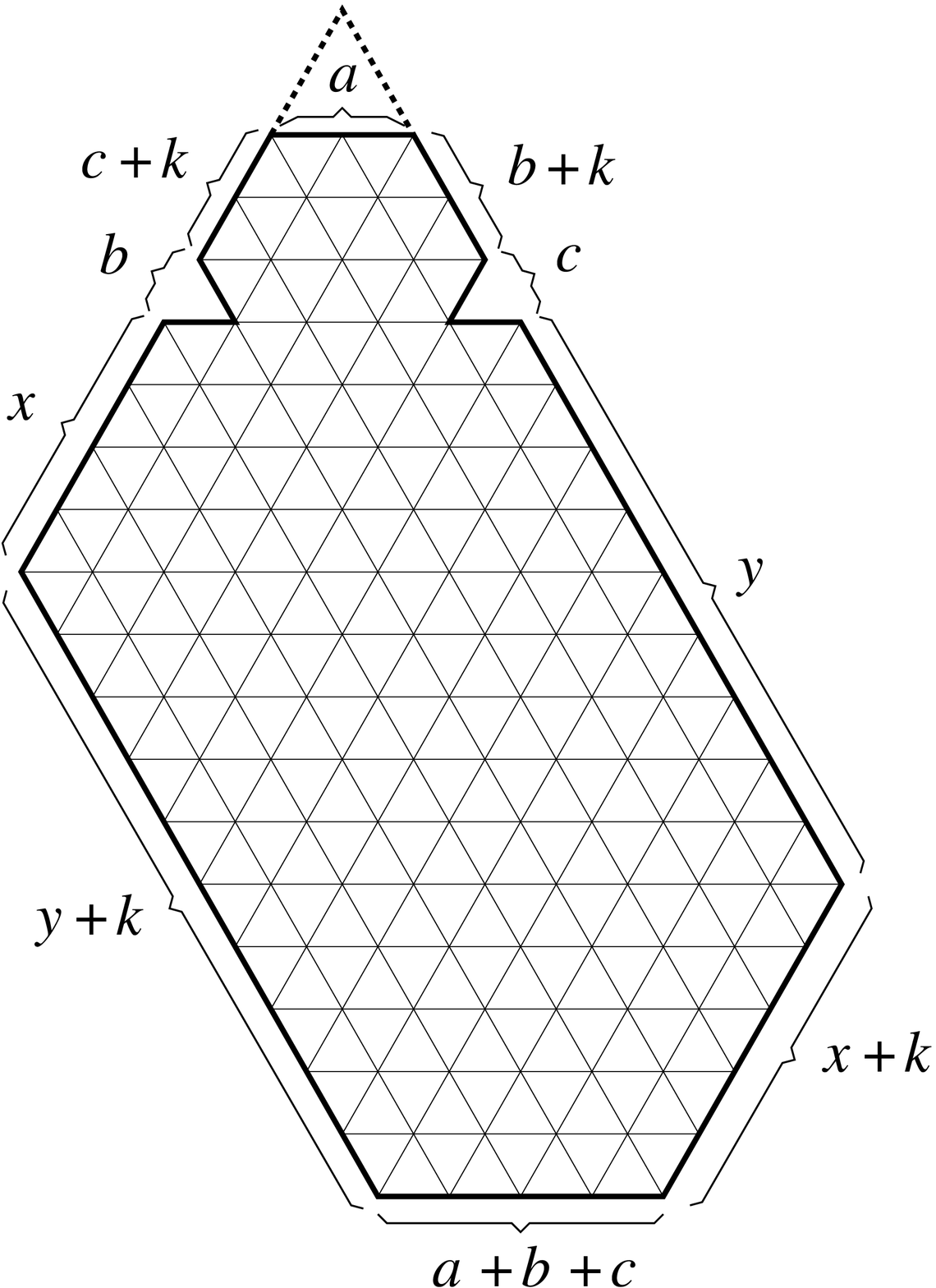}}
\hfill
}
  \caption{\label{fea} The magnet bar region $I_{x,y}(a,b,c,m)$ for $a=4$, $b=1$, $c=3$, $m=2$, $x=3$ and $y=1$ (left) and the snowman region $S_{x,y}(a,b,c,k)$ for $a=2$, $b=1$, $c=1$, $k=1$, $x=4$ and $y=9$ (right).}
\end{figure}

\begin{theo} \cite[Theorem 3.1]{ff}
\label{tea}
For non-negative integers $x,y,a,b,c,m$, the number of lozenge tilings of the region $I=I_{x,y}(a,b,c,m)$ is given by
\begin{equation}
\M(I_{x,y}(a,b,c,m))  
=
{\w^{(I)}\dfrac{\ka_A^{(I)}\ka_B^{(I)}\ka_C^{(I)}}{\ka_{BC}^{(I)}\ka_{AC}^{(I)}\ka_{AB}^{(I)}}}
\,P(x,y,a+b+c+m),
\label{eea}
\end{equation}
where the weight $\w$ and the couples $\ka$ are given by \eqref{ebb}--\eqref{ebdc} $($\!with $I$ viewed as a triad hexagon with the focal points $A$, $B$ and $C$ being the top, left and right vertices of the triangular dent of side $m$ along the base, and bowties of types $(c,m)$, $(0,0)$ and $(0,0)$, respectively$)$, and $P$ is given by \eqref{eaa}.

\end{theo}

The second family consists of the {\it snowman regions} $S_{x,y}(a,b,c,k)$ described on the right in Figure \ref{fea}. The region itself is determined by the thick solid line contour. The thick dotted lines on top indicate how $S_{x,y}(a,b,c,k)$ can be viewed as a triad hexagon: The focal points $A$, $B$, $C$ are the top of the triagle of side $a$, the left vertex of the triangular dent of side $b$, and the right vertex of the triangular dent of side $c$, and the corresponding bowties are of type $(0,a)$, $(0,b)$ and $(0,c)$, respectively\footnote{ The top side of this triad hexagon has length zero.}.

The following is a special case of Theorem 2.1 of \cite{2i}. The case $x=y$, $b=c$ is an earlier result of Rohatgi (see \cite{Rohatgi2dent}). Again, the form of the formula is new, adapted for our use of it in the proof of Theorem \ref{tba}.


\begin{theo}
\label{teb}
For non-negative integers $x,y,a,b,c,k$, the number of lozenge tilings of the region $S=S_{x,y}(a,b,c,k)$ is given by
\begin{equation}
\M(S_{x,y}(a,b,c,k))  
=
{\w^{(S)}\dfrac{\ka_A^{(S)}\ka_B^{(S)}\ka_C^{(S)}}{\ka_{BC}^{(S)}\ka_{AC}^{(S)}\ka_{AB}^{(S)}}}
\,P'(x+b+k,y+c+k,k),
\label{eeb}
\end{equation}
where the weight $\w$ and the couples $\ka$ are given by \eqref{ebb}--\eqref{ebdc} $($\!with $S$ viewed as a triad hexagon with the focal points $A$, $B$ and $C$ being the top of the triagle of side $a$, the left vertex of the triangular dent of side $b$, and the right vertex of the triangular dent of side $c$, and bowties of type $(0,a)$, $(0,b)$ and $(0,c)$, respectively$)$, and $P'$ is given by
\begin{equation}
P'(x,y,z)=\frac{\h(x)\h(y)\h(z)\h(x+y-z)}{\h(x+y)\h(y-x)\h(z-x)}.  
\label{eec}
\end{equation}
%

\end{theo}

\begin{figure}[h]
  \centerline{
\hfill
{\includegraphics[width=0.44\textwidth]{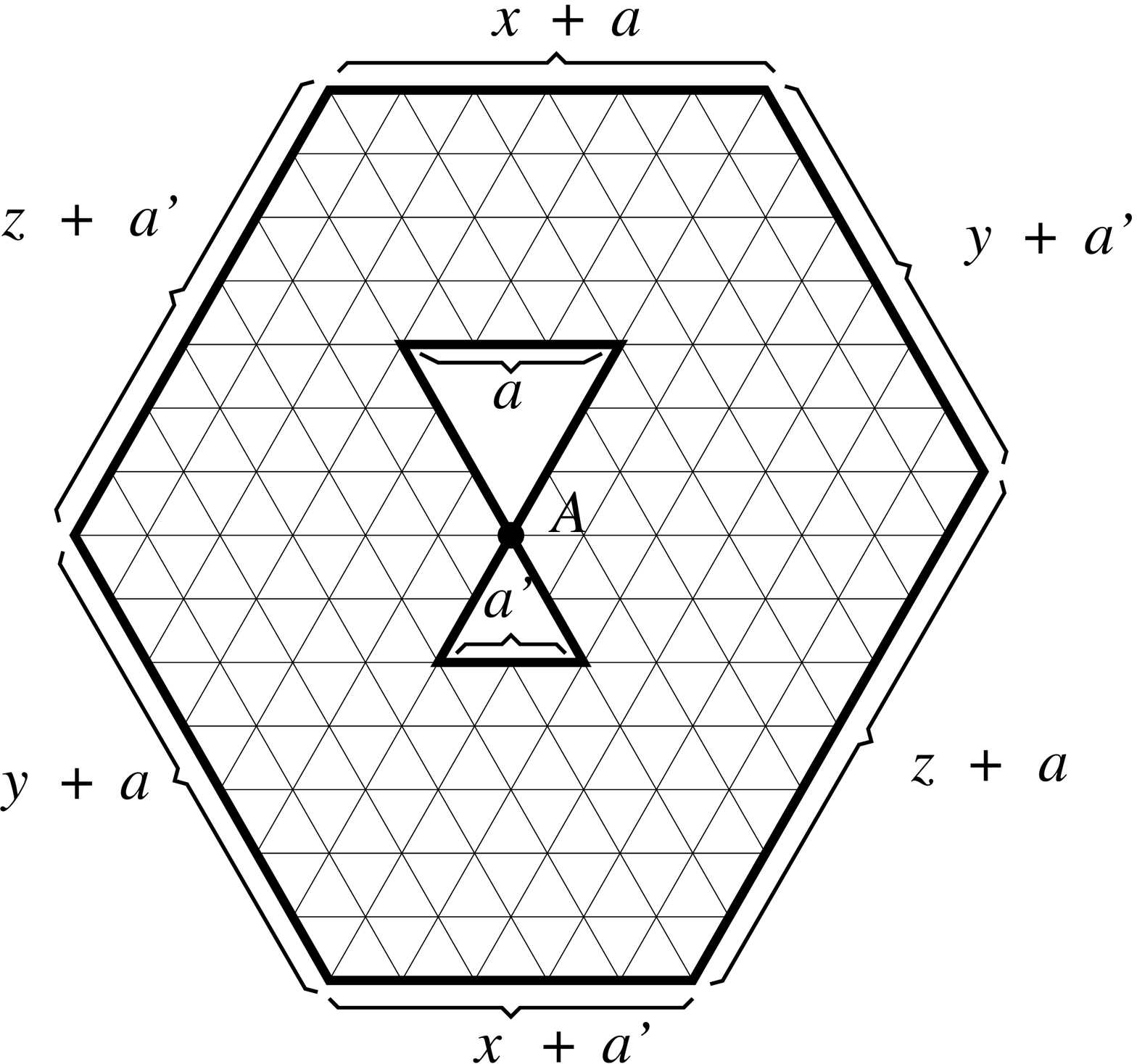}}
\hfill
{\includegraphics[width=0.38\textwidth]{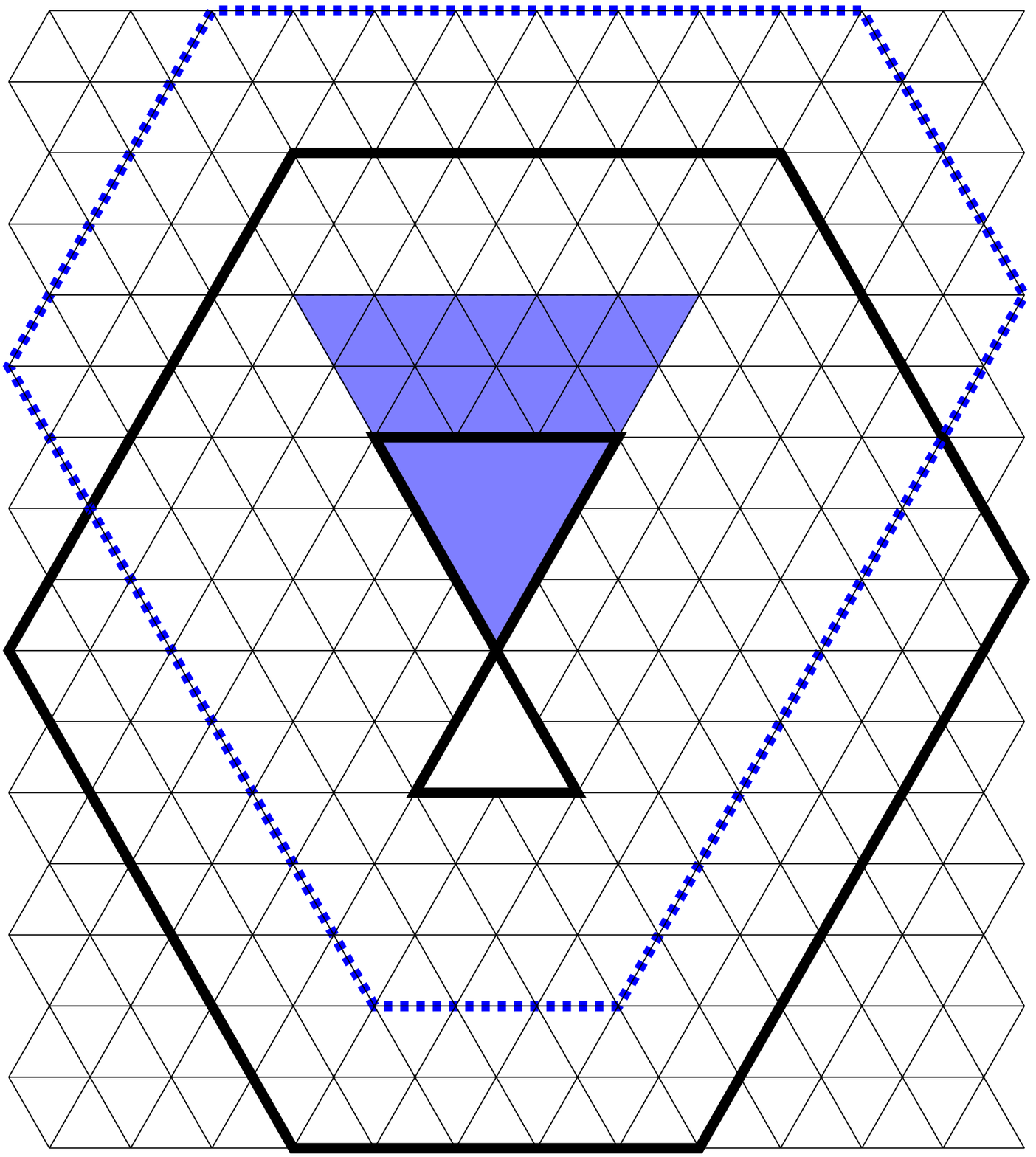}}
\hfill
}
  \caption{\label{fha} The hourglass region $G=G_{x,y,z}^A(a,a')$ for $x=3$, $y=4$, $z=5$, $a=3$ and $a'=2$ (left). The region $G$ (shown in thick solid lines) and the corresponding region $\bar{G}$ (shown in thick dotted lines and shaded bowtie) with the bowtie completely squeezed out (right).}
\end{figure}

\section{Hourglass regions}

Our proofs are based on Kuo's graphical condensation method (see \cite{Kuo}). For ease of reference, we state below the particular instance of Kuo's general results that we need for our proofs.

\begin{theo} \cite[Theorem 2.1]{Kuo}
\label{tca}
Let $G=(V_1,V_2,E)$ be a plane bipartite graph in which $|V_1|=|V_2|$. Let vertices $\alpha$, $\beta$, $\gamma$ and $\delta$ appear cyclically on a face of $G$. If $\alpha,\gamma\in V_1$ and $\beta,\delta\in V_2$, then
\begin{equation}
  \label{eca}
\M(G)\M(G-\{\alpha,\beta,\gamma,\delta\})=\M(G-\{\alpha,\beta\})\M(G-\{\gamma,\delta\})+\M(G-\{\alpha,\delta\})\M(G-\{\beta,\gamma\}).
\end{equation}
\end{theo}


In this section we prove the special case of Theorem \ref{tba} in which two of the removed bowties are empty. The resulting region is described in the picture on the left in Figure \ref{fha}. We call it an {\it hourglass region,} and we denote it by $G_{x,y,z}^A(a,a')$ (as for triad hexagons, $A$ denotes the bowtie node).

\begin{prop}
\label{tha}
Let $G=G_{x,y,z}^A(a,a')$ be an hourglass region, and let $\thickbar{G}$ be the region obtained from $G$ by completely squeezing out the $a'$-lobe. Then
\begin{equation}
\frac{\M(G)}{\M(\thickbar{G})}=
\frac{\h(a)\h(a')}{\h(a+a')}\frac{\ka_A\ka_B\ka_C}{\ka_{BC}\ka_{AC}\ka_{AB}},
\label{eha}
\end{equation}
where the couples $\ka$ are given by \eqref{ebb}--\eqref{ebdc}, with $G$ viewed as a triad hexagon with the focal points $A$, $B$ and $C$ being the top, left and right vertices of the bottom bowtie lobe, and bowties of types $(a,a')$, $(0,0)$ and $(0,0)$, respectively.

\end{prop}

\begin{proof} By Lemma \ref{tileability}, all hourglass regions are tileable. In particular, $\M(\thickbar{G})\neq0$, and the ratio on the left hand side of \eqref{eha} is well defined.

We prove the statement by induction, using Kuo's graphical condensation identity \eqref{eca} at the induction step. Consider the special case of Figure \ref{fca} when the bottom two bowties are empty --- this corresponds to the hourglass regions under consideration in this section.

Choosing $\alpha$, $\beta$, $\gamma$ and $\delta$ as shown in the top right picture in Figure \ref{fca}, and assuming the forced lozenges come in the pattern shown in Figure \ref{fca}, we obtain
\begin{align}
&
\M(G_{x,y,z}^{A}(a,a'))\M(G_{x,y-1,z-1}^{A}(a,a'))
=
\M(G_{x,y-1,z}^{A}(a,a'))\M(G_{x,y,z-1}^{A}(a,a'))
\nonumber
\\[5pt]
&\ \ \ \ \ \ \ \ \ \ \ \ \ \ \ \ \ \ \ \ \ \ \ \ \ \ \ \ \ 
+\M(G_{x-1,y,z}^{A}(a,a'))\M(G_{x+1,y-1,z-1}^{A}(a,a')).
\label{ehb}
\end{align}  

\begin{figure}[h]
  \centerline{
\hfill
{\includegraphics[width=0.28\textwidth]{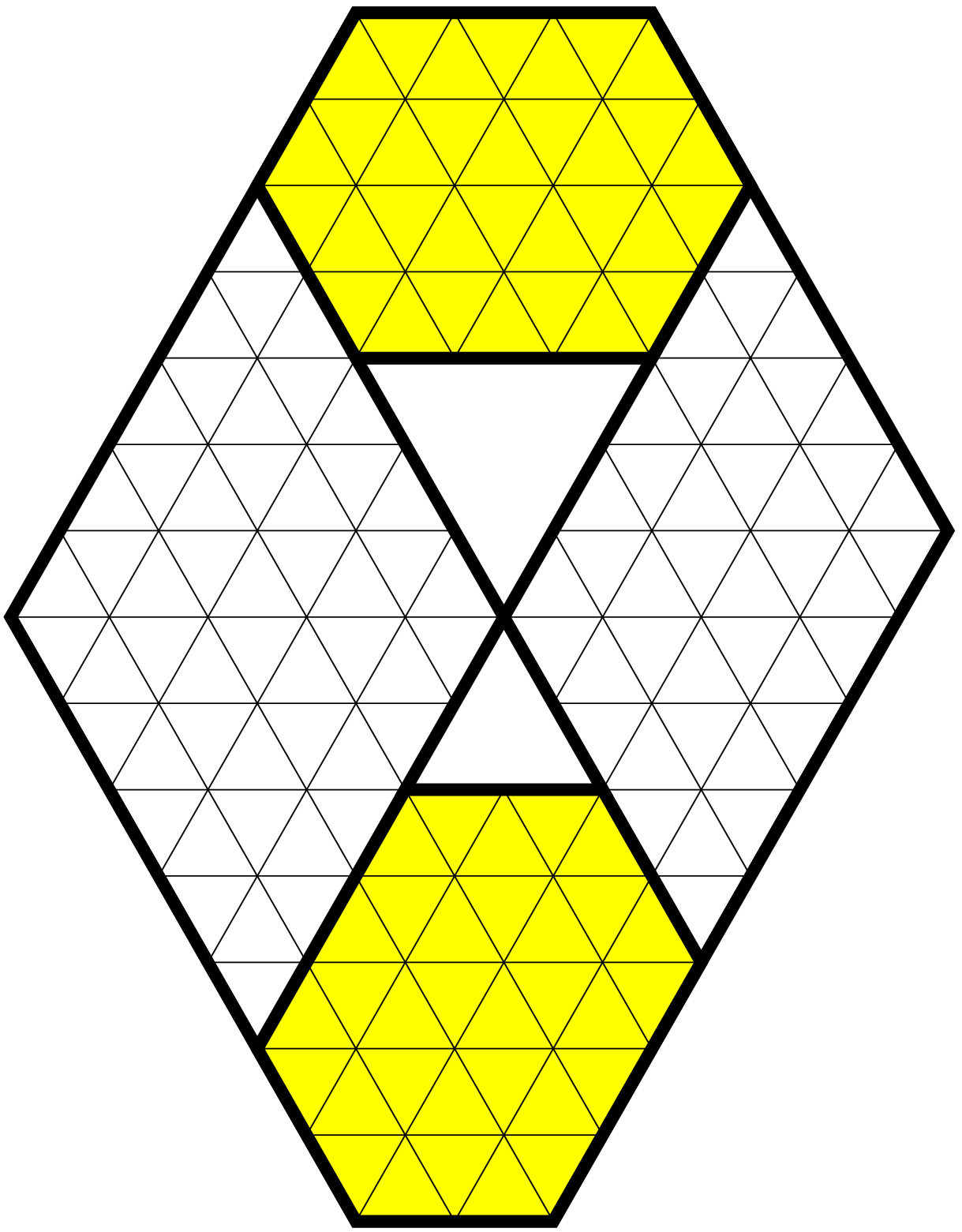}}
\hfill
{\includegraphics[width=0.30\textwidth]{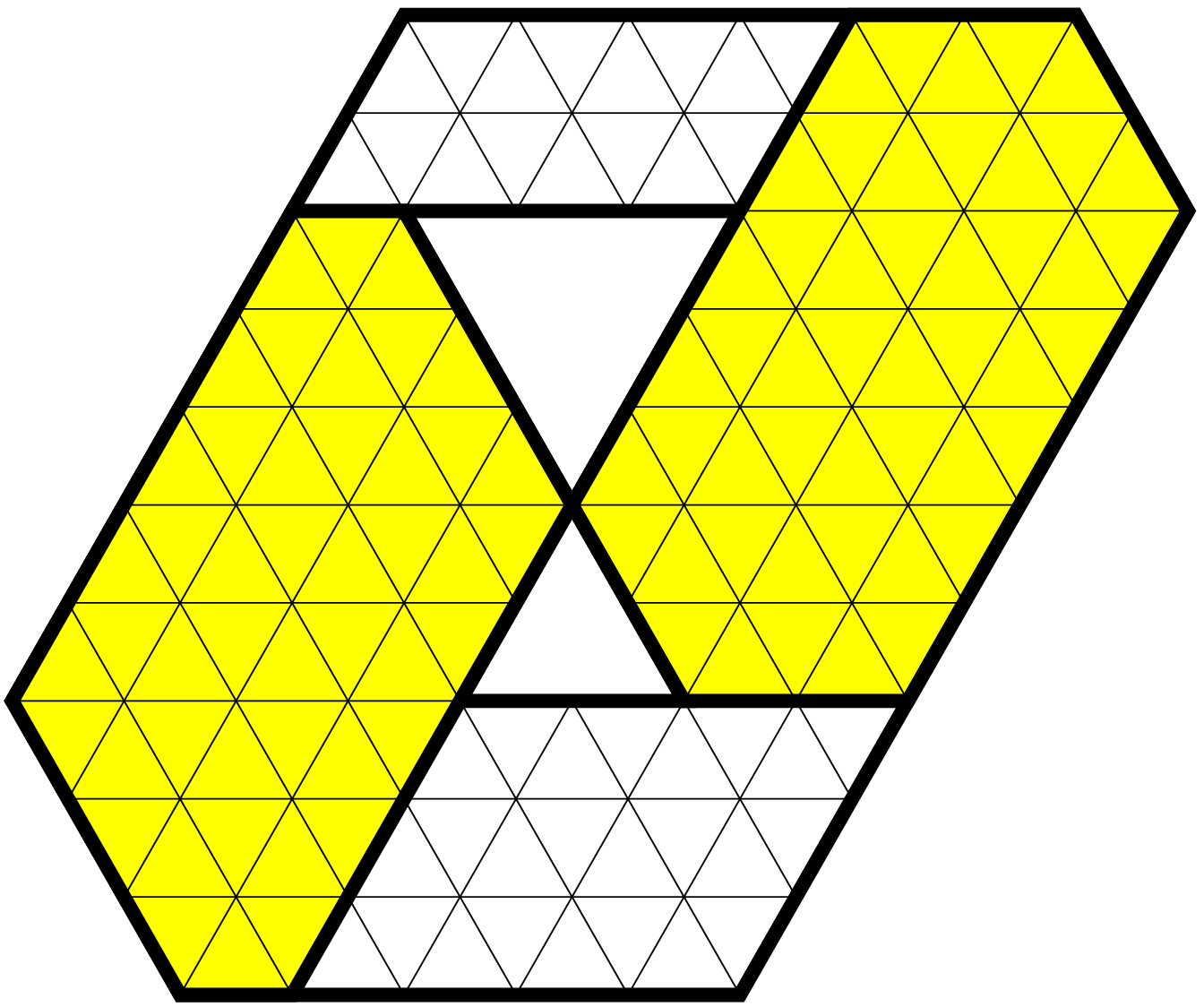}}
\hfill
}
  \caption{\label{fhb} The hourglass region $G$ when $x=0$; here $y=4$, $z=5$, $a=3$ and $a'=2$ (left). The region $G$ when $y=0$; here $x=3$, $z=5$, $a=3$ and $a'=2$ (right).}
\end{figure}

\begin{figure}[h]
  \centerline{
\hfill
{\includegraphics[width=0.36\textwidth]{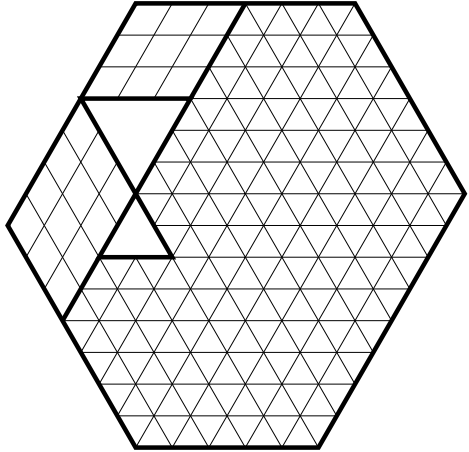}}
\hfill
}
  \caption{\label{fhc} The hourglass region $G$ when the bowtie touches the northwestern side of the hexagon; here $x=3$, $y=4$, $z=5$, $a=3$ and $a'=2$.}
\end{figure}

\parindent0pt
In order for all the regions involved in \eqref{ehb} to be defined, we need to have $x,y,z\geq1$. In order for the forced lozenges to be indeed as shown in Figure \ref{fca}, the bowtie must touch none of the top, bottom, northwestern or southwestern sides of the outer hexagon. Furthermore, since we will use \eqref{ehb} as a recurrence relation expressing $\M(G_{x,y,z}^{A}(a,a'))$ in terms of the other five tiling counts involved in \eqref{ehb}, we need to make sure that its coefficient in \eqref{ehb}, $\M(G_{x,y-1,z-1}^{A}(a,a'))$, is non-zero. Lemma \ref{tileability} implies that any hourglass region is tileable, so the coefficient of  $\M(G_{x,y,z}^{A}(a,a'))$ in \eqref{ehb} is indeed non-zero.



\parindent15pt
Therefore, the following will be base cases of our induction: (1) $x$, $y$ or $z$ is zero; (2) the bowtie touches the top or the bottom side of the outer hexagon; and (3) the bowtie touches the northwestern or the southwestern side of the hexagon. Then \eqref{ehb} can be used to prove the statement by induction on $x+y+z$, as for the each of the five $G$-regions involved in it, the sum of the $x$-, $y$- and $z$-parameters is strictly less than $x+y+z$.

If $x=0$, the region $G$ looks as shown on the left in Figure \ref{fhb}. Due to the fact that the top side of the top bowtie lobe has the same length as the top side of the hexagon, the upper shaded hexagon must be internally tiled in any tiling of $G$. For the same reason, also the bottom shaded hexagon is internally tiled. The leftover portion of $G$ consists of two uniquely tileable parallelograms. It follows that in this case
\begin{equation}
\M(G)=\M(G_{0,y,z}^A(a,a'))=P(a,\de(A,NW)-a,\de(A,NE)-a)\,P(a',\de(A,SW)-a',\de(A,SE)-a'),
\label{ehc}
\end{equation}
where $P$ is given by \eqref{eaa}.
Since $\thickbar{G}=G_{0,y,z}^{A}(a+a',0)$, it follows from \eqref{ehc} (using also the picture on the right in Figure \ref{fha} to relate the distances from $A$ to the outer sides in $G$ and $\thickbar{G}$) that
\begin{equation}
\M(\thickbar{G})=P(a+a',\de(A,NW)-a+a',\de(A,NE)-a+a').
\label{ehd}
\end{equation}
The statement follows by combining equations \eqref{ehc} and \eqref{ehd}. The base cases $y=0$ and $z=0$ follow by a similar argument.

It is apparent that if the bowtie touches either the top or the bottom side of the hexagon, the hourglass region becomes a magnet bar region (see the previous section). Therefore this base case follows from Theorem \ref{tea}.


If the bowtie touches the northwestern side of the hexagon, the situation is as pictured in Figure \ref{fhc}. After removing the forced lozenges, the leftover region is a special case of a snowman region. This base case follows therefore by Theorem \ref{teb}. The case when the bowtie touches the sothwestern side follows by symmetry.

For the induction step, let $x,y,z>0$ and assume that equation \eqref{eha} holds for any tileable hourglass region for which the sum of the $x$-, $y$- and $z$-parameters is strictly less than $x+y+z$. Let $G_{x,y,z}^A(a,a')$ be an hourglass region in which the bowtie does not touch any of the top, bottom, northwestern or southwestern sides. We need to show that \eqref{eha} holds.

By our assumptions, all the regions in \eqref{ehb} are well defined, and $\M(G_{x,y-1,z-1}^{A}(a,a'))>0$. We can therefore use \eqref{ehb} to express $\M(G_{x,y,z}^{A}(a,a'))$ in terms of the number of tilings of the other five regions involved. By the induction hypothesis, we can use the equivalent restatement
\begin{equation}
\M(G)
=
\M(\thickbar{G})
\frac{\h(a)\h(a')}{\h(a+a')}\frac{\ka_A\ka_B\ka_C}{\ka_{BC}\ka_{AC}\ka_{AB}}
\label{ehe}
\end{equation}
of formula \eqref{eha} for each of the latter. To finish the induction step, we need to verify that the resulting expression for $\M(G_{x,y,z}^{A}(a,a'))$ is equal to the right hand side of \eqref{eha}. This
is a special case of the verification at
the induction step in the proof of Theorem \ref{tba}, which we present in great detail in Sections 8 and 9. \end{proof}

\section{Based hourglass regions}

In this section we consider the family of regions $F_{d,e,f,y,z}(a,a',b',c')$ described in Figure \ref{fia}; we call them {\it based hourglass regions}. The region $F_{d,e,f,y,z}(a,a',b',c')$ is defined for any non-negative integers $d$, $e$, $f$, $y$, $z$, $a$, $a'$, $b'$ and $c'$ satisfying $f\leq y+z$ (a condition equivalent to the statement that the top of the bowtie is weakly below the top of the outer hexagon).

Given a based hourglass region $F=F_{d,e,f,y,z}(a,a',b',c')$, its {\it companion cored hexagon} $F_0$ is the region described in Figure \ref{fib}. It is obtained from the region $\thickbar{F}$ by removing the portion under the line containing the bottom focal edge.

\begin{prop}
\label{basedhg}
Let $d$, $e$, $f$, $y$, $z$, $a$, $a'$, $b'$ and $c'$ be non-negative integers with $f\leq y+z$, and consider the based hourglass region $F=F_{d,e,f,y,z}(a,a',b',c')$ and its companion cored hexagon~$F_0$. Then
\begin{align}
\frac{\M(F)}{\M(F_0)}
=
\frac{\h(a)\h(a')\h(b')\h(c')}{\h(a+a')\h(b'+c')}
&
\frac{\h(f+a'+b'+c')^2\h(f+a+a')}{\h(f+a+a'+b'+c')\h(f+a'+b')\h(f+a'+c')}
\nonumber
\\[10pt]
\times
&
\dfrac
{\dfrac{\ka_A^{(F)}\ka_B^{(F)}\ka_C^{(F)}}{\ka_{BC}^{(F)}\ka_{AC}^{(F)}\ka_{AB}^{(F)}}}
{\dfrac{\ka_{A_0}^{(\thickbar{F})}\ka_{B_0}^{(\thickbar{F})}\ka_{C_0}^{(\thickbar{F})}}
{\ka_{B_0C_0}^{(\thickbar{F})}\ka_{A_0C_0}^{(\thickbar{F})}\ka_{A_0B_0}^{(\thickbar{F})}}},
\label{eia}
\end{align}

\begin{figure}[h]
  \centerline{
\hfill
{\includegraphics[width=0.60\textwidth]{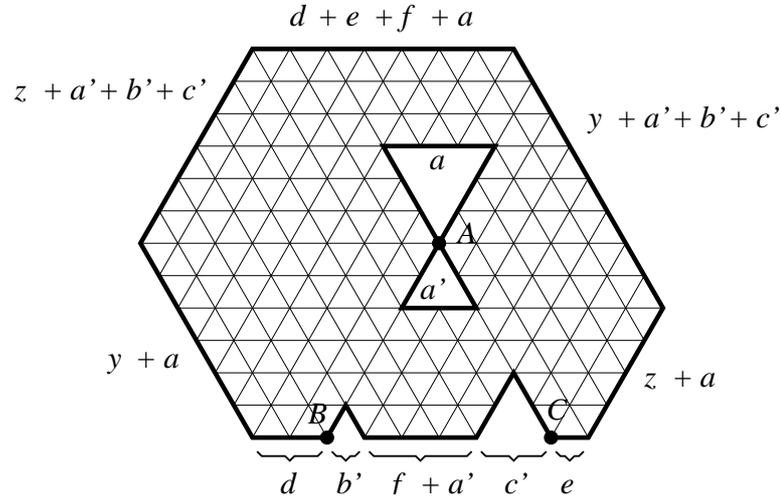}}
\hfill
}
  \caption{\label{fia} The based hourglass region $F=F_{d,e,f,y,z}(a,a',b',c')$ for $d=2$, $e=1$, $f=1$, $y=3$, $z=1$, $a=3$, $a'=2$, $b'=1$ and $c'=2$.}
\vskip-0.10in
\end{figure}

\begin{figure}[h]
  \centerline{
\hfill
{\includegraphics[width=0.60\textwidth]{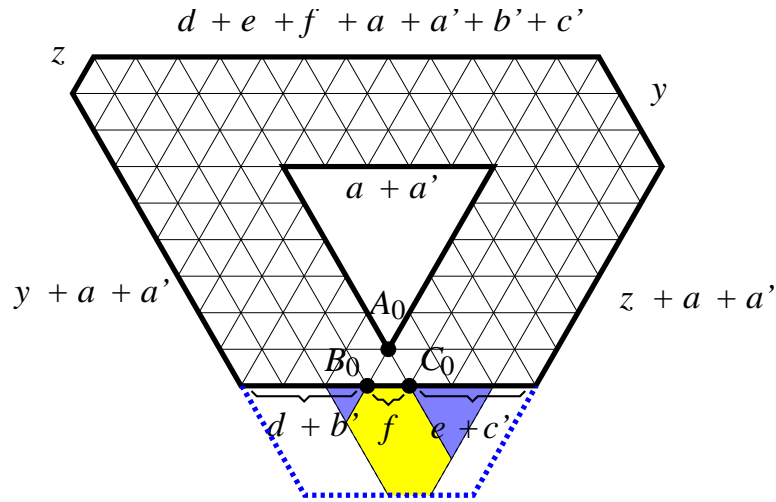}}
\hfill
}
  \caption{\label{fib} The companion cored hexagon $F_0$ (bounded by the thick solid lines) of the based hourglass region $F$ in Figure \ref{fia}. It is obtained from the region $\bar{F}$ by removing the bottom two holes, an internally tiled hexagon and two parallelograms along which the tiling is forced.}
\end{figure}

\parindent0pt
where the couples\footnote{ Recall that the couples $\ka$ are defined by equations (2.3)--(2.8).} $\ka$ in the numerator fraction refer to the region $F$ with focal points $A$, $B$, $C$ as indicated in Figure $\ref{fia}$, and the couples $\ka$ in the denominator fraction refer to the region $\thickbar{F}$ with focal points $A_0$, $B_0$, $C_0$ as indicated in Figure $\ref{fib}$. Explicitly, we have
\begin{align}
\ka_A^{(F)}&=H(f+a'+b'+c')H(y+z+a-f)
\label{eiaa}
\\
\ka_B^{(F)}&=H(d)H(z+e+f+a+a'+b'+c')
\label{eiab}
\\
\ka_C^{(F)}&=H(e)H(y+d+f+a+a'+b'+c')
\label{eiac}
\\
%
%
\ka_{BC}^{(F)}&=H(0)H(y+z+a+a'+b'+c')
\label{eiad}
\\
\ka_{AC}^{(F)}&=H(d+f+a'+b'+c')H(z+e+a)
\label{eiae}
\\
\ka_{AB}^{(F)}&=H(e+f+a'+b'+c')H(y+d+a)
\label{eiaf}
\end{align}
and
\begin{align}
\ka_{A_0}^{(\thickbar{F})}&=H(f+b'+c')H(y+z+a+a'-f)
\label{eiag}
\\
\ka_{B_0}^{(\thickbar{F})}&=H(d+b')H(z+e+f+a+a'+c')
\label{eiah}
\\
\ka_{C_0}^{(\thickbar{F})}&=H(e+c')H(y+d+f+a+a'+b')
\label{eiai}
%
%
%
\\
\ka_{B_0C_0}^{(\thickbar{F})}&=H(b'+c')H(y+z+a+a')
\label{eiaj}
\\
\ka_{A_0C_0}^{(\thickbar{F})}&=H(d+f+b')H(z+e+a+a'+c')
\label{eiak}
\\
\ka_{A_0B_0}^{(\thickbar{F})}&=H(e+f+c')H(y+d+a+a'+b').
\label{eial}
\end{align}

\end{prop}

\parindent0pt
\textsc{Remark 3.}
Note that in the region $\thickbar{F}$ the hexagon indicated by a shading in Figure \ref{fib} is internally tiled, and that each tiling is forced along the parallelograms to its left and right. Using this and formula \eqref{eaa}, one readily sees that equation \eqref{eia} can be stated equivalently as
\begin{align}
\frac{\M(F)}{\M(\thickbar{F})}
=
\frac{\w^{(F)}}{\w^{(\thickbar{F})}}
\dfrac
{\dfrac{\ka_A^{(F)}\ka_B^{(F)}\ka_C^{(F)}}{\ka_{BC}^{(F)}\ka_{AC}^{(F)}\ka_{AB}^{(F)}}}
{\dfrac{\ka_{A_0}^{(\thickbar{F})}\ka_{B_0}^{(\thickbar{F})}\ka_{C_0}^{(\thickbar{F})}}
{\ka_{B_0C_0}^{(\thickbar{F})}\ka_{A_0C_0}^{(\thickbar{F})}\ka_{A_0B_0}^{(\thickbar{F})}}},
\label{eib}
\end{align}
where the weights $\w$ are given by \eqref{ebb}.

\medskip
\parindent15pt
{\it Proof of Proposition \ref{basedhg}.} By Lemma \ref{tileability}, all based hourglass regions are tileable. In particular, $\M(F_0)\neq0$, so the ratio on the left hand side of \eqref{eia} is well defined.

We prove the statement by induction, using Kuo condensation at the induction step. The picture on the top left in Figure \ref{fic} shows the region $F_{d,e,f,y,z}(a,a',b',c')$. Choosing $\alpha$, $\beta$, $\gamma$ and $\delta$ as shown in the top right picture in Figure \ref{fic}, and assuming the forced lozenges come in the pattern shown in Figure \ref{fic}, we obtain
\begin{align}
&
\M(F_{d,e,f,y,z}(a,a',b',c'))\M(F_{d-1,e,f,y,z-1}(a,a',b',c'))
=
\nonumber
\\[5pt]
&\ \ \ \ \ \ \ \ \ \ \ \ \ \ \ \ \ \ \ \ \ \ \ \ \ \ \ \ \ \ \ \ 
\M(F_{d,e,f,y-1,z}(a,a',b',c'))\M(F_{d-1,e,f,y+1,z-1}(a,a',b',c'))
\nonumber
\\[5pt]
&\ \ \ \ \ \ \ \ \ \ \ \ \ \ \ \ \ \ \ \ \ \ \ \ \ \ \ \ \ 
+\M(F_{d-1,e,f,y,z}(a,a',b',c'))\M(F_{d,e,f,y,z-1}(a,a',b',c')).
\label{eic}
\end{align}  

\parindent0pt
In order for all the regions involved in \eqref{eic} to be defined, we need to have $d,y,z\geq1$ (note that these imply that the lengths of the southwestern, northwestern, northern and northeastern sides in $F_{d,e,f,y,z}(a,a',b',c')$ are all positive, and thus there is room on them to accomodate the unit triangles $\alpha$, $\beta$, $\gamma$ and $\delta$ indiacted in Figure \ref{fic}). In order for the forced lozenges to be indeed as shown in Figure \ref{fic}, the bowtie must touch none of the top, northwestern, northeastern or southwestern sides of $F_{d,e,f,y,z}(a,a',b',c')$. Furthermore, since we will use \eqref{eic} as a recurrence relation expressing $\M(F_{d,e,f,y,z}(a,a',b',c'))$ in terms of the other five tiling counts involved in \eqref{eic}, we need to make sure that its coefficient in \eqref{eic}, $\M(F_{d-1,e,f,y,z-1}(a,a',b',c'))$, is non-zero. Lemma \ref{tileability} implies that any based hourglass region is tileable, so for $d,z\geq1$ the coefficient of  $\M(F_{d,e,f,y,z}(a,a',b',c'))$ in \eqref{eic} is indeed non-zero.

\parindent15pt
If we take the mirror image of the region $F_{d,e,f,y,z}(a,a',b',c')$ across the vertical, we obtain the region $F_{e,d,f,z,y}(a,a',c',b')$. It follows from the discussion in the previous paragraph that if $e,y,z\geq1$ and the bowtie touches none of the sides above the base,
we can still use identity \eqref{eic} at the induction step.

Therefore, the base cases for our induction will be the following: (1) $y$ or $z$ is zero; (2) $d=e=0$; (3) the bowtie touches the top side; (4) the bowtie touches the northwestern or the northeastern side; and (5) the bowtie touches the southwestern or southeastern side.

\begin{figure}[h]
  \centerline{
\hfill
{\includegraphics[width=0.38\textwidth]{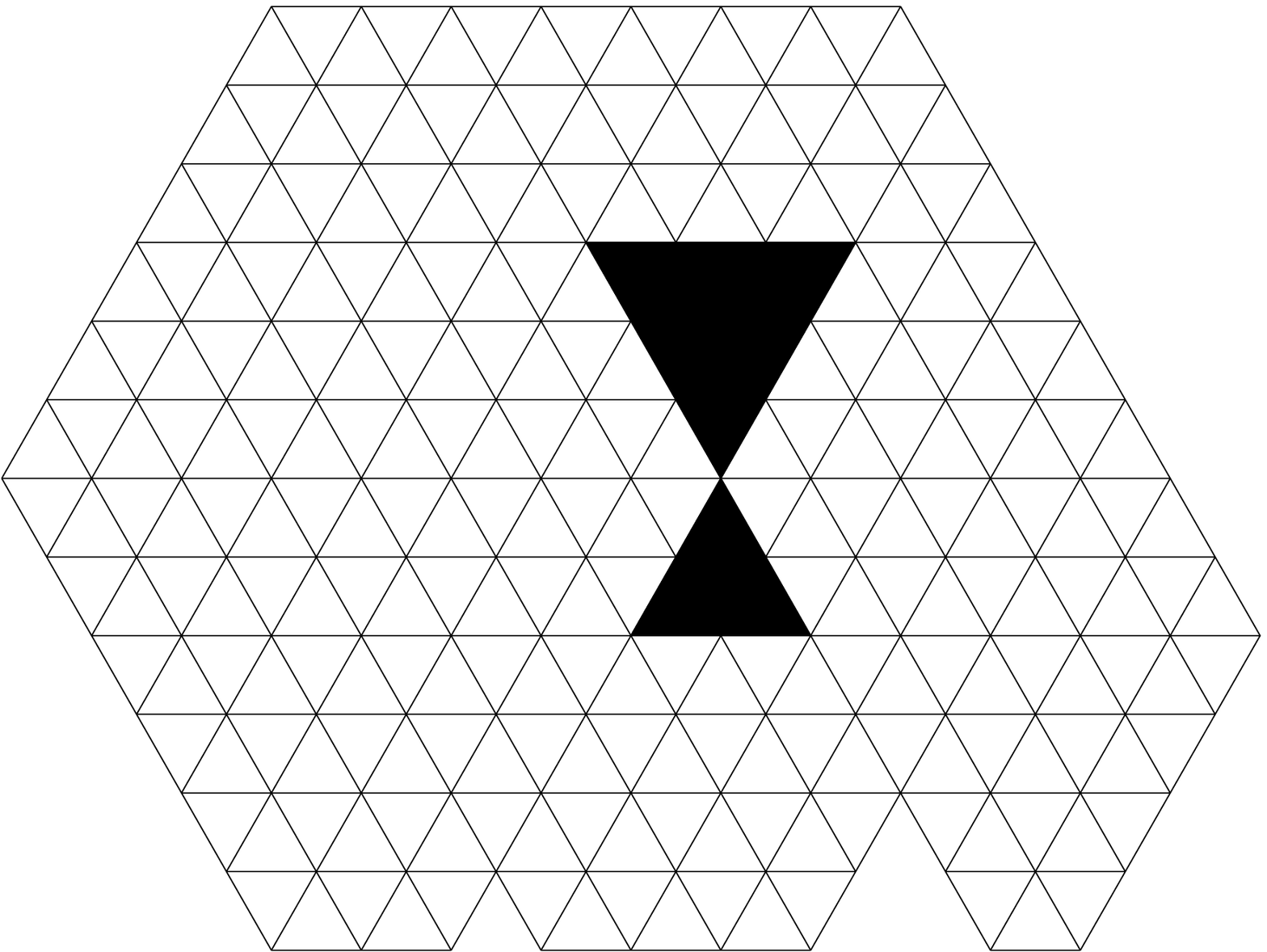}}
\hfill
{\includegraphics[width=0.38\textwidth]{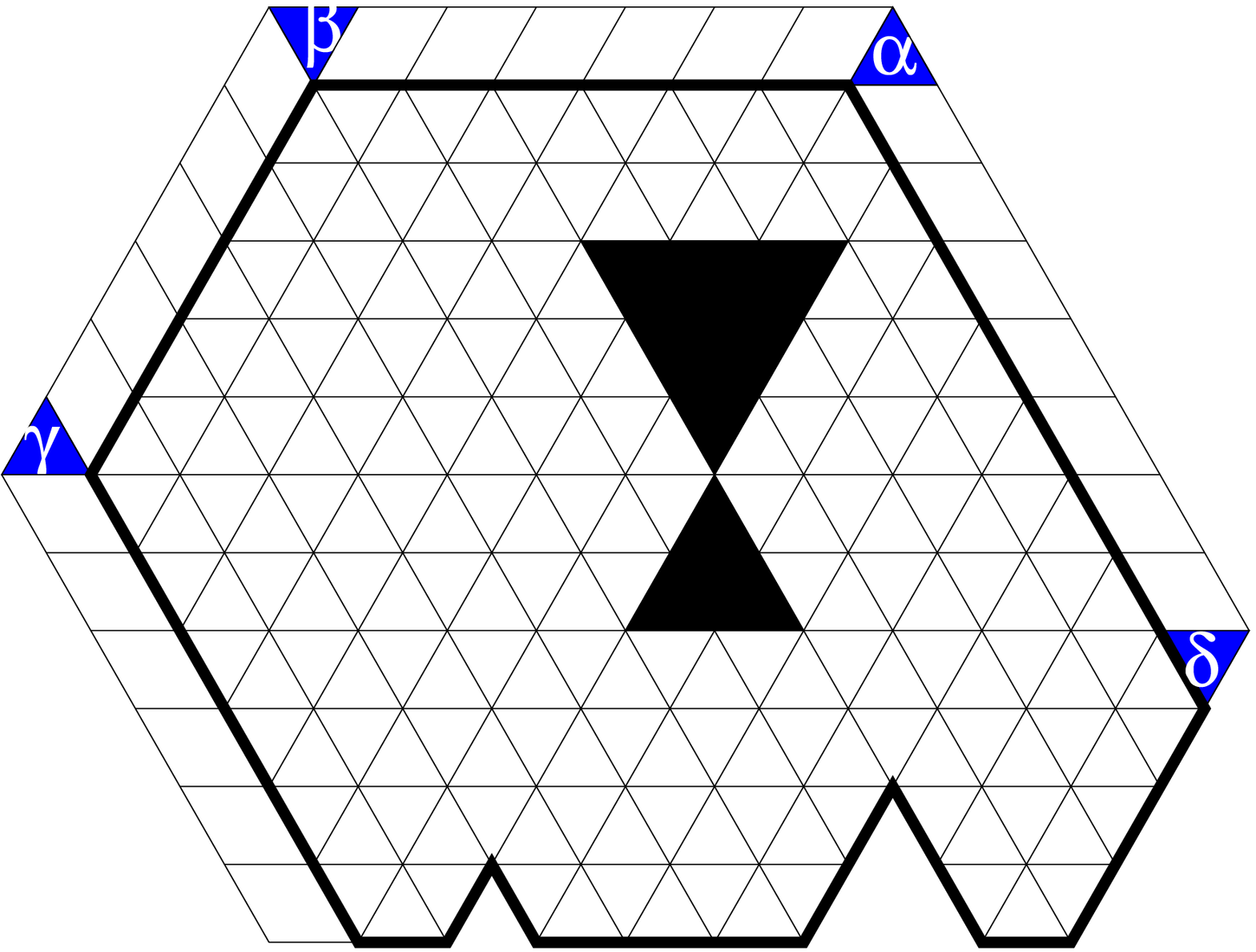}}
\hfill
}
\vskip0.2in
  \centerline{
\hfill
{\includegraphics[width=0.38\textwidth]{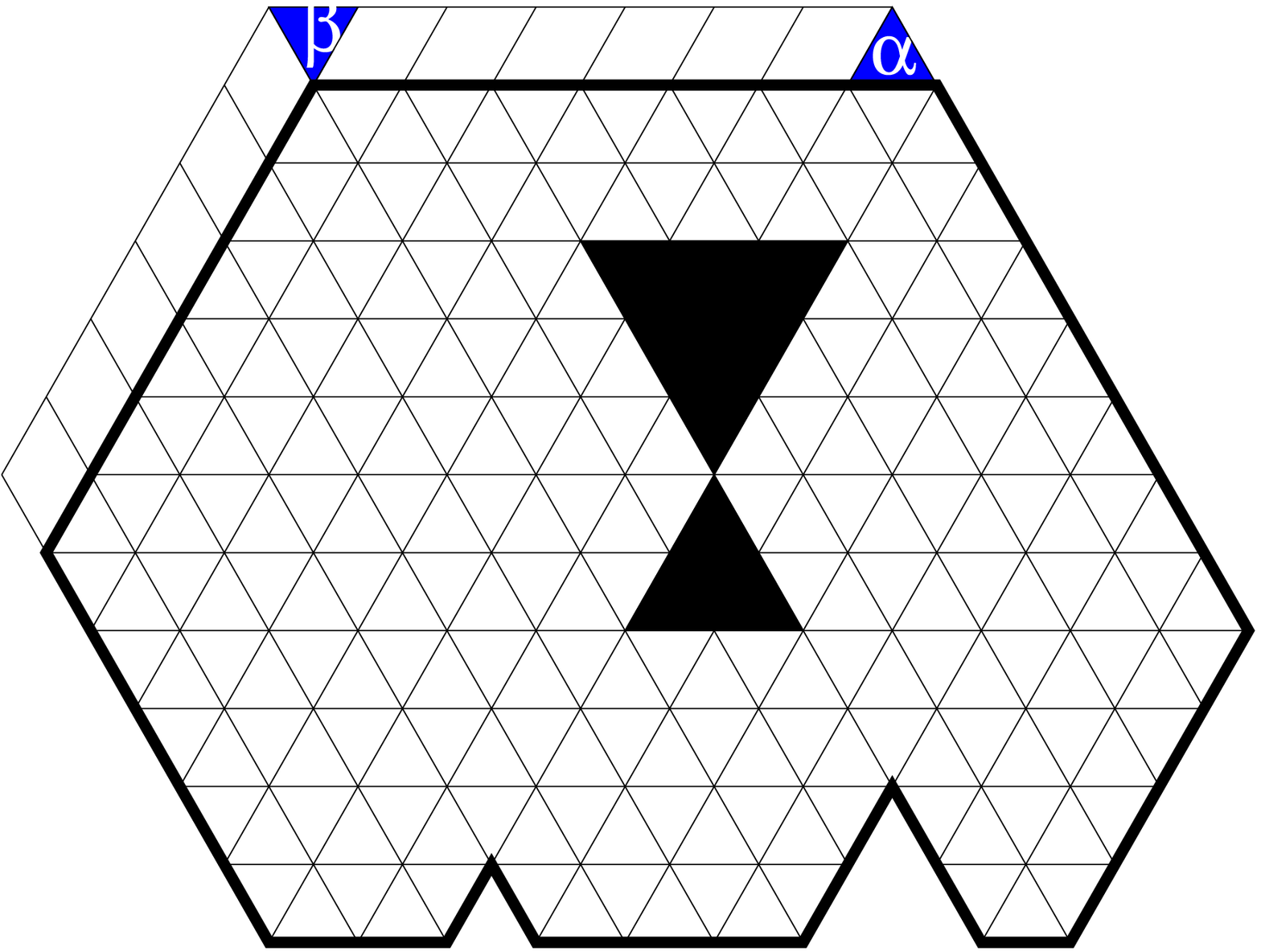}}
\hfill
{\includegraphics[width=0.38\textwidth]{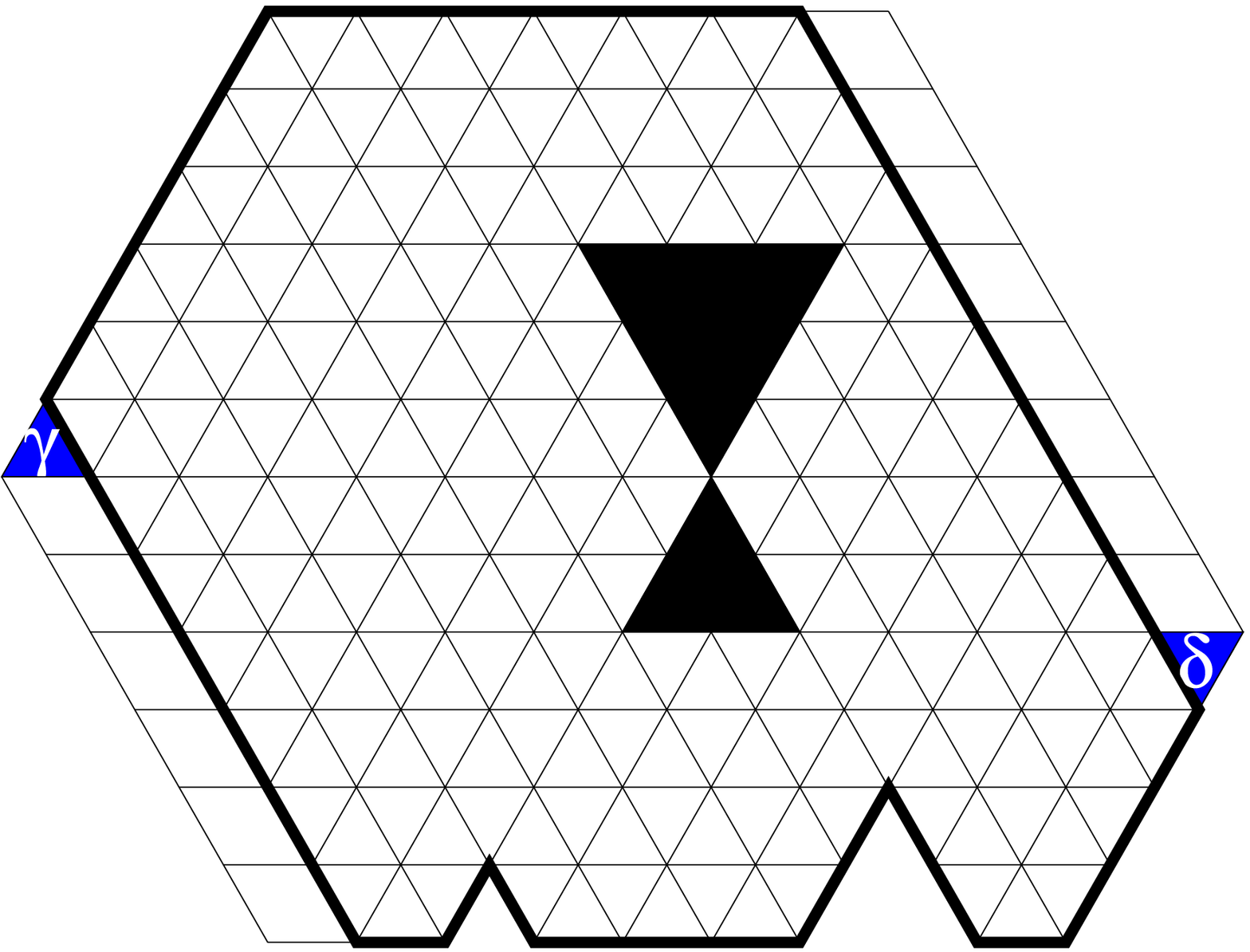}}
\hfill
}
\vskip0.2in
  \centerline{
\hfill
{\includegraphics[width=0.38\textwidth]{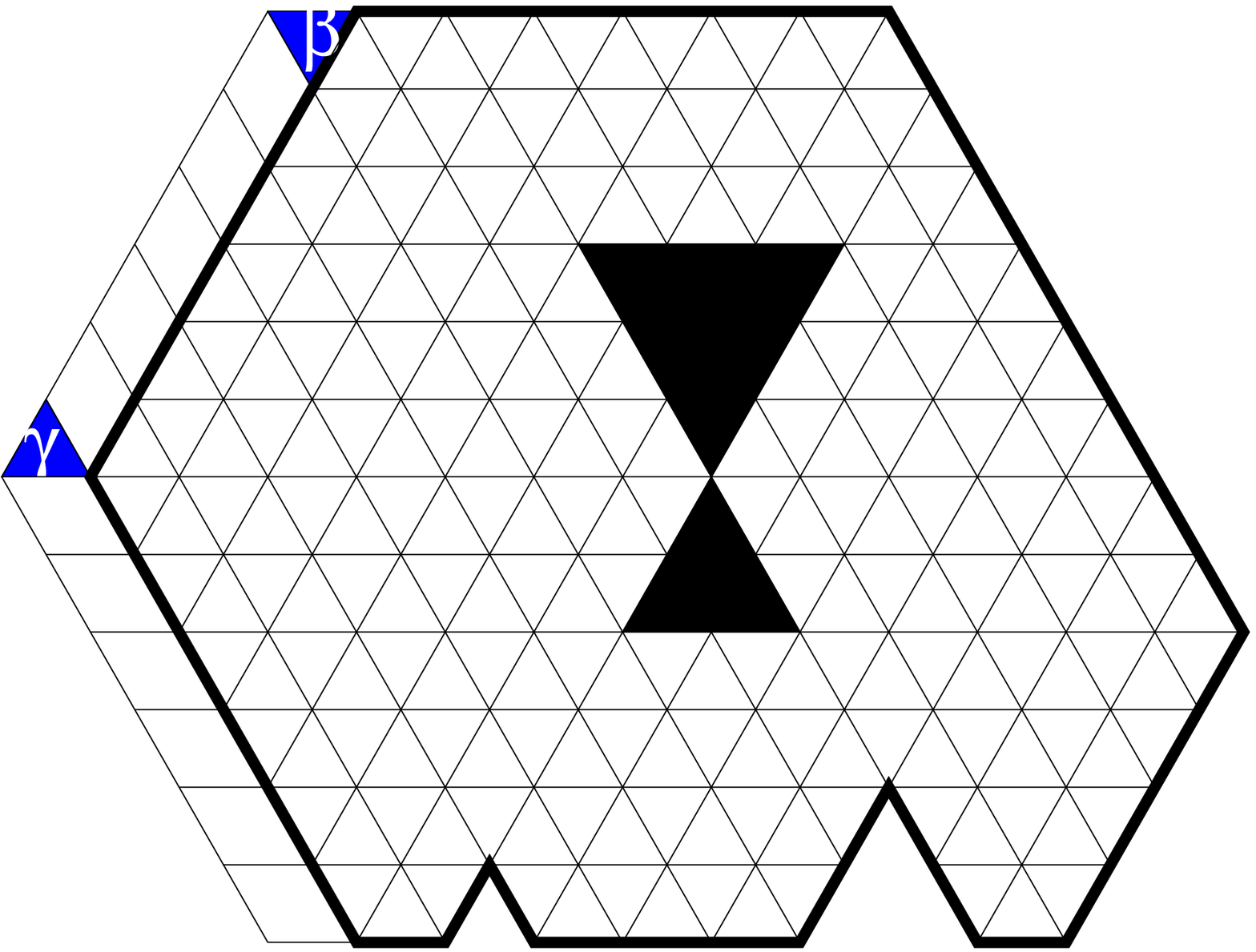}}
\hfill
{\includegraphics[width=0.38\textwidth]{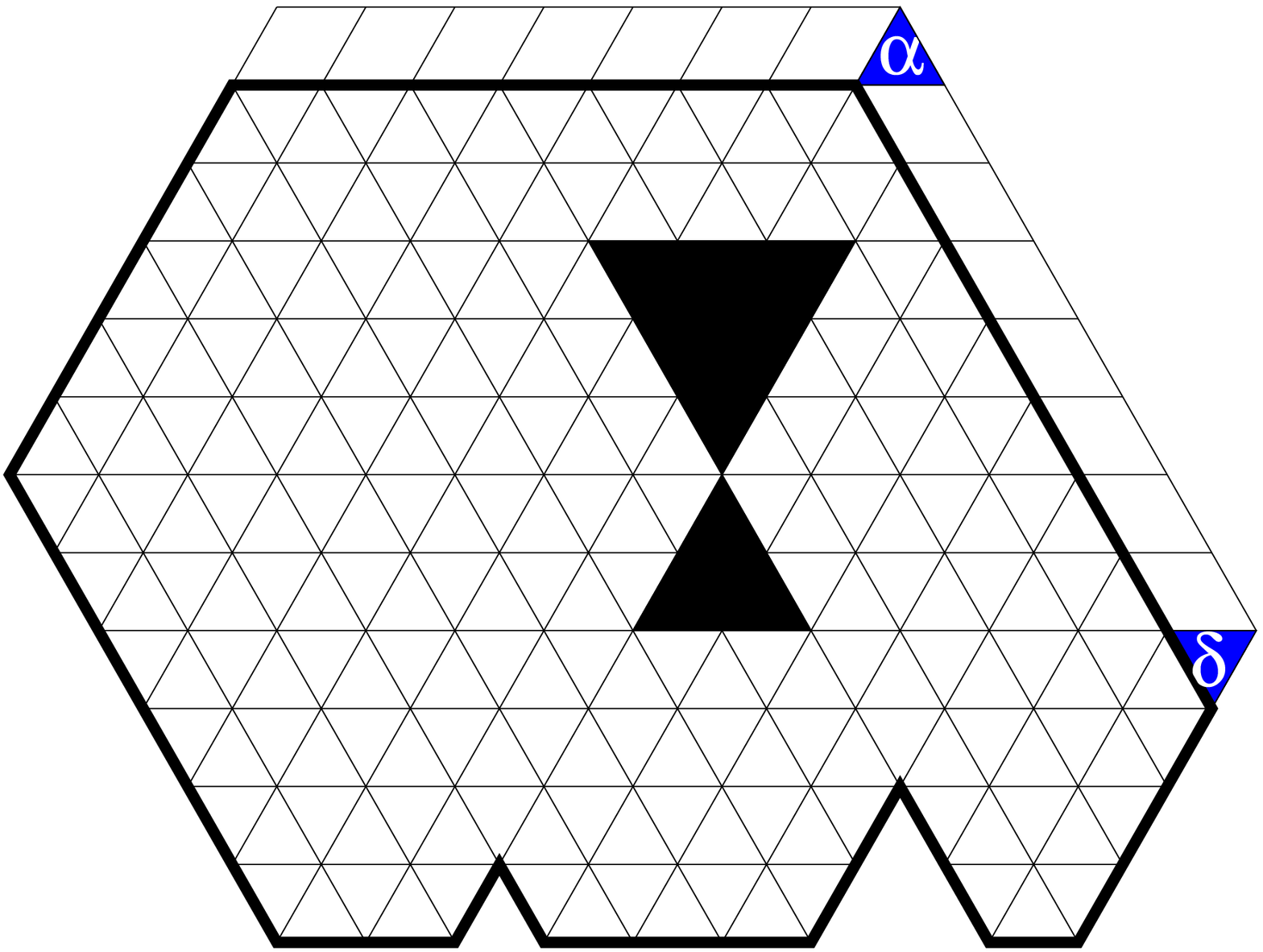}}
\hfill
}
  \caption{\label{fic} Obtaining the recurrence \eqref{eic} for the based hourglass regions.}
\end{figure}

Consider the hexagon $H$ cut out by the western angle of the outer hexagon and the western outside angle of the bowtie from the horizontal strip bounded by the base of the based hourglass region and the top of the bowtie. Suppose the bowtie touches the northwestern side. Then the top side of $H$ shrinks to a point, and since opposite sides of $H$ differ by the same amount, we obtain that $d-0=a-(y+a)$, which implies $y=0$. Therefore base case (4) reduces to base case (1). Base case (5) follows from Section 4, as the only way the bowtie can touch the southwestern or southeastern side is if it sits along the base and $b'=c'=0$.

Suppose $d=e=0$. Since the region $F_{0,0,f,x,y}(a,a',b',c')$ is by assumption tileable, it follows that the upward extension of the northeastern side of the $b'$-lobe crosses the northwestern side of the boundary (indeed, otherwise it would cross the interior of the top side, and lozenges forced by the forced lozenge at focal point $B$ would eventually leave a unit triangle in the northwestern corner that cannot be covered by a non-overlapping lozenge). An analogous statement holds for the $c'$-lobe. Therefore, the pattern of forced lozenges is as shown on the left in Figure \ref{fid}, and upon their removal one is left with an hourglass region. It is then not hard to see that equation \eqref{eia} follows from Proposition \ref{tha}.

\begin{figure}[h]
  \centerline{
\hfill
{\includegraphics[width=0.40\textwidth]{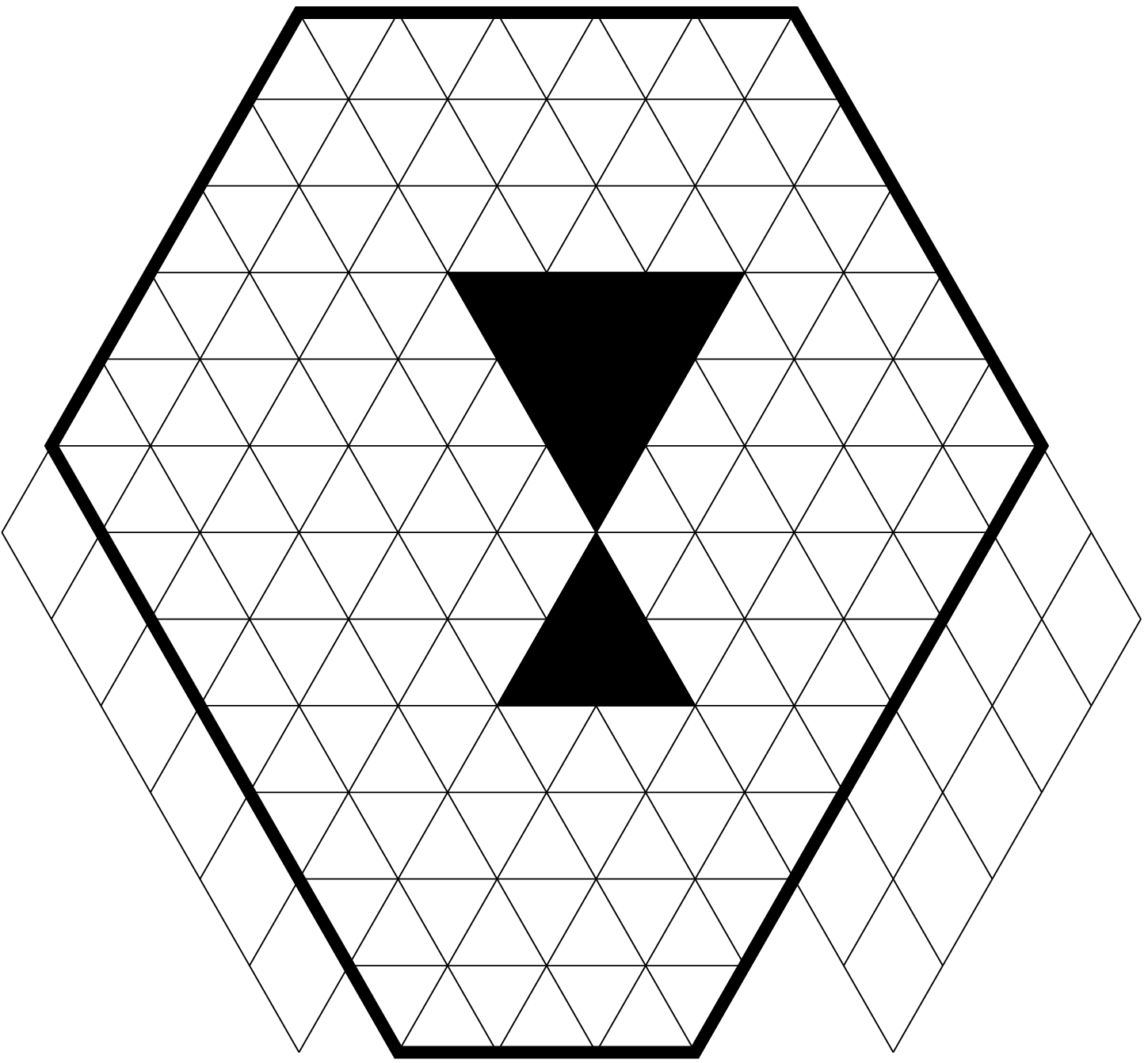}}
\hfill
{\includegraphics[width=0.40\textwidth]{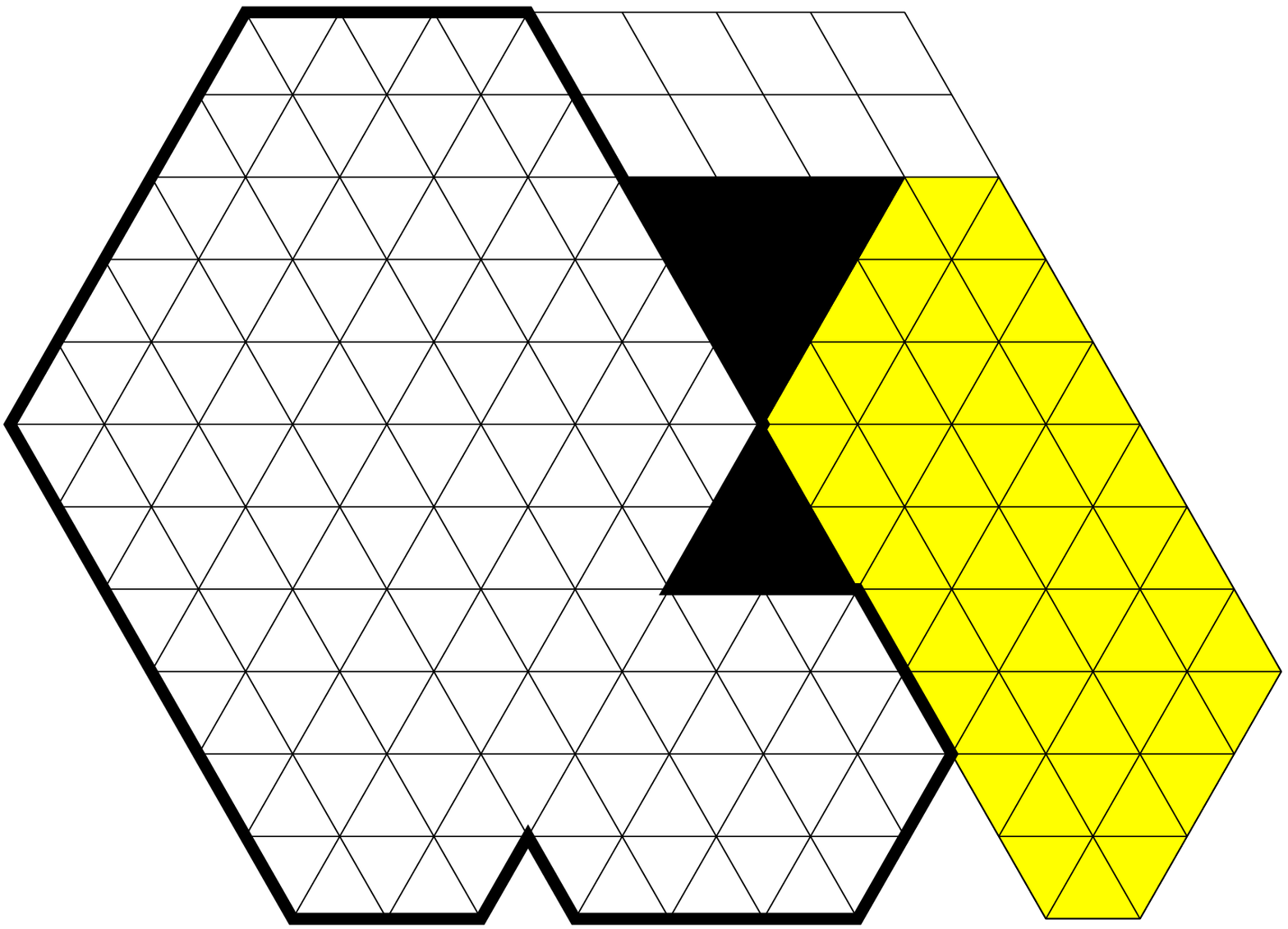}}
\hfill
}
  \caption{\label{fid} The base cases $d=e=0$ (left) and $z=0$ (right).}
\vskip-0.10in
\end{figure}

\begin{figure}[h]
\vskip0.15in
  \centerline{
\hfill
{\includegraphics[width=0.40\textwidth]{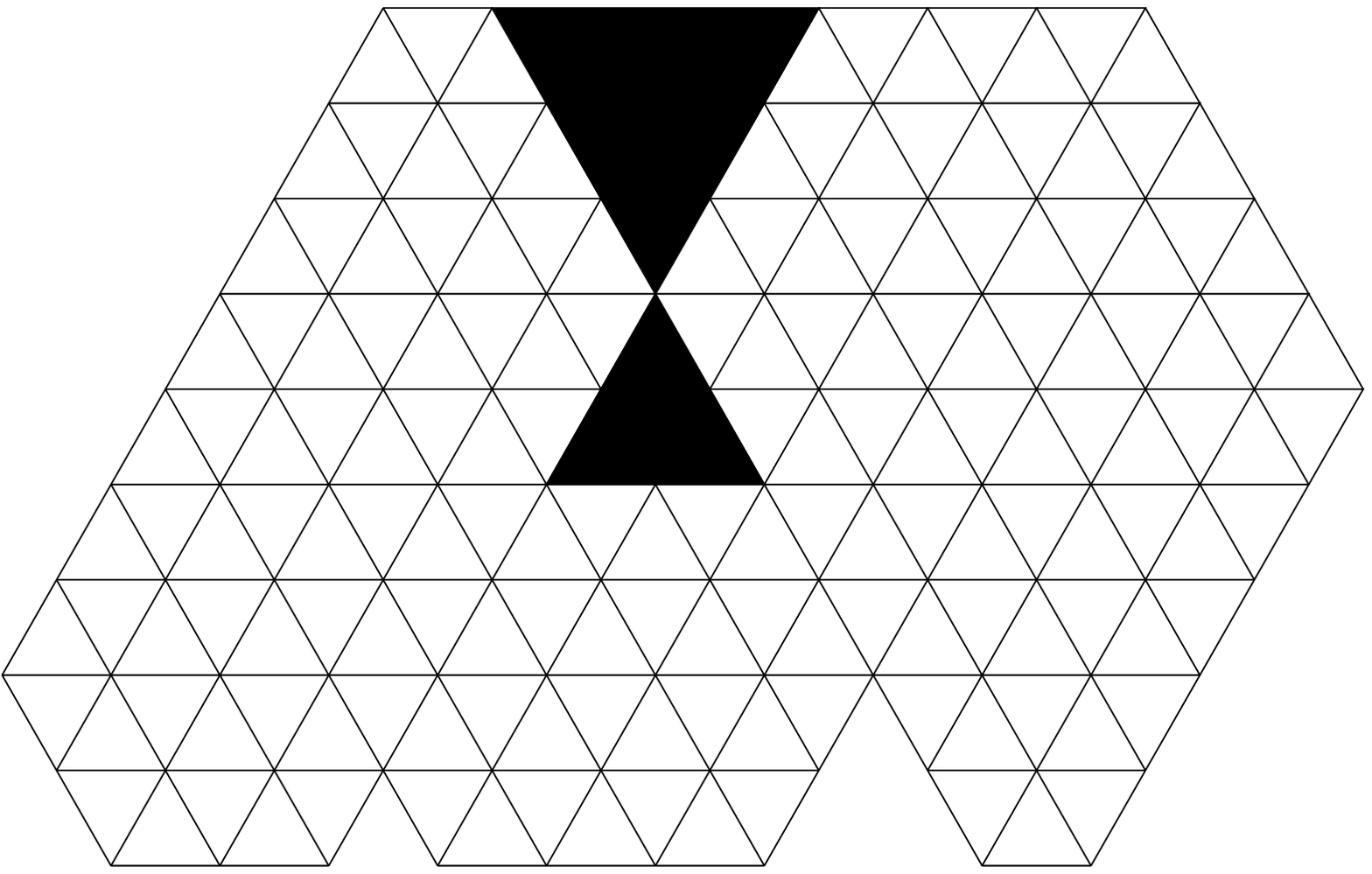}}
\hfill
{\includegraphics[width=0.40\textwidth]{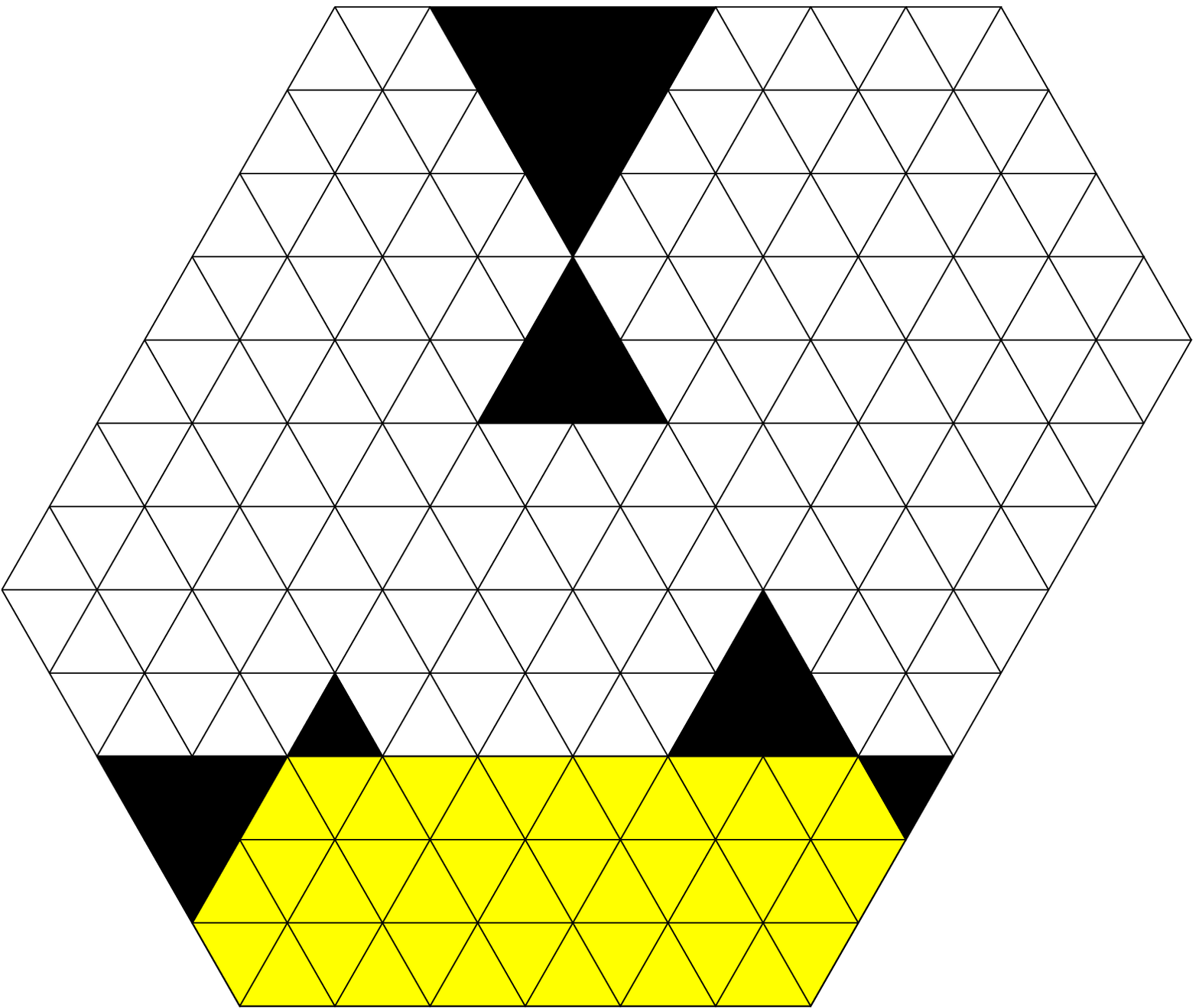}}
\hfill
}
  \caption{\label{fie} The base case when the bowtie touches the top side (left). Extension to a region with three bowties touching alternate sides of a hexagon (right); the shaded hexagon is always internally tiled.}
\vskip-0.10in
\end{figure}

Consider now the base case (1). By symmetry, it is enough to treat the case $z=0$. It follows from the dimensions in the picture on the left in Figure \ref{fia} that in this case the shaded hexagon shown on the right in Figure \ref{fid} is always internally tiled. This in turn forces the indicated lozenges to be present in every tiling. The region obtained after removing the shaded hexagon and these forced tiles is a snowman region. Thus, the number of lozenge tilings of the based hourglass region is in this case equal to the product of the number of tilings of the shaded hexagon (which is given by \eqref{eaa}) and the number of tilings of a snowman region (which is given by Theorem \ref{teb}). Using the resulting formula for both regions on the left hand side of \eqref{eia}, one readily verifies that the resulting expression agrees with the formula on the right hand side of \eqref{eia}.

The remaining base case of our induction is the case when the bowtie touches the top side. Then the based hourglass region is as pictured on the left in Figure \ref{fie}. Extend it downwards by including below it a hexagon of side-lengths $f+a'+b'+c'$, $c$, $b$, $f+a'+b'+c'$, $c$, $b$ (clockwise from top), as shown on the right in Figure \ref{fie}. The resulting region is a triad hexagon with the bowties sharing an edge with the boundary, and its tilings are therefore enumerated by Lai's formula \cite{Lai3dent} presented here in Theorem \ref{3btouching}. Since the added hexagon is necessarily internally tiled, and the number of its tilings is given by formula \eqref{eaa}, this yields an explicit product formula for the number of lozenge tilings of the based hourglass region in this base case. Using this formula for the numerator and denominator on the left hand side of \eqref{eib}, one readily checks that the resulting expression agrees with the one on the right hand side of \eqref{eib}.


For the induction step, let $d,e,f,y,z$ be non-negative integers with $y,z\geq1$, at least one of $d$ and $e$ positive and $f<y+z$ (as mentioned at the beginning of this section, the latter condition is equivalent to stating that the bowtie does not touch the top side), so that the bowtie does not touch any of the sides above the base. Assume that~\eqref{eia} holds for all based hourglass regions whose $d$-, $e$-, $x$- and $y$-parameters add up to strictly less than $d+e+y+z$. We need to deduce that \eqref{eia} holds for $F_{d,e,f,y,z}(a,a',b',c')$. 

By the second pararaph following \eqref{eic}, we may assume without loss of generality that $d\geq1$. Under our assumptions, $F_{d-1,e,f,y,z-1}(a,a',b',c')$ is a well-defined based hourglass region. Therefore, as we have seen in the first paragraph of this proof, $\M(F_{d-1,e,f,y,z-1}(a,a',b',c'))>0$, and we can use identity \eqref{eic} to express  $\M(F_{d,e,f,y,z}(a,a',b',c'))$ in terms of tiling counts of five smaller based hourglass regions, for which the induction hypothesis applies. Using formula~\eqref{eia} for each of these five regions, we obtain this way an expression for $\M(F_{d,e,f,y,z}(a,a',b',c'))$. We need to show that this expression agrees with the one provided by equality \eqref{eia}.

Using Remark 3, after clearing denominators, the equality we need to check becomes an equality very similar to equation \eqref{eda}. The argument that proves \eqref{eda} (presented in detail in Section 9) is readilly seen to prove the needed equality as well. \epf


\section{Sphinx regions}

In this section we consider the family of regions $X_{d,x,y,z}(a,b,c,a',b',c')$ described in Figure \ref{fja}; we call them {\it sphinx regions}. The region $X_{d,x,y,z}(a,b,c,a',b',c')$ is defined for any non-negative integers $d$, $x$, $y$, $z$, $a$, $b$, $c$, $a'$, $b'$ and $c'$ satisfying $x\leq y+z$ (this is equivalent to the statement that the top of the bowtie is weakly below the top of the outer hexagon; this follows from  the readily checked fact that the focal distance is $x+d+a'+b'+c'$).

Given a sphinx region $X=X_{d,x,y,z}(a,b,c,a',b',c')$, it is not hard to see that its corresponding region $\thickbar{X}$ is the region described in Figure \ref{fjb}. 

\begin{prop}
\label{sphinx}
Let $d$, $x$, $y$, $z$, $a$, $b$, $c$, $a'$, $b'$ and $c'$ be non-negative integers with $x\leq y+z$, and consider the sphinx region $X=X_{d,x,y,z}(a,b,c,a',b',c')$ and its corresponding region $\thickbar{X}$. Then
\begin{align}
\frac{\M(X)}{\M(\thickbar{X})}
=
\dfrac
{\w^{(X)}\dfrac{\ka_A^{(X)}\ka_B^{(X)}\ka_C^{(X)}}{\ka_{BC}^{(X)}\ka_{AC}^{(X)}\ka_{AB}^{(X)}}}
{\w^{(\thickbar{X})}\dfrac{\ka_{A_0}^{(\thickbar{X})}\ka_{B_0}^{(\thickbar{X})}\ka_{C_0}^{(\thickbar{X})}}
{\ka_{B_0C_0}^{(\thickbar{X})}\ka_{A_0C_0}^{(\thickbar{X})}\ka_{A_0B_0}^{(\thickbar{X})}}},
\label{eja}
\end{align}
where the couples\footnote{ Recall that the couples $\ka$ are defined by equations (2.3)--(2.8).} $\ka$ in the numerator fraction refer to the region $X$ with focal points $A$, $B$, $C$ as indicated in Figure $\ref{fja}$, and the ones in the denominator fraction refer to the region $\thickbar{X}$ with focal points $A_0$, $B_0$, $C_0$ as indicated in Figure $\ref{fjb}$. Explicitly, we have
\begin{align}
\ka_A^{(X)}&=H(a+y+z-x)H(x+b+c+2d+a'+b'+c')
\label{ejaa}
\\
\ka_B^{(X)}&=H(b)H(x+z+d+a+c+a'+b'+c')
\label{ejab}
\\
\ka_C^{(X)}&=H(c)H(x+y+d+a+b+a'+b'+c')
\label{ejac}
\\
%
%
\ka_{BC}^{(X)}&=H(y+z+d+a+a'+b'+c')H(d+b+c)
\label{ejad}
\\
\ka_{AC}^{(X)}&=H(x+d+b+a'+b'+c')H(z+a+c)
\label{ejae}
\\
\ka_{AB}^{(X)}&=H(x+d+c+a'+b'+c')H(y+a+b)
\label{ejaf}
\end{align}

\begin{figure}[h]
  \centerline{
\hfill
{\includegraphics[width=0.70\textwidth]{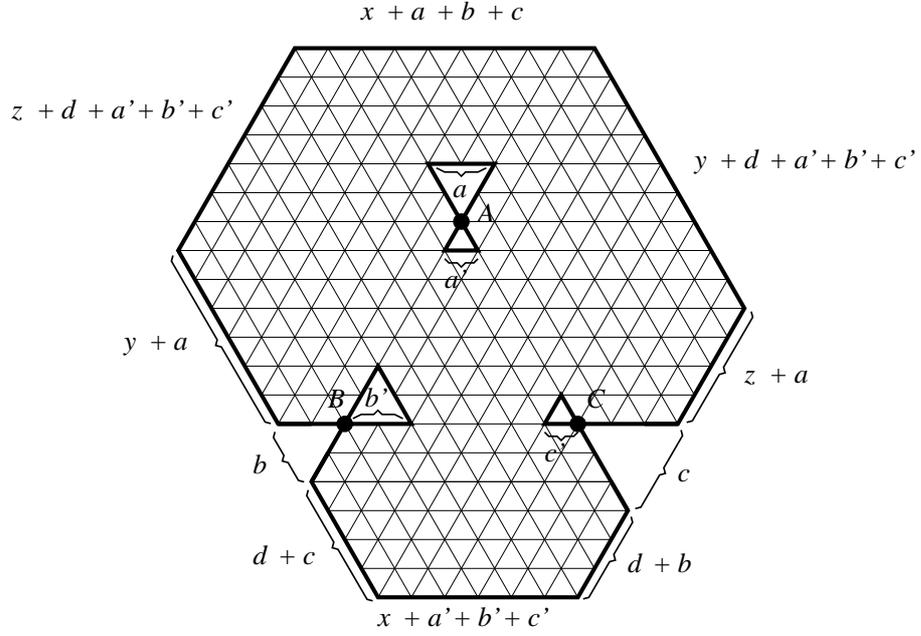}}
\hfill
}
  \caption{\label{fja} The sphinx region $X=X_{d,x,y,z}(a,b,c,a',b',c')$ for $d=1$, $x=2$, $y=4$, $z=2$, $a=2$, $b=2$, $c=3$, $a'=1$, $b'=2$, $c'=1$.}
\vskip-0.10in
\end{figure}

\begin{figure}[h]
\vskip0.15in
  \centerline{
\hfill
{\includegraphics[width=0.70\textwidth]{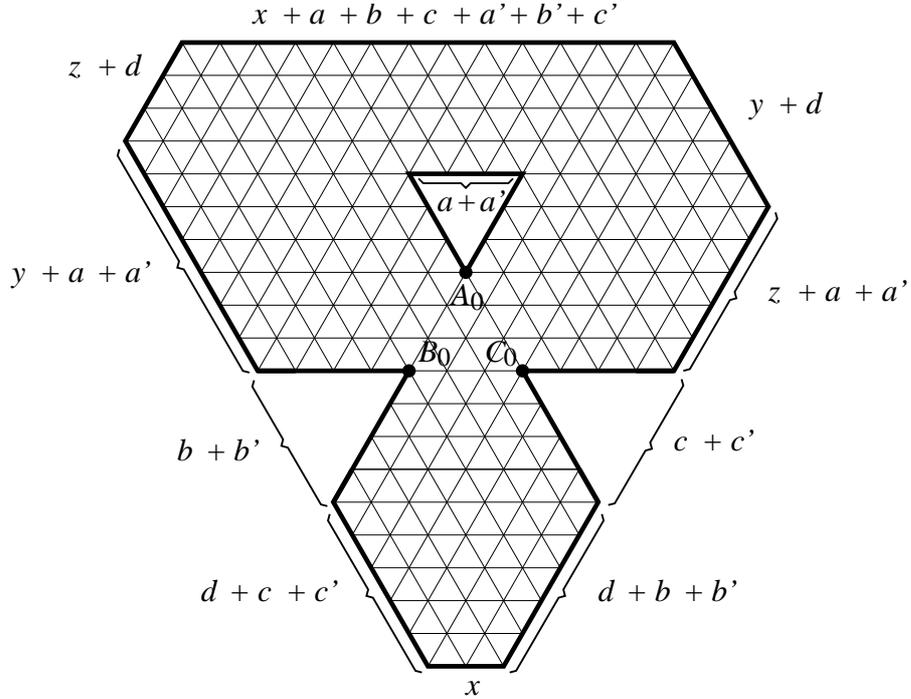}}
\hfill
}
\vskip-0.15in
  \caption{\label{fjb} The region $\bar{S}$ corresponding to the sphinx region $X$ in Figure \ref{fja}.}
\end{figure}

and
\begin{align}
\ka_{A_0}^{(\thickbar{X})}&=H(x+2d+b+b'+c+c')H(y+z+a+a'-x)
\label{ejag}
\\
\ka_{B_0}^{(\thickbar{X})}&=H(b+b')H(x+z+d+a+a'+c+c')
\label{ejah}
\\
\ka_{C_0}^{(\thickbar{X})}&=H(c+c')H(x+y+d+a+a'+b+b')
\label{ejai}
%
%
%
\\
\ka_{B_0C_0}^{(\thickbar{X})}&=H(d+b+b'+c+c')H(y+z+d+a+a')
\label{ejaj}
\\
\ka_{A_0C_0}^{(\thickbar{X})}&=H(x+d+b+b')H(z+a+a'+c+c')
\label{ejak}
\\
\ka_{A_0B_0}^{(\thickbar{X})}&=H(x+d+c+c')H(y+a+a'+b+b').
\label{ejal}
\end{align}

\end{prop}

\begin{figure}[h]
  \centerline{
\hfill
{\includegraphics[width=0.38\textwidth]{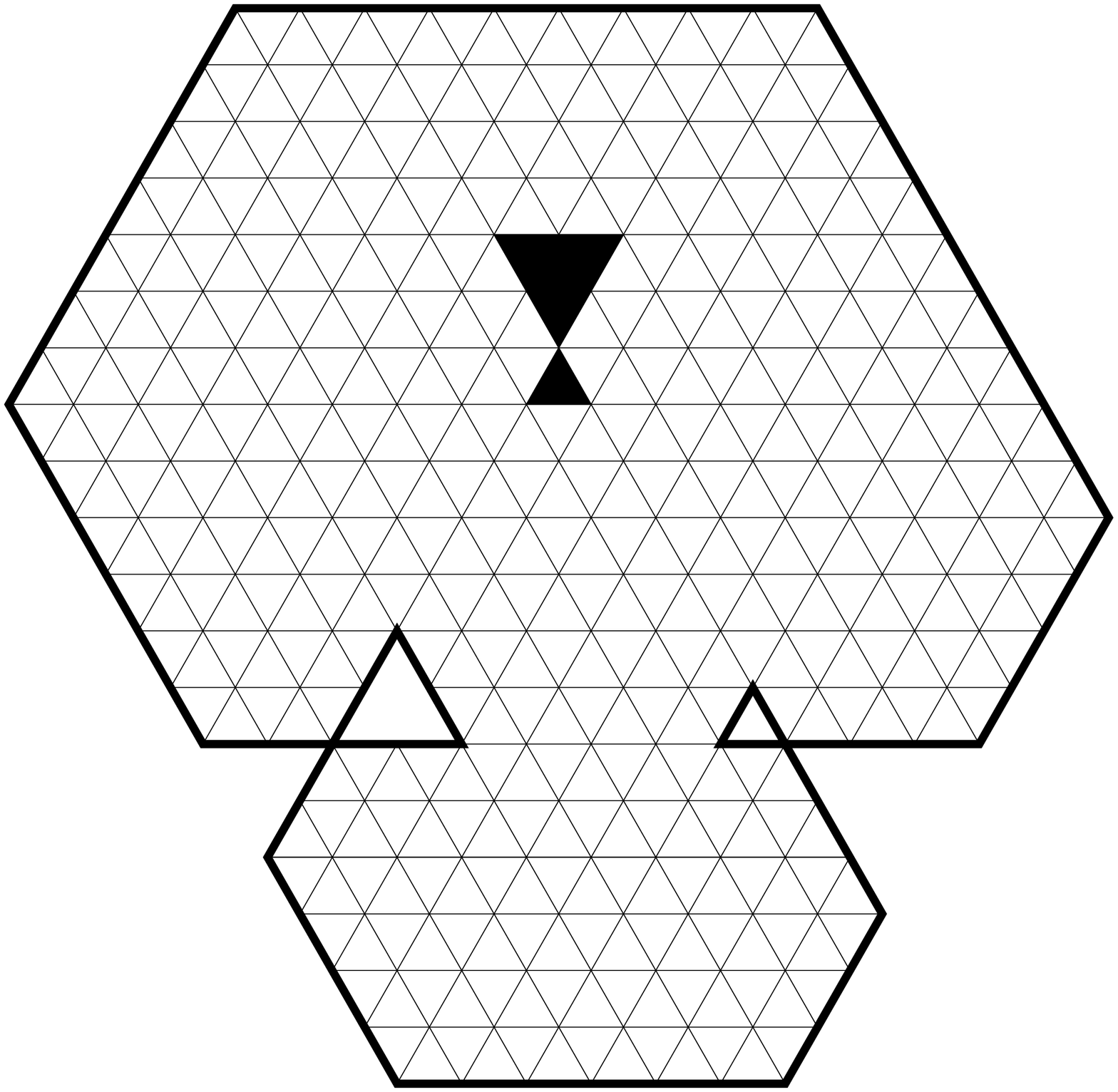}}
\hfill
{\includegraphics[width=0.38\textwidth]{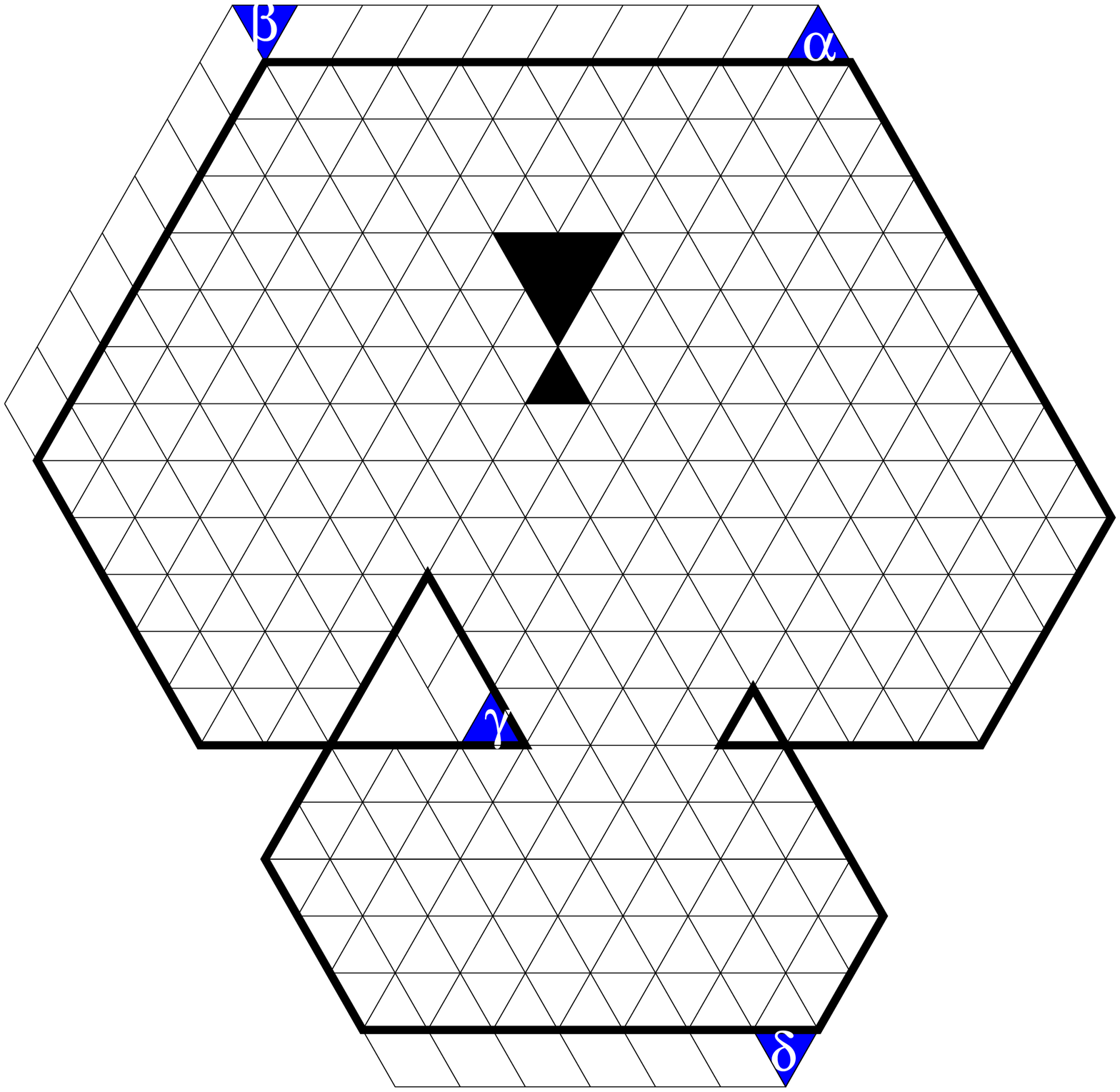}}
\hfill
}
\vskip0.2in
  \centerline{
\hfill
{\includegraphics[width=0.38\textwidth]{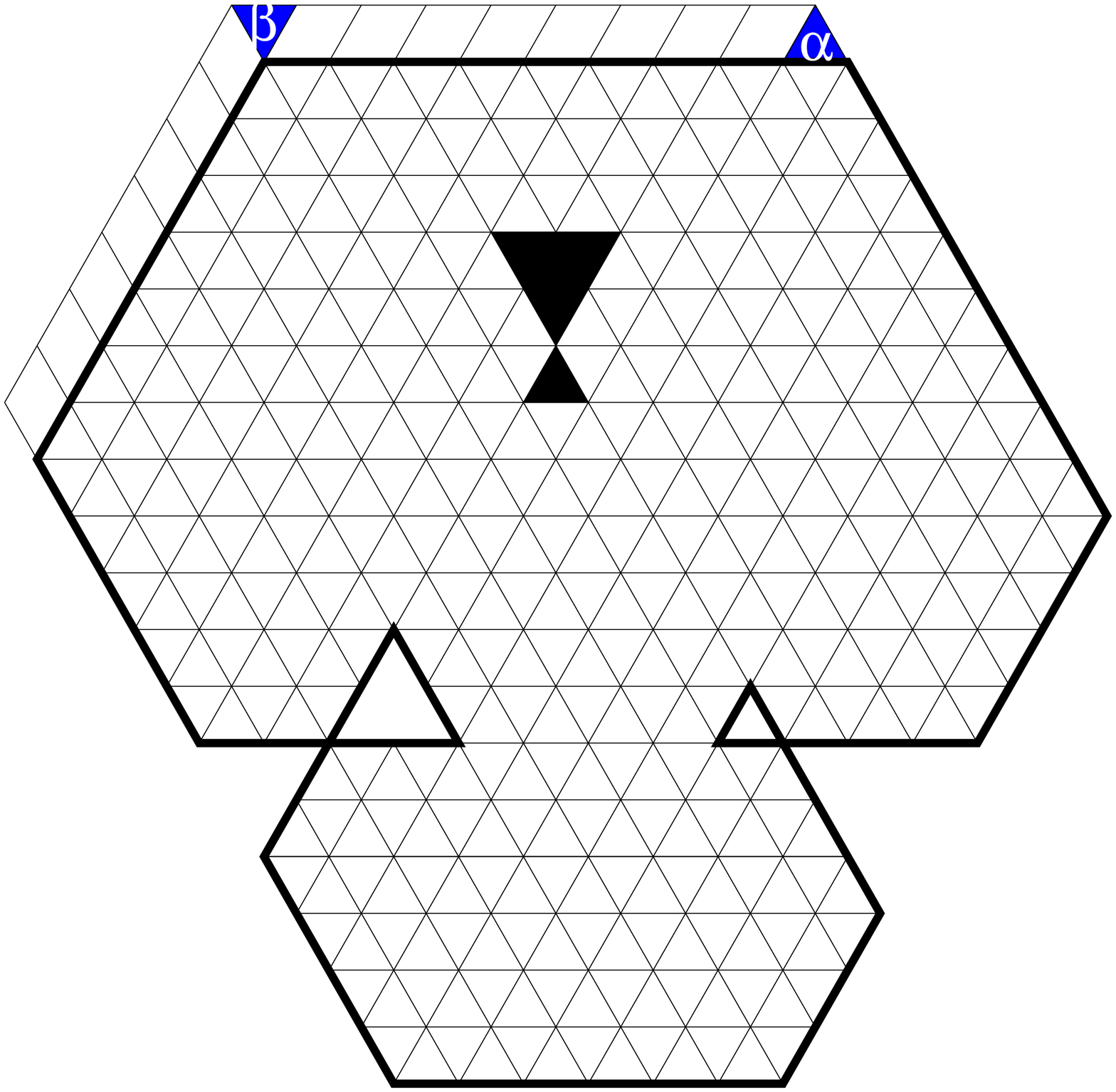}}
\hfill
{\includegraphics[width=0.38\textwidth]{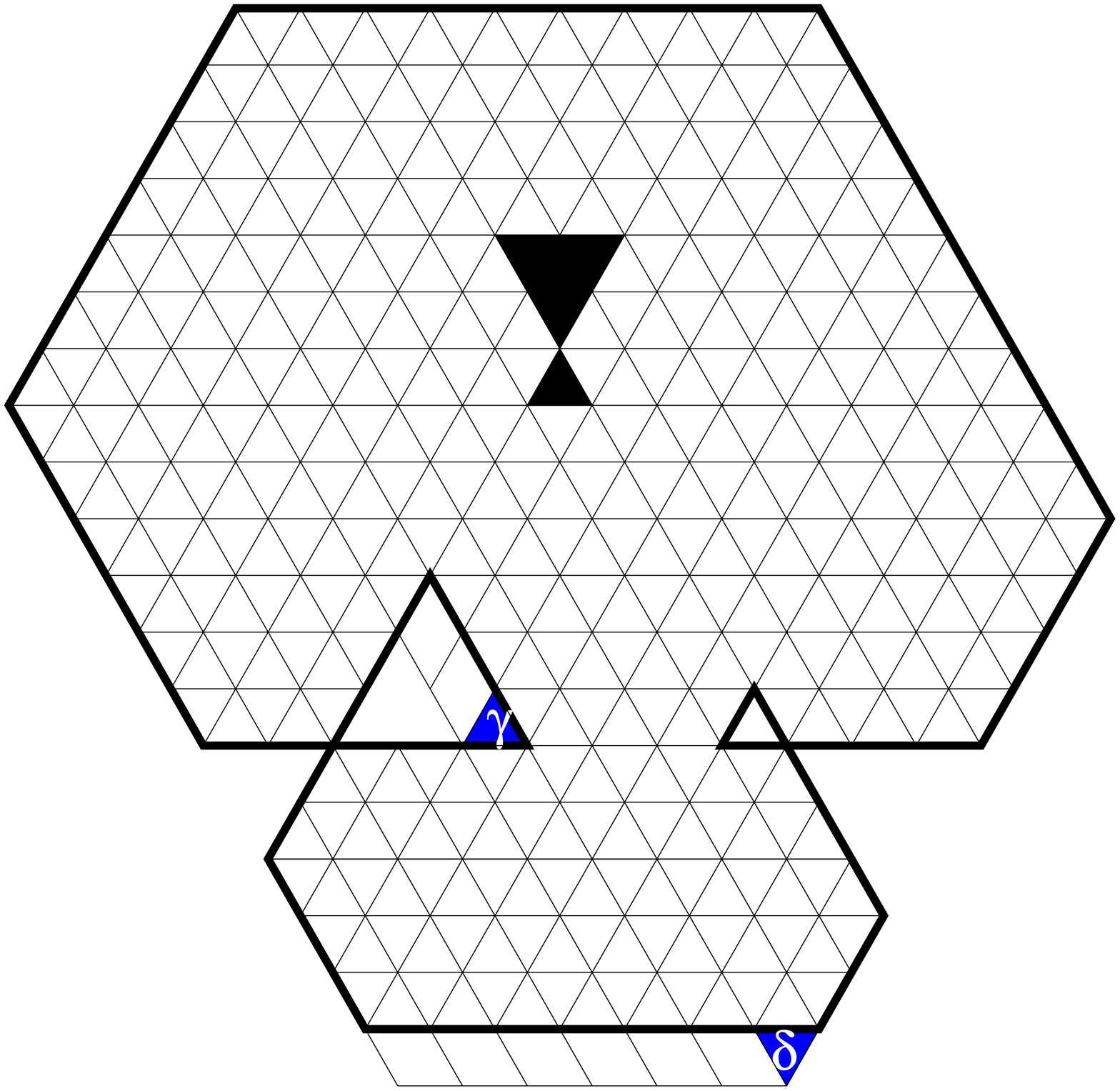}}
\hfill
}
\vskip0.2in
  \centerline{
\hfill
{\includegraphics[width=0.38\textwidth]{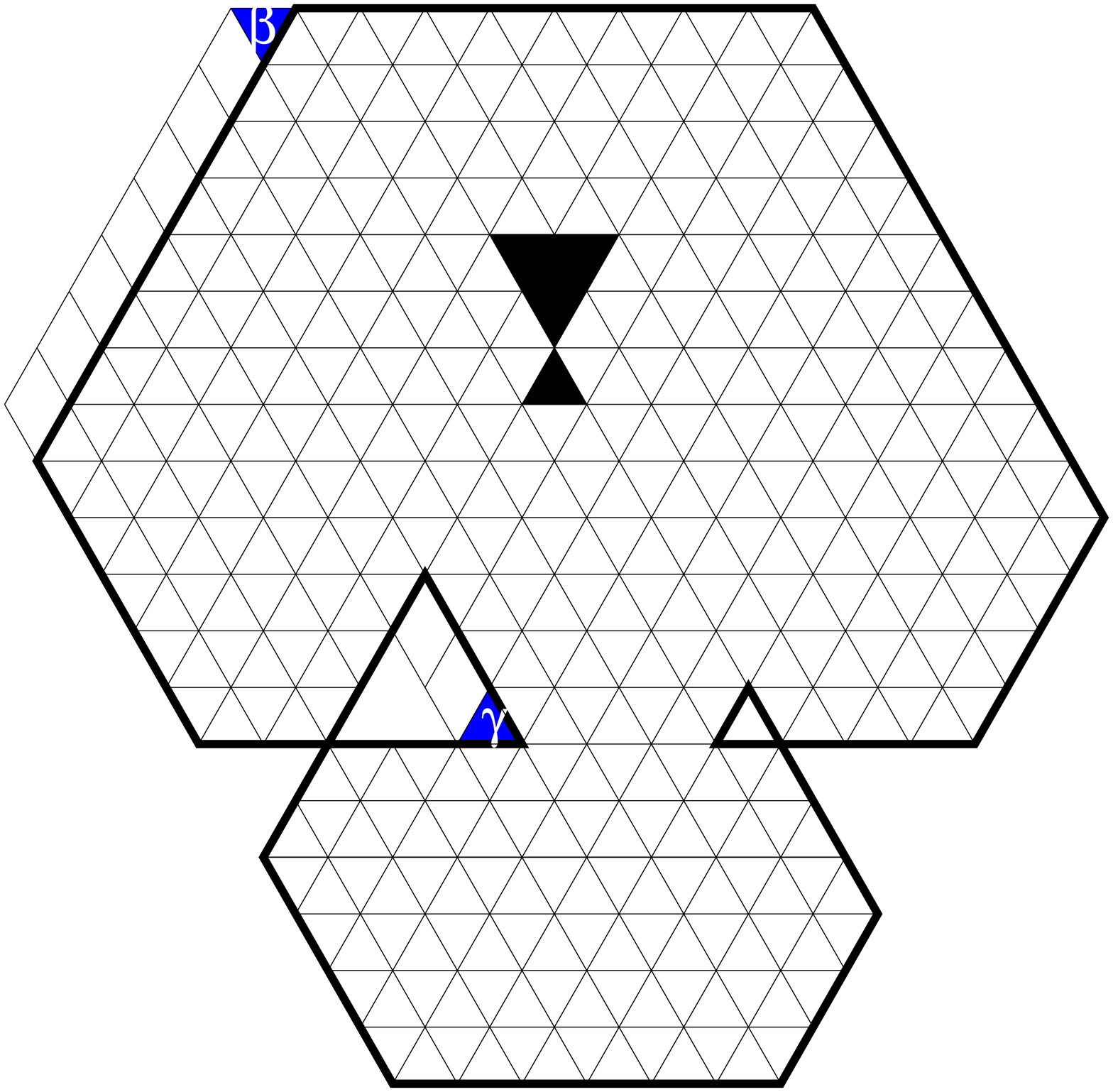}}
\hfill
{\includegraphics[width=0.38\textwidth]{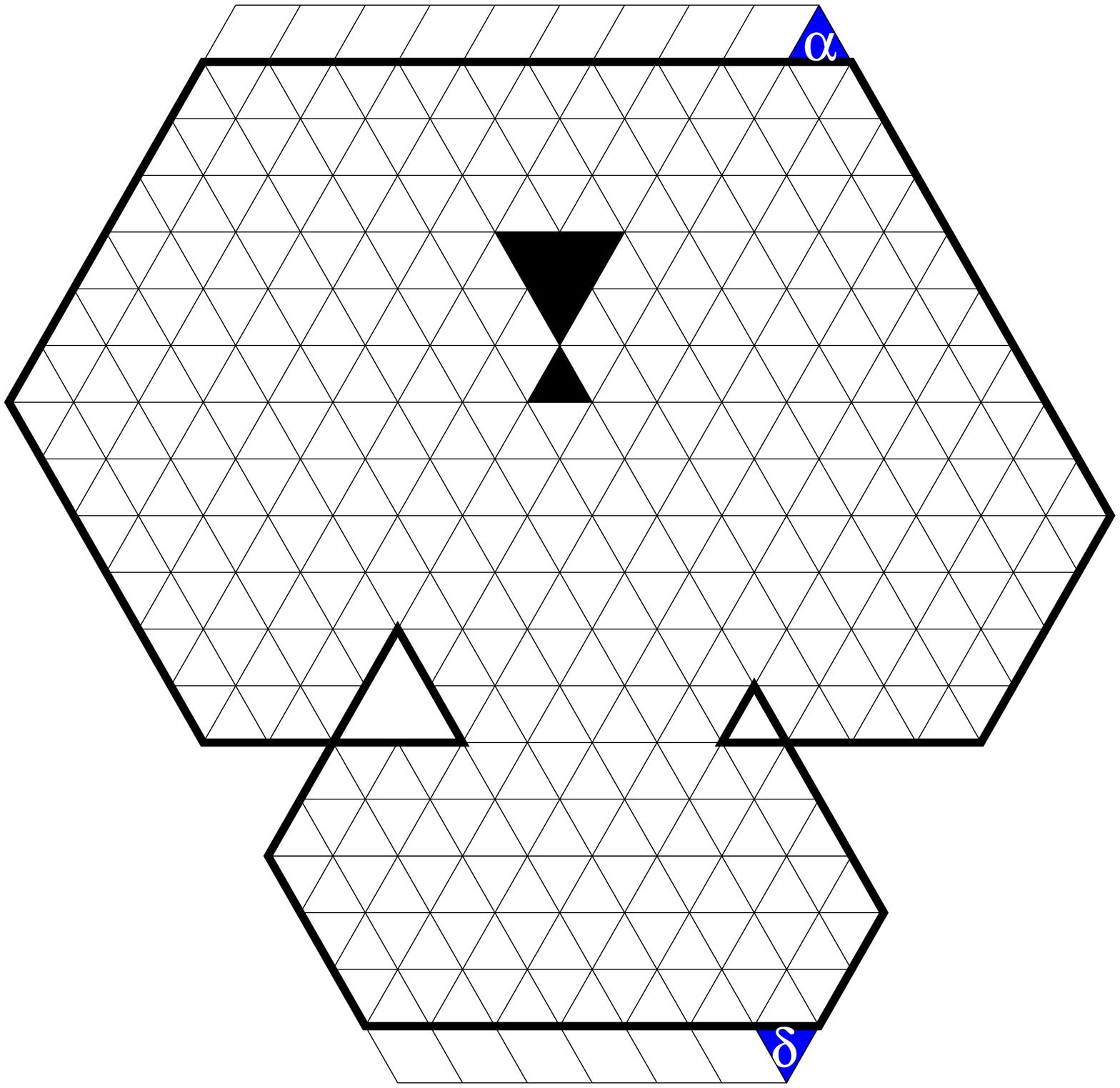}}
\hfill
}
  \caption{\label{fjc} Obtaining the recurrence \eqref{ejb} for the sphinx regions.}
\end{figure}

\begin{proof} When regarded as a triad hexagon, all three depths of a sphinx region are non-negative. Therefore, by Lemma \ref{tileability}, all sphinx regions are tileable. In particular, $\M(\thickbar{S})\neq0$, so the ratio on the left hand side of \eqref{eja} is well defined.

We prove the statement by induction, using Kuo condensation at the induction step. The picture on the top left in Figure \ref{fjc} shows the region $X_{d,x,y,z}(a,b,c,a',b',c')$. Choosing $\alpha$, $\beta$, $\gamma$ and $\delta$ as shown in the top right picture in Figure \ref{fjc}, and assuming the forced lozenges come in the pattern shown in Figure \ref{fjc}, we obtain
\begin{align}
&
\M(X_{d,x,y,z}(a,b,c,a',b',c'))\M(X_{d-1,x,y-1,z}(a,b,c,a'+1,b',c'))
=
\nonumber
\\[5pt]
&\ \ \ \ \ \ \ \ \ \ \ \ \ \ \ \ \ \ \ \ \ \ 
\M(X_{d,x,y-1,z}(a,b,c,a',b',c'))\M(X_{d-1,x,y,z}(a,b,c,a'+1,b',c'))
\nonumber
\\[5pt]
&\ \ \ \ \ \ \ \ \ \ \ \ \ \ \ \ \ \ \ \ \ \
+\M(X_{d,x-1,y-1,z}(a,b,c,a'+1,b',c'))\M(X_{d-1,x+1,y,z}(a,b,c,a',b',c')).
\label{ejb}
\end{align}  

\parindent0pt
Note that if $d,x,y,z\geq1$ and $x<y+z$, then all the regions involved in \eqref{ejb} are well defined.
Moreover, these inequalities imply that ``there is room'' for the unit triangles $\alpha$, $\beta$, $\gamma$  and $\delta$ shown in Figure \ref{fjc}: The sides of $X$ along which the unit triangles $\alpha$, $\beta$ and $\delta$ are taken are all positive under the assumed inequalities, and so is the distance between the $b'$- and $c'$-lobes (indeed, the latter is readily seen to be equal to $x+d$).
Furthermore, if we assume in addition that the bowtie does not touch the northwestern side, it is easy to see that these inequalities also imply that the pattern of forced lozenges is as shown in Figure \ref{fjc}, and therefore identity \eqref{ejb} holds. Since $X_{d-1,x,y-1,z}(a,b,c,a'+1,b',c')$ is a well defined sphinx region, and since all sphinx regions are tileable (see the first paragraph in this proof), it follows that we can use \eqref{ejb} to express $\M(X_{d,x,y,z}(a,b,c,a',b',c'))$ in terms of tilings counts of the other five sphinx regions. As the sum of the $d$- and $y$-parameters is strictly less in these five regions than in $X_{d,x,y,z}(a,b,c,a',b',c')$, we proceed by induction on $d+y$.


\parindent15pt
Our base cases will be (1) $d=0$; (2) $x=0$; (3) $y=0$ or $z=0$, (4) $x=y+z$;
and (5)  the case when the bowtie touches the northwestern side of the hexagon.

In case (5), after removing the forced lozenges, one is left with a snowman region, and the statement follows by Theorem \ref{teb}. 

If $d=0$, the hexagonal sub-region of $X$ under the line $BC$ has opposite sides of equal lenths, and therefore must be internally tiled. Since the leftover region is a based hourglass region, this case follows from Proposition \ref{basedhg} and formula \eqref{eaa}.

For $y=0$, the side of $X$ of length $y+a$ matches the length of the left side of the $a$-lobe. This creates a hexagonal sub-region of $X$ in its western corner that must be internally tiled. In turn, this determines a parallelogram of forced lozenges
(this situation is similar to the one in the picture on the right in Figure \ref{fid}).
The region left after removing the internally tiled hexagon and the forced parallelogram is a special case of based hourglass regions. Thus this case also follows from Proposition \ref{basedhg} and formula \eqref{eaa}. The case $z=0$ follows by symmetry.

Base case (4) follows by Theorem \ref{3btouching}, and base case (2) by Section 7.

For the induction step, let $d,x,y,z\geq1$ be integers with $x<y+z$ (as mentioned at the beginning of this section, the latter condition is equivalent to stating that the bowtie does not touch the top side). Assume that~\eqref{eja} holds for all sphinx regions for which the sum of their $d$- and $y$-parameters is strictly less than $d+y$. We need to deduce that \eqref{eja} holds for $X_{d,x,y,z}(a,b,c,a',b',c')$. 

As described in the paragraph following \eqref{ejb}, express  $\M(X_{d,x,y,z}(a,b,c,a',b',c'))$ in terms of tiling counts of five smaller sphinx regions, for which the induction hypothesis applies. Using formula~\eqref{eja} for each of these five regions, we obtain this way an explicit expression for $\M(X_{d,x,y,z}(a,b,c,a',b',c'))$. We need to show that this expression agrees with the one provided by equality \eqref{eja}.

After clearing denominators, the equality we need to check becomes an equality very similar to equation \eqref{eda}. The argument that proves \eqref{eda} (presented in detail in Section 9) is readilly seen to prove the needed equality as well. 
\end{proof}

\begin{figure}[h]
  \centerline{
\hfill
{\includegraphics[width=0.44\textwidth]{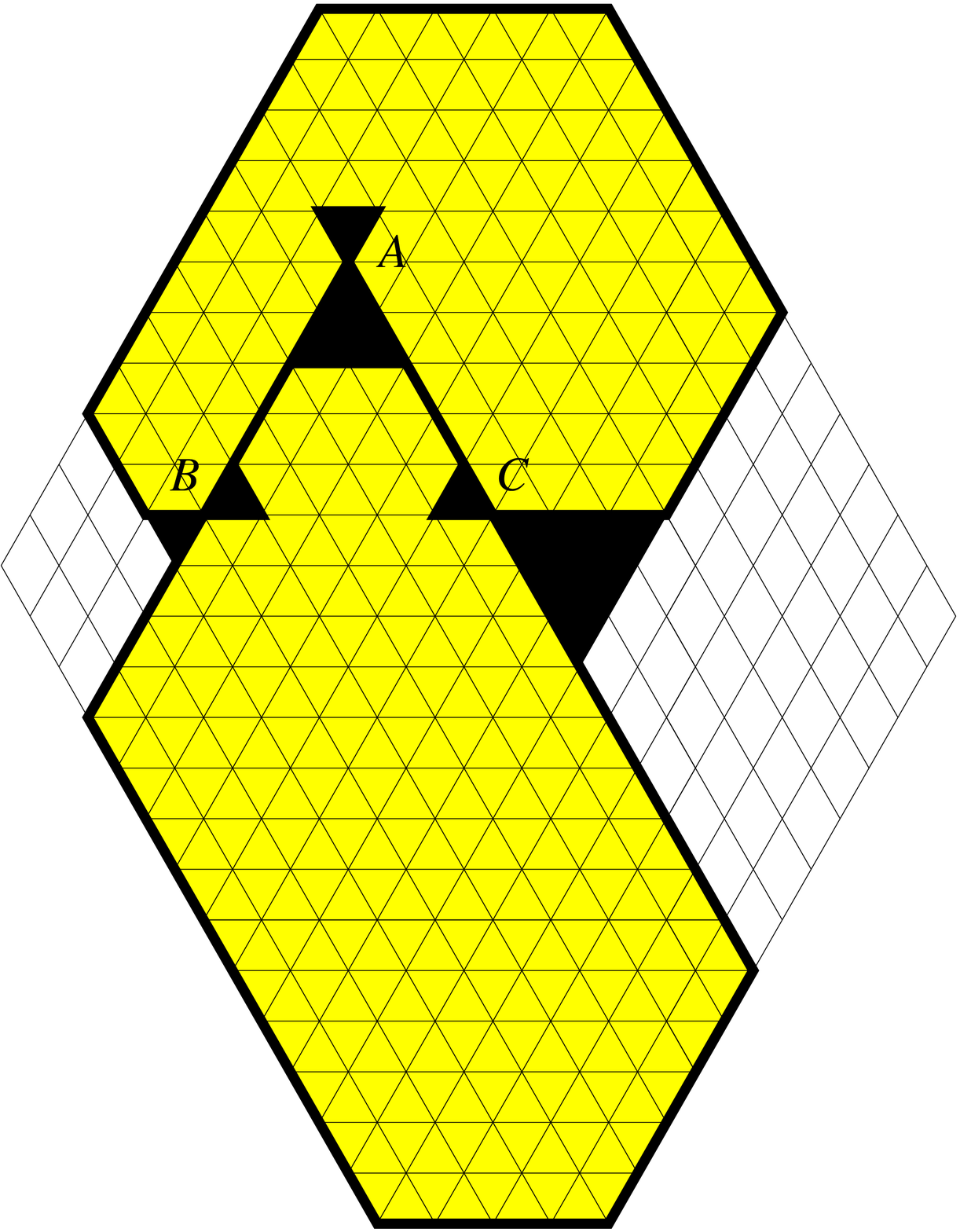}}
\hfill
{\includegraphics[width=0.44\textwidth]{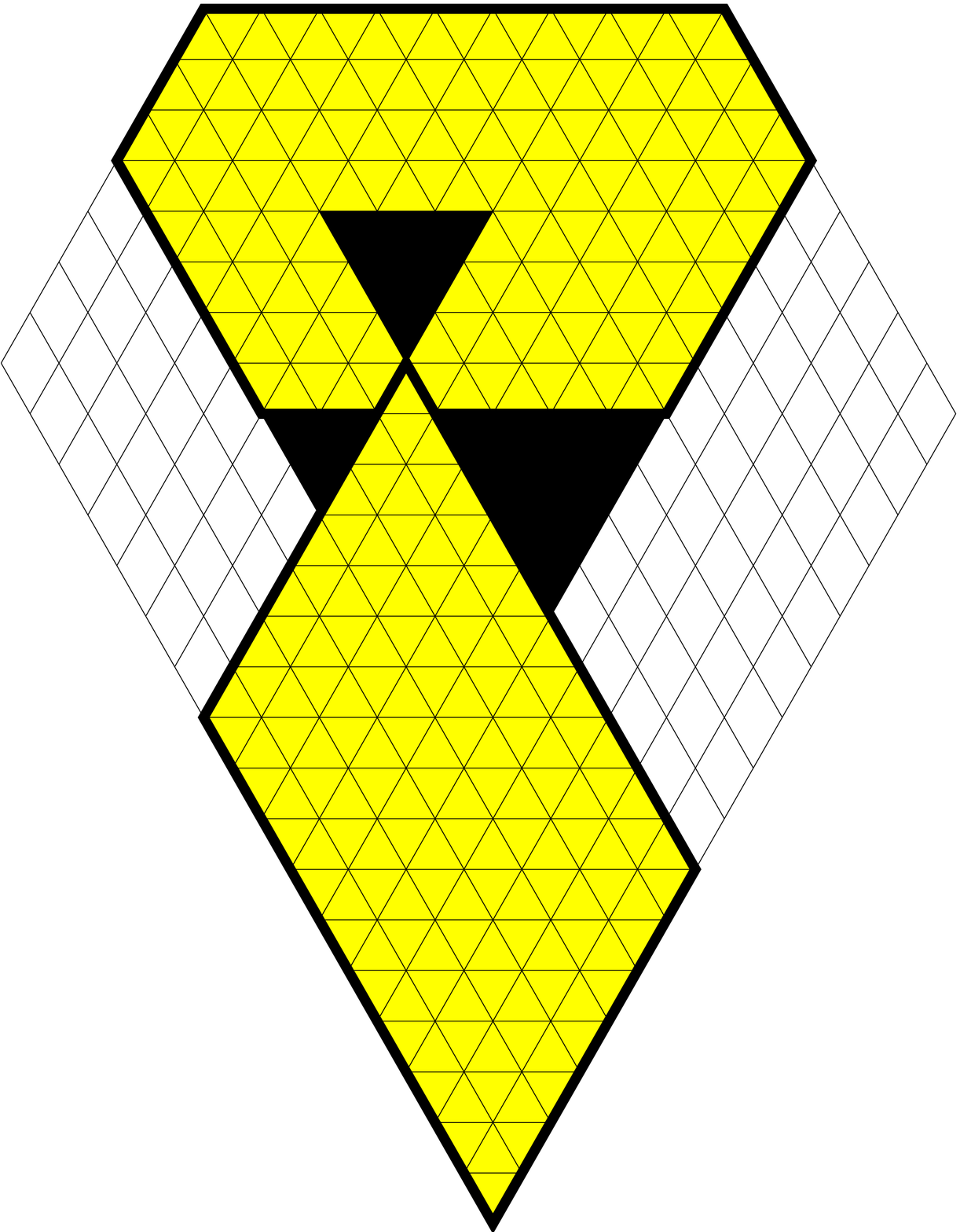}}
\hfill
}
  \caption{\label{ffa} An example of a triad hexagon $R$ with $x=0$ (left) and the corresponding region $\bar{R}$ (right).}
\end{figure}

\section{The case $x=0$}


In this section we prove the special case of Theorem \ref{tba} when $x=0$ and $Q=\thickbar{R}$. We need to prove that
\begin{equation}
\label{efa}
\frac{\M(R)}{\M(\thickbar{R})}=\dfrac
{\w^{(R)}\dfrac{\ka_A^{(R)}\ka_B^{(R)}\ka_C^{(R)}}{\ka_{BC}^{(R)}\ka_{AC}^{(R)}\ka_{AB}^{(R)}}}
{\w^{(\thickbar{R})}\dfrac{\ka_{A_0}^{(\thickbar{R})}\ka_{B_0}^{(\thickbar{R})}\ka_{C_0}^{(\thickbar{R})}}{\ka_{B_0C_0}^{(\thickbar{R})}\ka_{A_0C_0}^{(\thickbar{R})}\ka_{A_0B_0}^{(\thickbar{R})}}},
\end{equation}

\parindent0pt
where $\thickbar{R}$ is obtained from $R$ by squeezing out completely all three bowties (see Section 2 for the definition), and $A_0$, $B_0$ and $C_0$ are its focal points.

\parindent15pt
Suppose $x=0$, and consider the triad hexagon $R=R_{0,y,z}^{A,B.C}(a,b,c,a',b',c')$ (an example is shown on the left in Figure \ref{ffa}). Consider also the region $\thickbar{R}$ obtained from $R$ by completely squeezing out its three bowties (see the picture on the right in Figure \ref{ffa}). 

Due to the fact that in this case the length of the top side of $R$ is $a+b+c$, in any tiling of $R$, the paths of lozenges that start upward from the lobes of sizes $a$, $b$ and $c$ are all the paths that end on the top side. This implies that the top shaded hexagon $I$ in Figure~\ref{ffa} is always internally tiled.

Since the length of the bottom side of $R$ is $a'+b'+c'$, the same argument shows that the bottom shaded hexagon $S$ on the left in Figure \ref{ffa} must also be internally tiled. Since the lozenge tiling is forced on the leftover portion of $R$ (see Figure \ref{ffa}), it follows that
\begin{equation}
\M(R)=\M(I)\M(S).
\label{efb}  
\end{equation}
The same argument shows that
\begin{equation}
\M(\thickbar{R})=\M(I')\M(S'),
\label{efc}  
\end{equation}
where $S'$ and $I'$ are the top and bottom shaded regions on the right in Figure \ref{ffa}, respectively.

Combining the above two equations gives
\begin{equation}
\frac{\M(R)}{\M(\thickbar{R})}=\frac{\M(I)}{\M(I')}\frac{\M(S)}{\M(S')}.
\label{efd}  
\end{equation}
Express both $\M(I)$ and $\M(I')$ using the formula of Theorem \ref{tea}. We claim that the resulting $P$-parts from the right hand side of \eqref{eea} are equal, and thus cancel out in the fraction $\M(I)/\M(I')$ in \eqref{efd}.

To see this, note that, when $I$ is obtained from $R$ as the top shaded region on the left in Figure \ref{ffa}, we have the following interpretation for the quantities in the arguments of $P$ in \eqref{eea}:
\begin{align}
x&=(\ell(NW)-f)-\de(b\textrm{-lobe},SW) 
\label{efea}
\\
y&=(\ell(NE)-f)-\de(c\textrm{-lobe},SE) 
\label{efeb}
\\
a+b+c+m&=\ell(N)+f,
\label{efec}
\end{align}
where $\ell(s)$ denotes the length of the side $s$ of the outer hexagon, and all things on the right hand sides refer to the region $R$. 

The same realization shows that, when the region $I$ in \eqref{eea} is obtained from $\thickbar{R}$ as the top shaded region on the right in Figure \ref{ffa}, the quantities $x$, $y$ and $a+b+c+m$ in the resulting arguments of $P$ have the same interpretation \eqref{efea}--\eqref{efec}, with the only difference that now the things on the right hand sides of \eqref{efea}--\eqref{efec} refer to the region $\thickbar{R}$.

However, it is a consequence of our definition of the bowtie squeezing operation that the quantities $\ell(NW)-f$, $\ell(N)+f$, $\ell(NE)-f$, as well as the distance of an outer lobe to the side facing it away from the other two bowties are invariant under bowtie squeezing. This proves our claim.

A similar argument proves the analogous claim that, when we express both $\M(S)$ and $\M(S')$ using the formula of Theorem \ref{teb}, the resulting $P'$-parts from the right hand side of \eqref{eeb} are equal, and thus cancel out in the fraction $\M(S)/\M(S')$ in \eqref{efd}. To see this, the needed analogs of \eqref{efea}--\eqref{efec} are
\begin{align}
k&=f-a'-b'-c'  
\label{efeea}
\\
x+b&=\de(b'\textrm{-lobe},SW)
\label{efeeb}
\\
y+c&=\de(c'\textrm{-lobe},SE).
\label{efeec}
\end{align}
%
Since both the difference between the focal length and the sum of the inner lobes and the distance between an inner lobe and the side facing it through the other two bowties are invariant under bowtie squeezing, this proves our second claim.

Using the above two claims, we obtain from \eqref{efd} and Theorems \ref{tea} and \ref{teb} that
\begin{equation}
\frac{\M(R)}{\M(\thickbar{R})}=  
\frac
{\w^{(I)}\dfrac{\ka_A^{(I)}\ka_B^{(I)}\ka_C^{(I)}}{\ka_{BC}^{(I)}\ka_{AC}^{(I)}\ka_{AB}^{(I)}}}
{\w^{(I')}\dfrac{\ka_{A_0}^{(I')}\ka_{B_0}^{(I')}\ka_{C_0}^{(I')}}{\ka_{B_0C_0}^{(I')}\ka_{A_0C_0}^{(I')}\ka_{A_0B_0}^{(I')}}}
\,
\frac
{\w^{(S)}\dfrac{\ka_A^{(S)}\ka_B^{(S)}\ka_C^{(S)}}{\ka_{BC}^{(S)}\ka_{AC}^{(S)}\ka_{AB}^{(S)}}}
{\w^{(S')}\dfrac{\ka_{A_0}^{(S')}\ka_{B_0}^{(S')}\ka_{C_0}^{(S')}}{\ka_{B_0C_0}^{(S')}\ka_{A_0C_0}^{(S')}\ka_{A_0B_0}^{(S')}}}
\label{eff}
\end{equation}
By the definition \eqref{ebb} of the weight $\w$, we have from the picture on the left in Figure \ref{ffa}
\begin{align}
&  
\w^{(I)}\w^{(S)}=\frac{\h(f)^4\h(a)\h(0)\h(0)\h(f)\h(0)\h(0)}{\h(f+a)\h(f+0)\h(f+0)\h(f-f)\h(f-0)\h(f-0)}
\\[10pt]
&\ \ \ \ \ \ \ \ \ \ \ \ \ \ \ \ \ \ \ \ \ \ \ \ \ \ \ 
\times
\frac{\h(f)^4\h(0)\h(0)\h(0)\h(a')\h(b')\h(c')}{\h(f+0)\h(f+0)\h(f+0)\h(f-a')\h(f-b')\h(f-c')}
\\[10pt]
&\ \ \ \ \ \ \ \ \ \ \ \ \ \ \ \ \ \ \ \ \ \ \ \ \ \ \ 
=
\frac{\h(f)^2\h(a)\h(a')\h(b')\h(c')}{\h(f+a)\h(f-a')\h(f-b')\h(f-c')}.
\label{efg}
\end{align}
Similarly, we obtain
\begin{align}
\frac{\ka_A^{(I)}}{\ka_{BC}^{(I)}}
\,
\frac{\ka_A^{(S)}}{\ka_{BC}^{(S)}}
&=
\frac{\h(\de(A,N))\h(f)}{\h(\de(BC,N))\h(0)}
\,
\frac{\h(0)\h(\de(A,S))}{\h(f)\h(\de(BC,S)}
\label{efia}
\\[10pt]
\frac{\ka_B^{(I)}}{\ka_{AC}^{(I)}}
\,
\frac{\ka_B^{(S)}}{\ka_{AC}^{(S)}}
&=
\frac{\h(\de(B,NE))\h(b)}{\h(\de(AC,NE))\h(f+b)}
\,
\frac{\h(f)\h(\de(B,SW))}{\h(0)\h(\de(AC,SW)}
\label{efia}
\\[10pt]
\frac{\ka_C^{(I)}}{\ka_{AB}^{(I)}}
\,
\frac{\ka_C^{(S)}}{\ka_{AB}^{(S)}}
&=
\frac{\h(\de(C,NW))\h(c)}{\h(\de(AB,NW))\h(f+c)}
\,
\frac{\h(f)\h(\de(C,SE))}{\h(0)\h(\de(AB,SE)}.
\label{efia}
\end{align}
Thus we obtain
\begin{align}
&
{\w^{(I)}\dfrac{\ka_A^{(I)}\ka_B^{(I)}\ka_C^{(I)}}{\ka_{BC}^{(I)}\ka_{AC}^{(I)}\ka_{AB}^{(I)}}}
{\w^{(S)}\dfrac{\ka_A^{(S)}\ka_B^{(S)}\ka_C^{(S)}}{\ka_{BC}^{(S)}\ka_{AC}^{(S)}\ka_{AB}^{(S)}}}  
\nonumber
\\[10pt]
&\ \ 
=
\frac{\h(f)^4\h(a)\h(b)\h(c)\h(a')\h(b')\h(c')}{\h(f+a)\h(f+b)\h(f+c)\h(f-a)\h(f-b)\h(f-c)}
\nonumber
\\[10pt]
&\ \ \ \
\times
\frac{\h(\de(A,N))\h(\de(A,S))}{\h(\de(BC,N))\h(\de(BC,S))}
\frac{\h(\de(B,NE))\h(\de(B,SW))}{\h(\de(AC,NE))\h(\de(AC,SW))}
\frac{\h(\de(C,NW))\h(\de(C,SE))}{\h(\de(AB,NW))\h(\de(AB,SE))}
\nonumber
\\[10pt]
&\ \ \ \ 
=
\w^{(R)}\dfrac{\ka_A^{(R)}\ka_B^{(R)}\ka_C^{(R)}}{\ka_{BC}^{(R)}\ka_{AC}^{(R)}\ka_{AB}^{(R)}}.
\label{efj}
\end{align}  
This shows that the product of the numerators on the right hand side of \eqref{eff} is equal to the numerator on the right hand side of \eqref{efa}. The very same argument implies that also the product of the denominators on the right hand side of \eqref{eff} is equal to the denominator on the right hand side of \eqref{efa}. Therefore \eqref{efa} follows from \eqref{eff}.



\bigskip

\section{Proof of Theorem \ref{tba}}

\bigskip

%

\begin{figure}[h]
  \centerline{
\hfill
{\includegraphics[width=0.41\textwidth]{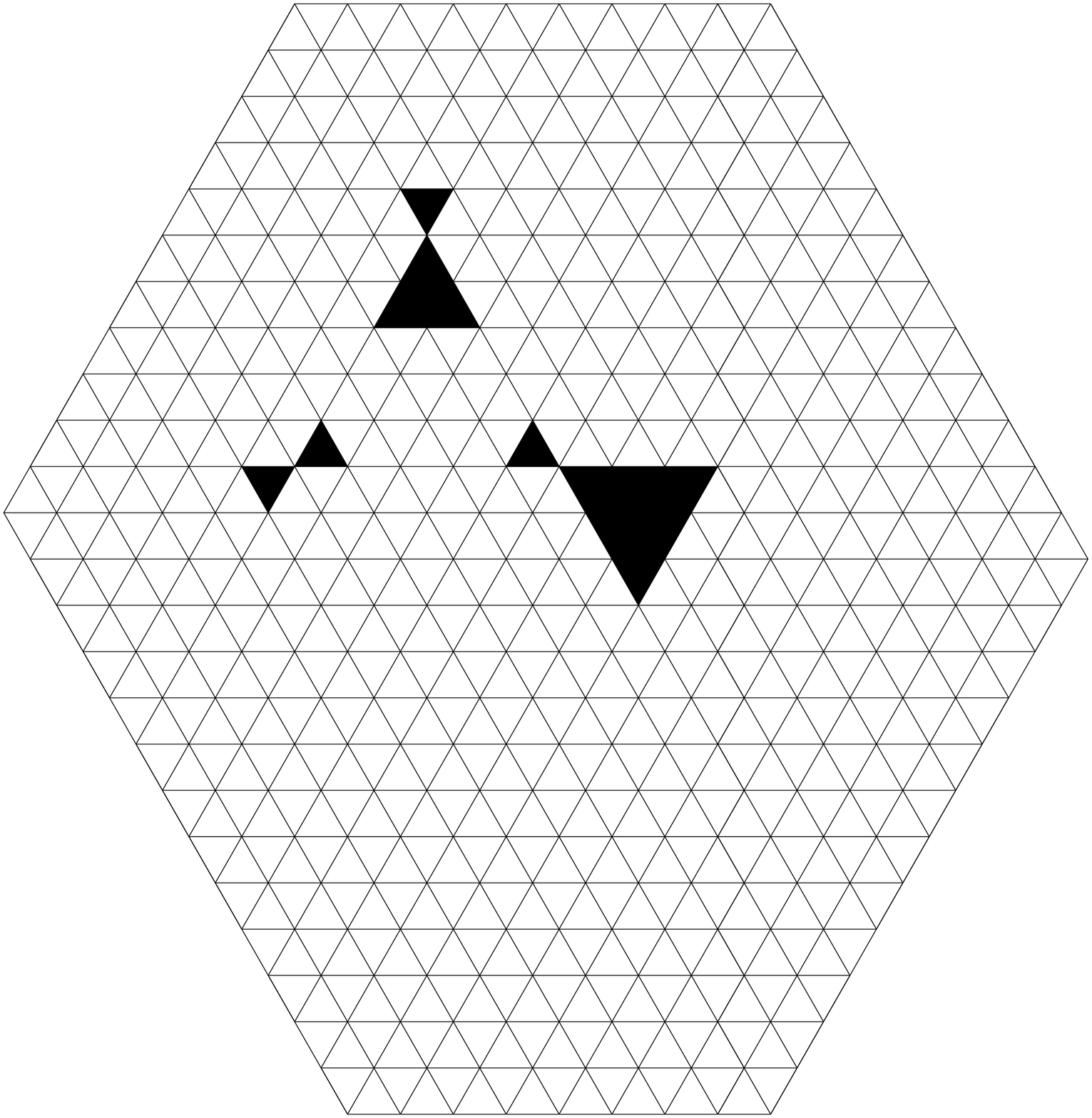}}
\hfill
{\includegraphics[width=0.41\textwidth]{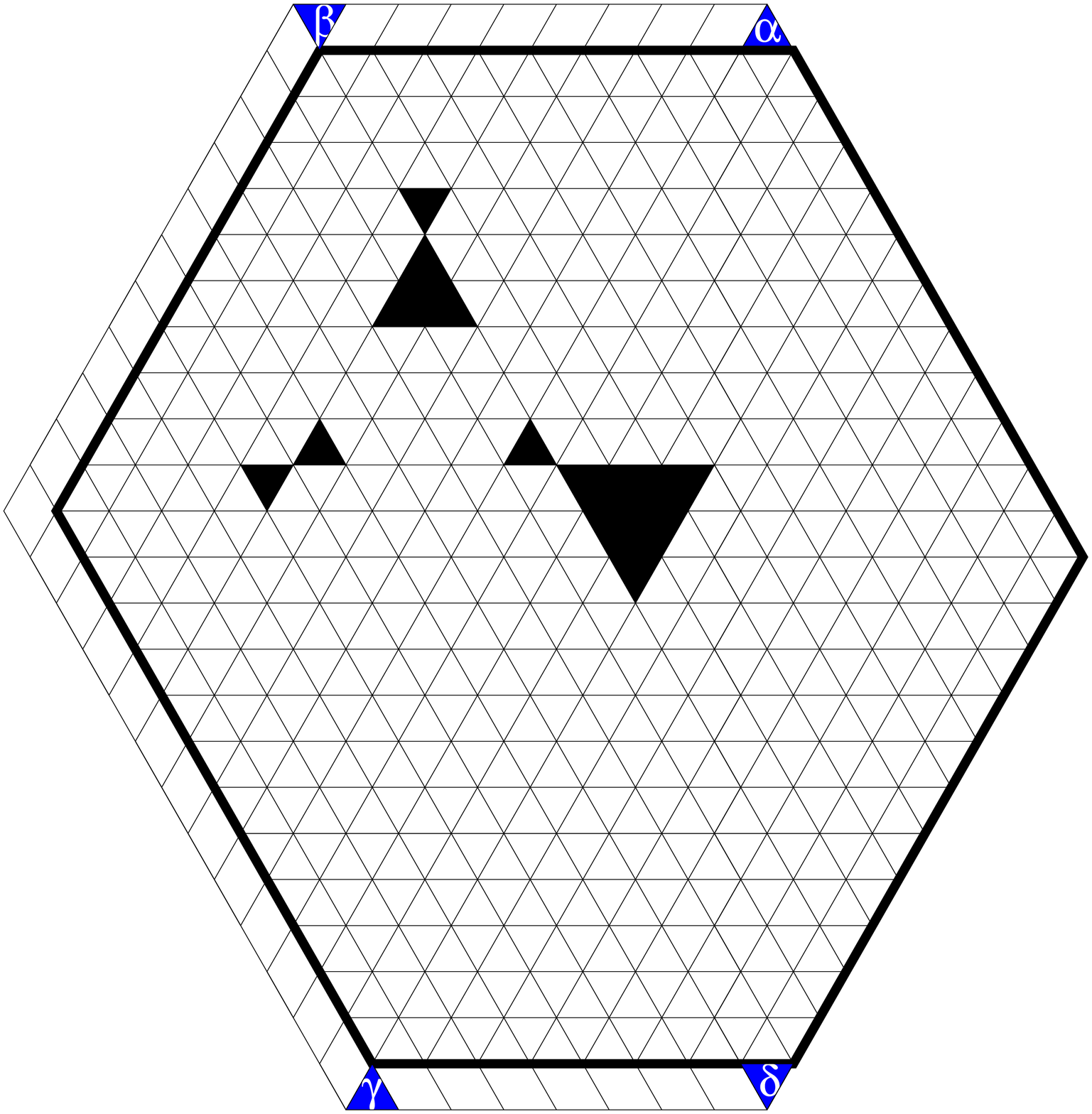}}
\hfill
}
\vskip0.15in
  \centerline{
\hfill
{\includegraphics[width=0.41\textwidth]{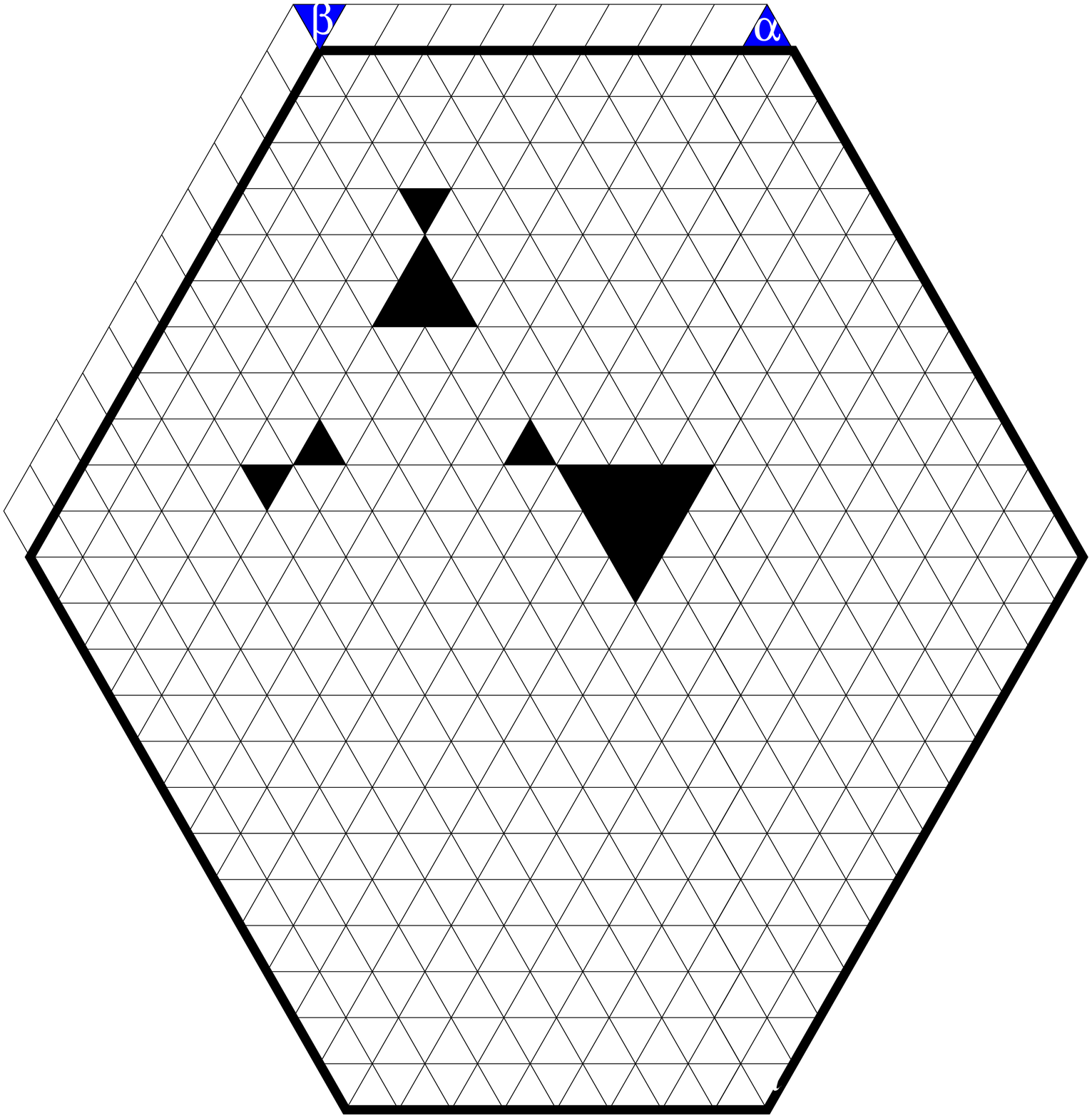}}
\hfill
{\includegraphics[width=0.41\textwidth]{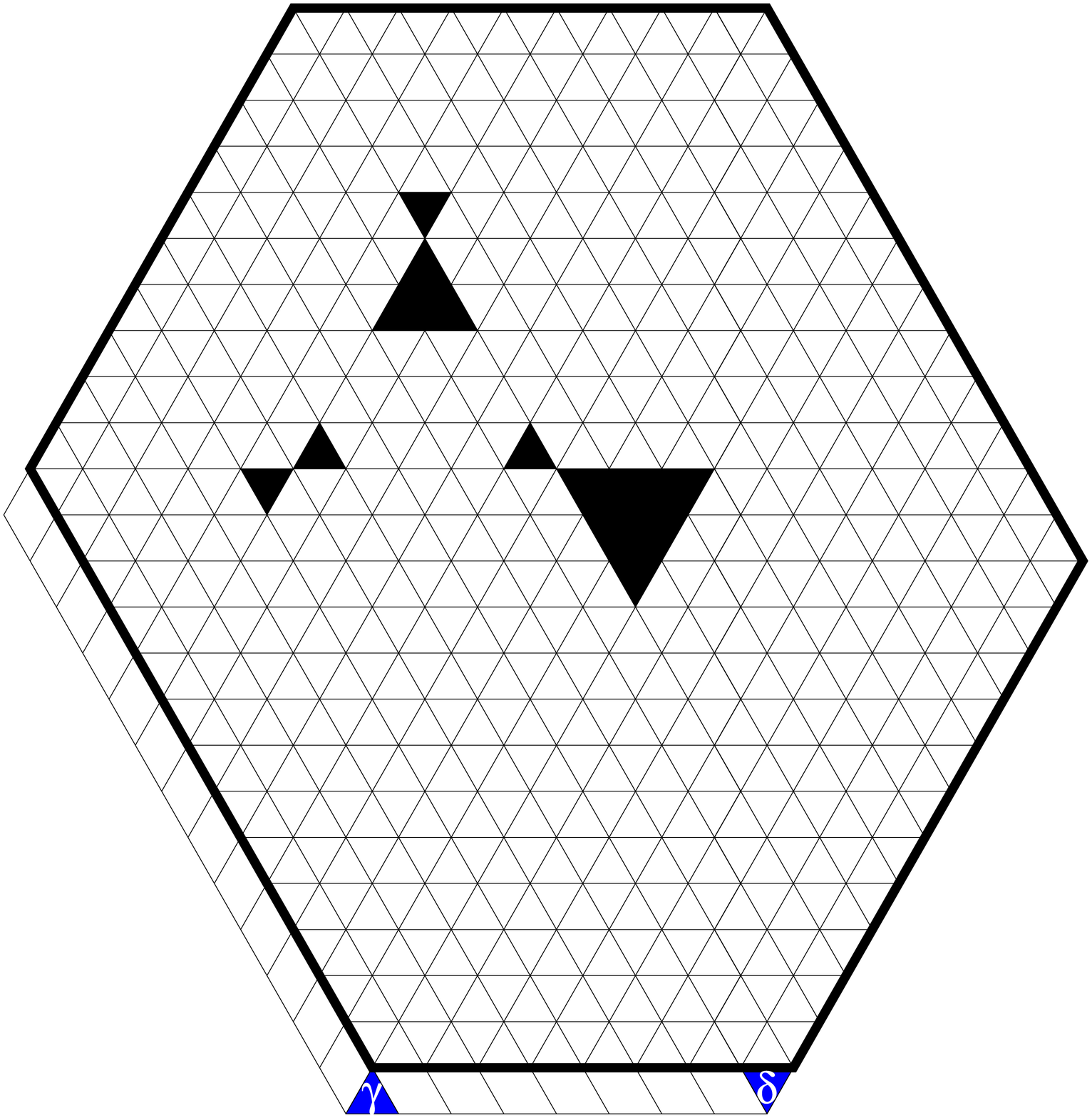}}
\hfill
}
\vskip0.15in
\centerline{
\hfill
{\includegraphics[width=0.41\textwidth]{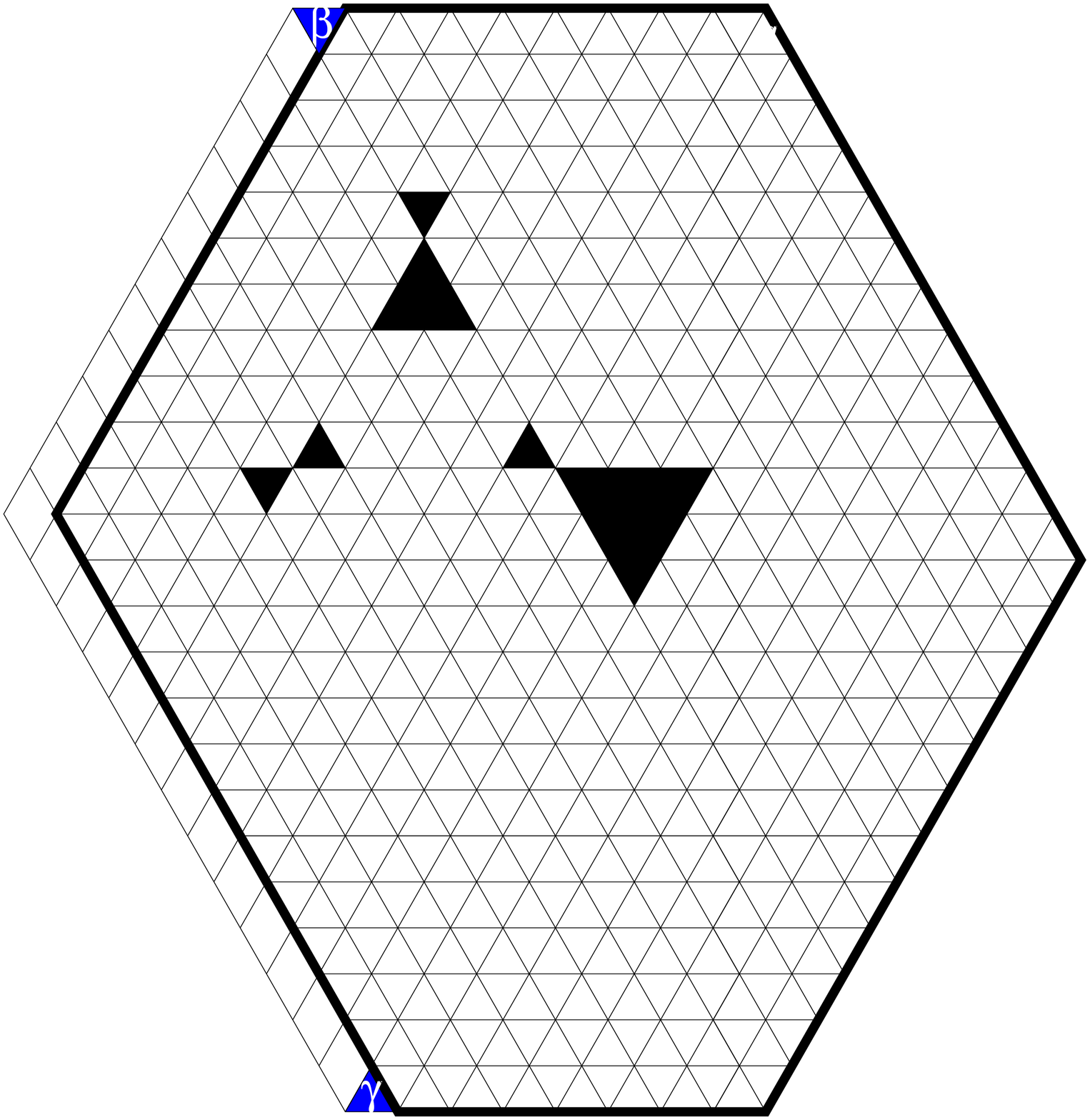}}
\hfill
{\includegraphics[width=0.41\textwidth]{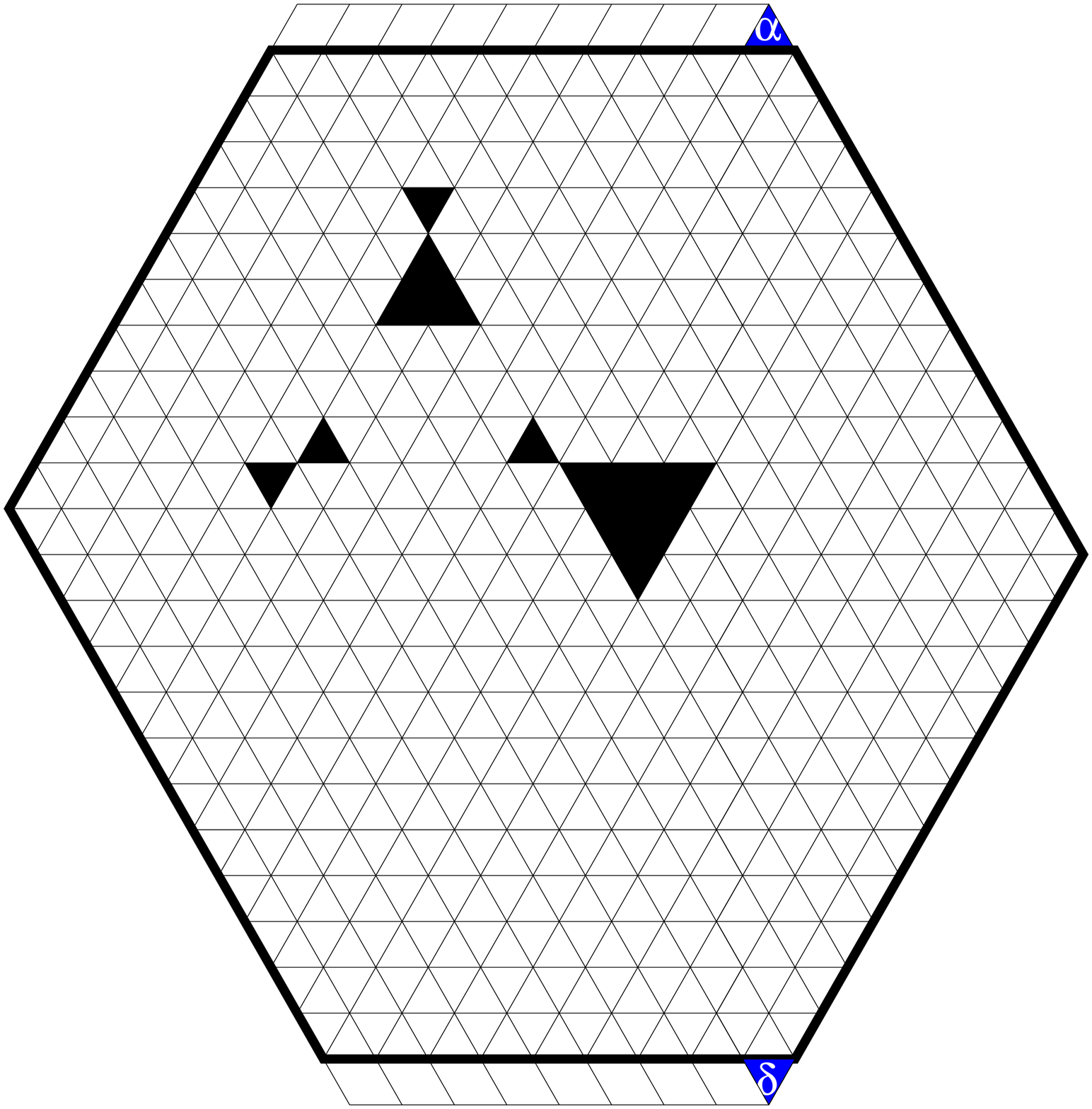}}
\hfill
}
  \caption{\label{fca} Obtaining the recurrence for the let hand side of \eqref{ecd}.}
\vskip-0.05in
\end{figure}

We claim that in order to prove the theorem, it suffices to show that for any triad hexagon $R$ we have

\begin{equation}
\label{ecb}
\frac{\M(R)}{\M(\thickbar{R})}=\dfrac
{\w^{(R)}\dfrac{\ka_A^{(R)}\ka_B^{(R)}\ka_C^{(R)}}{\ka_{BC}^{(R)}\ka_{AC}^{(R)}\ka_{AB}^{(R)}}}
{\w^{(\thickbar{R})}\dfrac{\ka_{A_0}^{(\thickbar{R})}\ka_{B_0}^{(\thickbar{R})}\ka_{C_0}^{(\thickbar{R})}}{\ka_{B_0C_0}^{(\thickbar{R})}\ka_{A_0C_0}^{(\thickbar{R})}\ka_{A_0B_0}^{(\thickbar{R})}}},
\end{equation}
where $\thickbar{R}$ is obtained from $R$ by squeezing out completely all three bowties (see Section 2 for the definition), and $A_0$, $B_0$ and $C_0$ are its focal points.

Indeed, assume \eqref{ecb} holds for any triad hexagon. Then it holds in particular for $R$ replaced by $Q$. Crucially, since $Q$ is obtained from $R$ by a sequence of bowtie squeezing operations (see Section 2 for their definition), the region obtained from $Q$ by completely squeezing out its bowties is also $\thickbar{R}$ (i.e. $\thickbar{Q}=\thickbar{R}$).

Therefore, we obtain

\begin{equation}
\label{ecc}
\frac{\M(Q)}{\M(\thickbar{R})}=\dfrac
{\w^{(Q)}\dfrac{\ka_{A_1}^{(Q)}\ka_{B_1}^{(Q)}\ka_{C_1}^{(Q)}}{\ka_{B_1C_1}^{(Q)}\ka_{A_1C_1}^{(Q)}\ka_{A_1B_1}^{(Q)}}}
{\w^{(\thickbar{R})}\dfrac{\ka_{A_0}^{(\thickbar{R})}\ka_{B_0}^{(\thickbar{R})}\ka_{C_0}^{(\thickbar{R})}}{\ka_{B_0C_0}^{(\thickbar{R})}\ka_{A_0C_0}^{(\thickbar{R})}\ka_{A_0B_0}^{(\thickbar{R})}}}.
\end{equation}

Combining equations \eqref{ecb} and \eqref{ecc} yields \eqref{ebe}, proving our claim.

We prove the equivalent form
\begin{equation}
\label{ecd}
\M(R)=\M(\thickbar{R})\dfrac
{\w^{(R)}\dfrac{\ka_A^{(R)}\ka_B^{(R)}\ka_C^{(R)}}{\ka_{BC}^{(R)}\ka_{AC}^{(R)}\ka_{AB}^{(R)}}}
{\w^{(\thickbar{R})}\dfrac{\ka_{A_0}^{(\thickbar{R})}\ka_{B_0}^{(\thickbar{R})}\ka_{C_0}^{(\thickbar{R})}}{\ka_{B_0C_0}^{(\thickbar{R})}\ka_{A_0C_0}^{(\thickbar{R})}\ka_{A_0B_0}^{(\thickbar{R})}}}
\end{equation}
of \eqref{ecb} by arguments that parallel those in Sections 3 and 4 of \cite{fv}. Namely, we prove \eqref{ecd} by showing that both sides satisfy the same recurrence. The formal proof is set up as a proof by induction.


The recurrence satisfied by the left hand side of \eqref{ecd} is obtained by applying Kuo condensation as follows.

%



Let $G$ be the planar dual graph of the region $R_{x,y,z}^{A,B,C}(a,b,c,a',b',c')$, choose the vertices $\alpha$, $\beta$, $\gamma$ and $\delta$ of $G$ to be the duals of the unit triangles indicated on the top right in Figure \ref{fca}, and apply Theorem \ref{tca}. Then all six graphs in the equation resulting from \eqref{eca} are planar duals of regions that become triad hexagons once all forced lozenges are removed from them (this is illustrated in Figure \ref{fca}). 

The change in the $x$-, $y$- and $z$-parameters of the resulting triad hexagons is easily read off from Figure \ref{fca}. As the lobe sizes $a$, $b$, $c$, $a'$, $b'$, $c'$ and the geometrical position of the focal points $A$, $B$ and $C$ remain unchanged for all resulting regions, one sees that \eqref{eca} becomes
\begin{align}
&
\M(R_{x,y,z}^{A,B,C}(a,b,c,a',b',c'))\M(R_{x,y-1,z-1}^{A,B,C}(a,b,c,a',b',c'))
=
\nonumber
\\[5pt]
&\ \ \ \ \ \ \ \ \ \ \ \ \ \ \ \ \ \ \ \ \ \ \ \ \ \ \ \ \ \ \ \ 
\M(R_{x,y-1,z}^{A,B,C}(a,b,c,a',b',c'))\M(R_{x,y,z-1}^{A,B,C}(a,b,c,a',b',c'))
\nonumber
\\[5pt]
&\ \ \ \ \ \ \ \ \ \ \ \ \ \ \ \ \ \ \ \ \ \ \ \ \ \ \ \ \ 
+\M(R_{x-1,y,z}^{A,B,C}(a,b,c,a',b',c'))\M(R_{x+1,y-1,z-1}^{A,B,C}(a,b,c,a',b',c')).
\label{ece}
\end{align}  
We use this recurrence to prove \eqref{ecd} by induction on $x+y+z$. As $x,y,z\geq1$ is a necessary condition in order for all the regions in \eqref{ece} to be defined, the cases when $x=0$, $y=0$ or $z=0$ will be base cases of our induction.

In fact, in order for \eqref{ece} to hold, an additional condition needs to hold: The pattern of forced lozenges needs to be as indicated in Figure \ref{fca}. Assuming $x,y,z\geq1$ (so that there is room for the indicated unit triangles along the corresponding edges), this happens if and only if $(i)$ the distance between the $a$-lobe and the norhern side of the hexagon is positive, $(ii)$ the distance between the $b$-lobe and the southwestern side of the hexagon is positive, $(iii)$ the NW-depth of the triad is positive, and $(iv)$ the S-depth of the triad is positive.

Our base cases will be the following: (1) $x=0$, $y=0$ or $z=0$; (2) at least two of the bowties touch the hexagon side they are facing; and (3) at least one of the three triad depths is equal to zero.

Each of these cases is handled separately. If none of them occurs, then $x,y,z\geq1$, all three triad depths are positive, and at most one bowtie touches the hexagon side it is facing. By symmetry, we may assume that the $(a,a')$- and $(b,b')$-bowtie do not touch the hexagon side they are facing, and therefore equation \eqref{ece} holds.

The detais of the case  $x=0$ were presented in the previous section. The cases when $y=0$  or $z=0$ follow by symmetry.




Base case (2) is the object of Section 6. Finally, base case (3) follows by Section 5.

For the induction step, let $x,y,z\geq1$ and assume that \eqref{ecd} holds for all triad hexagons with the sum of their $x$-, $y$- and $z$-parameters strictly less than $x+y+z$. Consider the triad hexagon $R_{x,y,z}^{A,B,C}(a,b,c,a',b',c')$, and assume that conditions $(i)--(iv)$ above hold. We need to deduce that~\eqref{ecd} holds also for $R_{x,y,z}^{A,B,C}(a,b,c,a',b',c')$.

As $x,y,z\geq1$ and conditions $(i)$--$(iv)$ hold, we can apply \eqref{ecd}. Since for the last five triad hexagons in \eqref{ecd} the sum of the $x$-, $y$- and $z$-parameters is strictly less than $x+y+z$, by the induction hypothesis, their number of lozenge tilings can be expressed as indicated by formula \eqref{ecd}. Do this for each of these five regions in \eqref{ecd}. This yields a certain expression for $\M(R_{x,y,z}^{A,B,C}(a,b,c,a',b',c'))$. To complete the proof, we need to verify that this expression agrees with the right hand side of~\eqref{ecd}. This amounts to checking that the right hand side of~\eqref{ecd} satisfies recurrence \eqref{ece}. We carry out this verification in Section 9. This completes the proof.~\epf

\section{Verifying that the right hand side of \eqref{ecd} satisfies recurrence \eqref{ece}}

Figure \ref{fca} illustrates the six regions in equation \eqref{ece}. On the top left is the triad hexagon $R=R_{x,y,z}^{A,B,C}(a,b,c,a',b',c')$. For ease of reference, denote the top right, center left, center right, bottom left and bottom right regions in Figure \ref{fca} by $R_2$, $R_3$, $R_4$, $R_5$ and $R_6$, respectively.

Then in order to verify that the right hand side of \eqref{ecd} satisfies recurrence \eqref{ece}, we need to prove that

\begin{align}
  &
\M(\thickbar{R})\,\dfrac
{\w^{(R)}\dfrac{\ka_A^{(R)}\ka_B^{(R)}\ka_C^{(R)}}{\ka_{BC}^{(R)}\ka_{AC}^{(R)}\ka_{AB}^{(R)}}}
{\w^{(\thickbar{R})}\dfrac{\ka_{A_0}^{(\thickbar{R})}\ka_{B_0}^{(\thickbar{R})}\ka_{C_0}^{(\thickbar{R})}}{\ka_{B_0C_0}^{(\thickbar{R})}\ka_{A_0C_0}^{(\thickbar{R})}\ka_{A_0B_0}^{(\thickbar{R})}}}
\ 
\M(\thickbar{R}_2)\,\dfrac
{\w^{(R_2)}\dfrac{\ka_A^{(R_2)}\ka_B^{(R_2)}\ka_C^{(R_2)}}{\ka_{BC}^{(R_2)}\ka_{AC}^{(R_2)}\ka_{AB}^{(R_2)}}}
{\w^{(\thickbar{R}_2)}\dfrac{\ka_{A_0}^{(\thickbar{R}_2)}\ka_{B_0}^{(\thickbar{R}_2)}\ka_{C_0}^{(\thickbar{R}_2)}}{\ka_{B_0C_0}^{(\thickbar{R}_2)}\ka_{A_0C_0}^{(\thickbar{R}_2)}\ka_{A_0B_0}^{(\thickbar{R}_2)}}}
\nonumber
\\[10pt]
&\ \ \ \ \ \ \ \ \ \ \ \ \ \ \ \ \ \ \ \ \ \ \ \ 
=
\M(\thickbar{R}_3)\,\dfrac
{\w^{(R_3)}\dfrac{\ka_A^{(R_3)}\ka_B^{(R_3)}\ka_C^{(R_3)}}{\ka_{BC}^{(R_3)}\ka_{AC}^{(R_3)}\ka_{AB}^{(R_3)}}}
{\w^{(\thickbar{R}_3)}\dfrac{\ka_{A_0}^{(\thickbar{R}_3)}\ka_{B_0}^{(\thickbar{R}_3)}\ka_{C_0}^{(\thickbar{R}_3)}}{\ka_{B_0C_0}^{(\thickbar{R}_3)}\ka_{A_0C_0}^{(\thickbar{R}_3)}\ka_{A_0B_0}^{(\thickbar{R}_3)}}}
\ 
\M(\thickbar{R}_4)\,\dfrac
{\w^{(R_2)}\dfrac{\ka_A^{(R_4)}\ka_B^{(R_4)}\ka_C^{(R_4)}}{\ka_{BC}^{(R_4)}\ka_{AC}^{(R_4)}\ka_{AB}^{(R_4)}}}
{\w^{(\thickbar{R}_4)}\dfrac{\ka_{A_0}^{(\thickbar{R}_4)}\ka_{B_0}^{(\thickbar{R}_4)}\ka_{C_0}^{(\thickbar{R}_4)}}{\ka_{B_0C_0}^{(\thickbar{R}_4)}\ka_{A_0C_0}^{(\thickbar{R}_4)}\ka_{A_0B_0}^{(\thickbar{R}_4)}}}
\nonumber
\\[10pt]
&\ \ \ \ \ \ \ \ \ \ \ \ \ \ \ \ \ \ \ \ \ \ \ \ 
+
\M(\thickbar{R}_5)\,\dfrac
{\w^{(R_5)}\dfrac{\ka_A^{(R_5)}\ka_B^{(R_5)}\ka_C^{(R_5)}}{\ka_{BC}^{(R_5)}\ka_{AC}^{(R_5)}\ka_{AB}^{(R_5)}}}
{\w^{(\thickbar{R}_5)}\dfrac{\ka_{A_0}^{(\thickbar{R_5})}\ka_{B_0}^{(\thickbar{R_5})}\ka_{C_0}^{(\thickbar{R_5})}}{\ka_{B_0C_0}^{(\thickbar{R_5})}\ka_{A_0C_0}^{(\thickbar{R_5})}\ka_{A_0B_0}^{(\thickbar{R_5})}}}
\ 
\M(\thickbar{R}_6)\,\dfrac
{\w^{(R_6)}\dfrac{\ka_A^{(R_6)}\ka_B^{(R_6)}\ka_C^{(R_6)}}{\ka_{BC}^{(R_6)}\ka_{AC}^{(R_6)}\ka_{AB}^{(R_6)}}}
{\w^{(\thickbar{R}_6)}\dfrac{\ka_{A_0}^{(\thickbar{R}_6)}\ka_{B_0}^{(\thickbar{R}_6)}\ka_{C_0}^{(\thickbar{R}_6)}}{\ka_{B_0C_0}^{(\thickbar{R}_6)}\ka_{A_0C_0}^{(\thickbar{R}_6)}\ka_{A_0B_0}^{(\thickbar{R}_6)}}},
\label{eda}
\end{align}
where for $i=2,\dotsc,6$, $\thickbar{R}_i$ is the region obtained from $R_i$ by completely squeezing out each of its three bowties. Note that the triad of bowties (and in particular the focal points) in the regions $R_i$ are the same as in $R$, and a similar statement relates the regions $\thickbar{R}_i$ and $\thickbar{R}$, for $i=2,\dotsc,6$.

It turns out that the products of the two fractions in each of the three terms in \eqref{eda} have the same value. Indeed, all twelve weights $\w$ are clearly equal, as the twelve involved regions share the same triad of bowties, and $\w$ only depends on the geometry of this triad (see \eqref{ebb}). Furthermore, one readily sees from equation \eqref{ebca} and Figure \ref{fca} that
\begin{equation}
  \ka_A^{(R)} \ka_A^{(R_2)}
  =
   \ka_A^{(R_3)} \ka_A^{(R_4)}
  =
   \ka_A^{(R_5)} \ka_A^{(R_6)}.
\label{edb}
\end{equation}
This is because in each of the six regions, the distance from the focal point $A$ to the northern boundary is equal to either $d$ or $d+1$, and its distance to the southern boundary is either $e$ or $e+1$, for some non-negative integers $d$ and $e$. Then by \eqref{ebca}, each of the three quentities in~\eqref{edb} is equal to $de(d+1)(e+1)$.

In a similar way, one sees that
\begin{equation}
  \ka_B^{(R)} \ka_B^{(R_2)}
  =
   \ka_B^{(R_3)} \ka_B^{(R_4)}
  =
   \ka_B^{(R_5)} \ka_B^{(R_6)},
\label{edc}
\end{equation}
\vskip0.05in
\begin{equation}
  \ka_C^{(R)} \ka_C^{(R_2)}
  =
   \ka_C^{(R_3)} \ka_C^{(R_4)}
  =
   \ka_C^{(R_5)} \ka_C^{(R_6)},
\label{edd}
\end{equation}
\vskip0.03in
and also that
\begin{equation}
  \ka_{BC}^{(R)} \ka_{BC}^{(R_2)}
  =
   \ka_{BC}^{(R_3)} \ka_{BC}^{(R_4)}
  =
   \ka_{BC}^{(R_5)} \ka_{BC}^{(R_6)},
\label{ede}
\end{equation}
\vskip0.05in
\begin{equation}
  \ka_{AC}^{(R)} \ka_{AC}^{(R_2)}
  =
   \ka_{AC}^{(R_3)} \ka_{AC}^{(R_4)}
  =
   \ka_{AC}^{(R_5)} \ka_{AC}^{(R_6)},
\label{edf}
\end{equation}
\vskip0.03in
and
\begin{equation}
  \ka_{AB}^{(R)} \ka_{AB}^{(R_2)}
  =
   \ka_{AB}^{(R_3)} \ka_{AB}^{(R_4)}
  =
   \ka_{AB}^{(R_5)} \ka_{AB}^{(R_6)}.
\label{edg}
\end{equation}
Therefore, the products of the couples $\ka$ at the numerators in \eqref{eda} are equal across the three terms. Since the barred regions at the denominators are special cases of regions at the numerators, the same conclusion holds also for the numerators in \eqref{eda}.

Thus, \eqref{eda} simplifies to
\begin{equation}
\M(\thickbar{R}) \M(\thickbar{R}_2) = 
\M(\thickbar{R}_3) \M(\thickbar{R}_4)+
\M(\thickbar{R}_5) \M(\thickbar{R}_6).
\end{equation}
However, this holds by Kuo's condensation identity \eqref{eca}. 

%
%
%
%
%
%
%

\section{Proof of Lemma \ref{tileability}}

\begin{figure}[h]
  \centerline{
\hfill
{\includegraphics[width=0.90\textwidth]{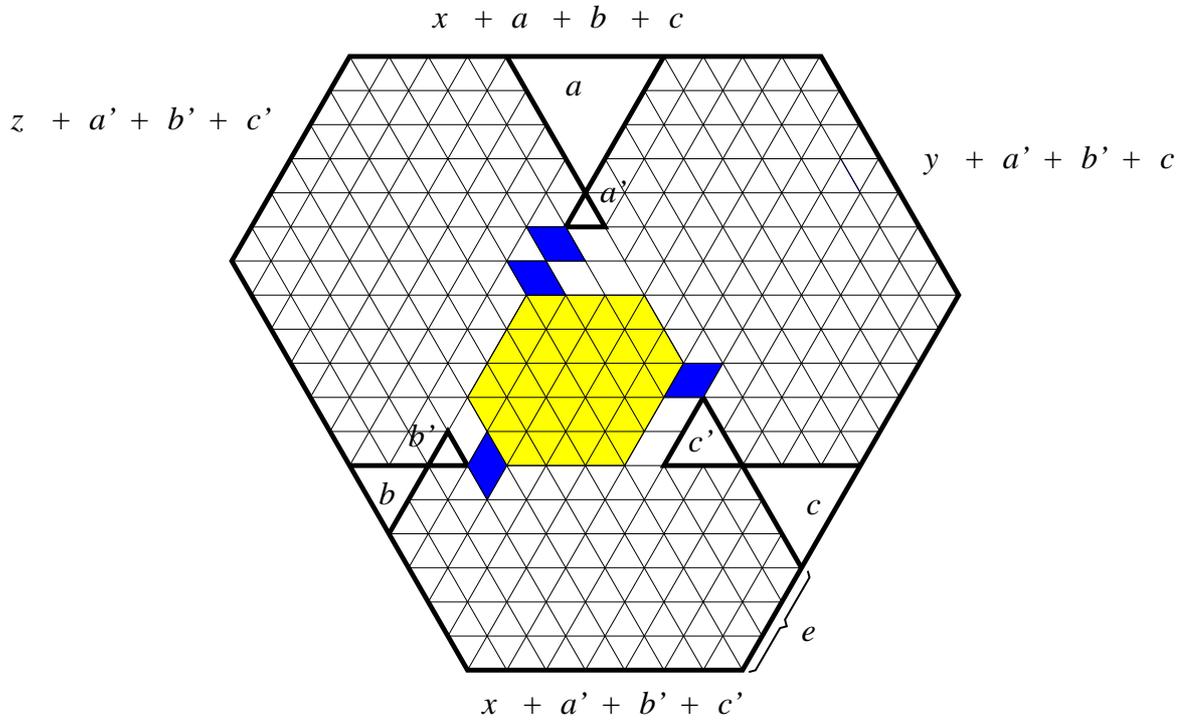}}
\hfill
}
  \caption{\label{fka} The case when all three bowties touch the hexagon side they face.}
\end{figure}

\begin{figure}[h]
  \centerline{
\hfill
{\includegraphics[width=0.60\textwidth]{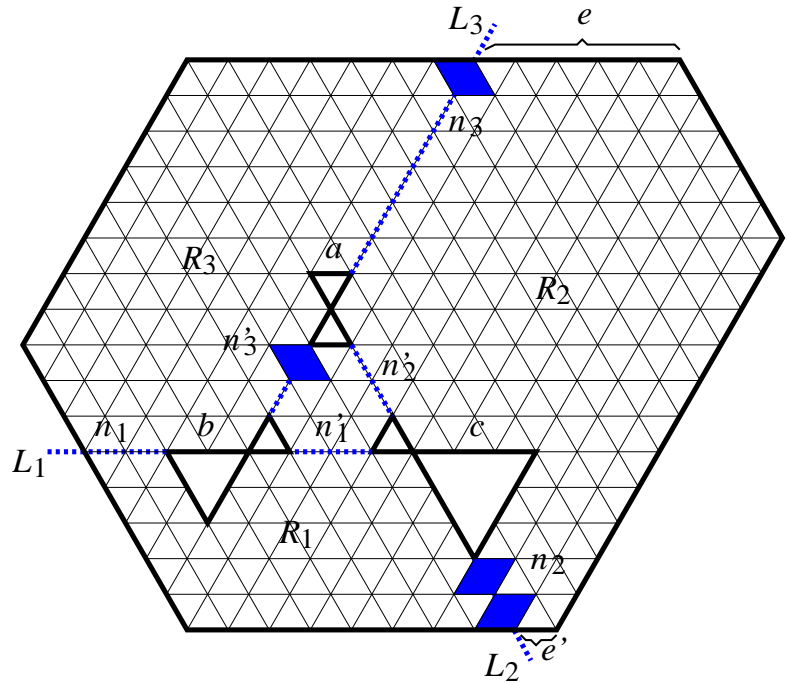}}
\hfill
}
  \caption{\label{fkb} The case when the southern depth is 0.}
\end{figure}

Part (b) follows directly from the definition of bowtie squeezing. Part (c) is an immediate consequence of parts (a) and (b).

To prove part (a), consider a triad hexagon $R=R_{x,y,z}^{A,B,C}(a,b,c,a',b',c')$ whose S-, NE- and NW-depths are non-negative. Note that if we translate eastward the two sides of $R$ making up its western boundary so that the translated sides still meet the top and bottom sides, but do not cross the bowtie triad, and we denote by $Q$ the smaller region obtained this way (which is also a triad hexagon), then $R$ is tileable if $Q$ is (indeed, simply tile the portion of $R$ that is in excess of $Q$ by its unique tiling). Analogous statements hold for the union of any two consecutive sides of the outer boundary of $R$.

Let $R^{(i)}$ be the region obtained from $R$ by translating its western boundary eastward, until it either touches the bowtie triad, or the NW-depth becomes zero, or the $x$-parameter becomes zero. Let $R^{(ii)}$ be the region obtained from $R^{(i)}$ by translating its eastern boundary westward, until it either touches the bowtie triad, or the NE-depth becomes zero, or the $x$-parameter becomes zero. Let $R^{(iii)}$ be the region obtained from $R^{(ii)}$ by translating its northeastern boundary sothwestward, until it either touches the bowtie triad, or the NE-depth becomes zero, or the $z$-parameter becomes zero. Finally, let $R^{(iv)}$ be the region obtained from $R^{(iii)}$ by translating its southwestern boundary northeastward, until it either touches the bowtie triad, or the S-depth becomes zero, or the $z$-parameter becomes zero. 

By the observation in the second paragraph of this proof, it suffices to show that $R^{(iv)}$ is tileable. By construction, $R^{(iv)}$ is a triad hexagon in which at least one of the following holds: (1) the three bowties touch the hexagon sides they are facing, or (2) one of the S-, NE- or NW-depths is equal to zero, or (3) one of $x$ and $z$ is zero. The latter case follows by Section 7 and case (1), as both the magnet bar and the snowman regions are special cases of the instance when each of the three bowties touches the facing hexagon side. We handle cases (1) and (2) below.

If the three bowties touch the hexagon sides they are facing, the situation is as pictured in Figure \ref{fka}. Denoting the focal distance by $f+a'+b'+c'$, we claim that $x\leq f$. Indeed, consider the hexagonal subregion below the bottom focal line. Since in any lattice hexagon the length difference of opposite sides is the same, we have that $(f+a'+b'+c')-(x+a'+b'+c')=e-b$, where $e$ is the length of the portion of the southeastern side below the $c$-lobe. Since the S-depth is non-negative, we have $e\geq b$, and we obtain $x\leq f$. By symmetry, we also get $y,z\leq f$.

Consider $f-x$ consecutive vertical lozenges straddling the bottom focal segment, next to the $b'$-lobe (the distance between the $b'$- and $c'$-lobes is $f+a'\geq f-x$, so there is room for these vertical lozenges). Consider analogous runs of $f-y$ and $f-z$ consecutive lozenges on the northeastern and northwestern focal segments, respectively, as indicated in Figure \ref{fka}. Remove the indicated $3f-x-y-z$ lozenges, and regard the leftover region as being made up of four parts: the part inside the focal triangle, and the three hexagonal regions outside it.

We claim that upon removing forced lozenges, the hexagonal region left over inside the focal triangle has equal opposite sides, and is therefore tileable. Indeed, the length of the southwestern side of this hexagon is $f-x+b'$, while the length of the side opposite it is $f+b'-(f-y)-(f-z)$, so the two are equal precisely if $2f=x+y+z$. However, this follows by noticing that the sum of the lengths of the southwestern and northwestern sides of the outer hexagon is equal to $\de(A,N)+\de(A,S)$.

Furthermore, the hexagonal region under the bottom focal line, with the dents created by the considered $f-x$ vertical lozenges, is balanced\footnote{ I.e., has the same number of up- and down-pointing unit triangles}, and this readily implies that it is tileable. By symmetry, so are the two dented hexagonal regions above the other two focal lines. It follows that the region $R$ itself is tileable.

In the remaining case, we may assume without loss of generality that the S-depth of $R$ is zero. There are eight cases to be distinguished here, depending on which side of the dotted lines $L_1$, $L_2$ and $L_3$ in Figure \ref{fkb} the western, southeastern and northeastern vertices of the hexagon are, respectively. The arguments in all cases are analogous. We discuss therefore only the case pictured in Figure~\ref{fkb}.

The three shown dotted lines $L_1$, $L_2$ and $L_3$ divide $R$ into four subregions: $R_1$, $R_2$, $R_3$ and the focal triangle. We show that six non-negative integers $n_1$, $n_2$, $n_3$, $n_1'$, $n_2'$, $n_3'$ exist so that if one takes a run of $n_1$ consecutive vertical lozenges straddling $L_1$ immediately to the right of the southwestern side, $n_1'$ consecutive vertical lozenges straddling $L_1$ immediately to the right of the $b'$-lobe, and four analogous runs of consecutive lozenges straddling $L_2$ and $L_3$, then the portions $R_1'$, $R_2'$, $R_3'$ and $T'$ of $R_1$, $R_2$, $R_3$ and the focal triangle left over after removing these $n_1+n_2+n_3+n_1'+n_2'+n_3'$ lozenges are each tileable. This will imply that $R$ itself is tileable.

It is easy to see that $R_1'$, $R_2'$, $R_3'$ and $T'$ are tileable if and only if they are balanced.

Choose $n_1=n_1'=0$, and $n_2=b$. Since the S-depth of $R$ is zero, this choice makes $R_1'$ balanced, hence tileable. A straightforward comparison of the number of up- and down-pointing unit triangles\footnote{ Using also the fact that the difference between the number of unit triangles of the two orientations in a lattice hexagon is equal to the difference between the lengths of any two opposite sides.} in $R_2'$ shows that, with this choice of $n_2$, $R_2'$ is balanced precisely if
\begin{equation}
n_2'=e'-e-c-b+n_3,  
\label{eka}
\end{equation}
where $e$ and $e'$ are the lengths indicated in Figure \ref{fkb}.

We claim that to finish the proof of the tileability of $R$, it suffices to show that there exist a solution $n_2'$ and $n_3$ of \eqref{eka} satisfying
\begin{align}
0&\leq n_2'\leq f
\label{ekb}
\\
0&\leq n_3\leq \delta,
\label{ekc}
\end{align}
where as before the focal length is $f+a'+b'+c'$, and $\delta$ is the distance (measured in lattice spacings) between the $a$-lobe and the top side of $R$. Indeed, by \eqref{eka}, $R_2$ is balanced, hence tilable. Choose $n_3'=f-n_2'$. This guarantees that $T'$ is balanced, so it is tilable. Furthermore, since $R$, $R_1'$, $R_2'$ and $T'$  are all balanced, it follows that $R_3'$ is also balanced, hence tilable. This completes the proof for case (2).

We now show that \eqref{eka} has solutions satisfying \eqref{ekb} and \eqref{ekc}. We have
\begin{align}
n_2'\leq f &\Leftrightarrow e'-e-c-b+n_3\leq f\nonumber\\
&\Leftrightarrow(f-y+b+c)-c-b+n_3\leq f\nonumber\\
&\Leftrightarrow n_3\leq y,
\label{ekd}
\end{align}
where we used that $e'-e=(f+a'+b'+c'+b+c)-(y+a'+b'+c')=f-y+b+c$ (which holds due to the fact that in the lattice hexagon $R_2$ the difference between the lengths of opposite sides is the same).

Similarly, we have
\begin{align}
n_2'\geq 0 &\Leftrightarrow e'-e-c-b+n_3\geq 0\nonumber\\
&\Leftrightarrow n_3\geq y-f.
\label{eke}
\end{align}
Therefore, in order to prove that \eqref{eka} has solutions satisfying \eqref{ekb} and \eqref{ekc}, we need to show that the intervals $[y-f,y]$ and $[0,\delta]$ are not disjoint. For this it suffices to prove that we cannot have $\delta<y-f$. Since in the hexagon $R_2$ the difference between the lengths of opposite sides is the same, have that 
\begin{align}
f-y+b+c&=(f+a'+b'+c'+b+c)-(y+a'+b'+c')\nonumber\\
&=(z+a+b+c)-(\delta+a)\nonumber\\
&=z-\delta+b+c,
\label{ekf}
\end{align}
and therefore $f-y=z-\delta$. But then $y-f=\delta-z\leq\delta$, and the proof is complete.

\section{Concluding remarks}

In this paper we presented a simple product formula that relates the number of lozenge tilings of two triad hexagons (hexagonal regions with a triad of bowties removed from them) that can be obtained from one another by a sequence of bowtie squeezing operations. One new aspect of this formula is that the number of tilings of the two involved regions is not given in general by a simple product formula, but their ratio always is (see \cite{Lshuffle}, \cite{LRshuffle} and \cite{Byun} for another similar phenomenon). Several previous results from the literature readily follow from our result, including  Lai's formula \cite{Lai3dent} for the number of lozenge tilings of hexagons with three dents, and Ciucu and Krattenthaler's formula \cite{ff} concering hexagons with a removed shamrock (see Section 2). 


Another new aspect is that our formula is conceptual --- it is determined by the geometry of the triad of bowties and the distances from the focal points to the sides of the hexagon. This provided us with three advantages: $(i)$ we were able to avoid the somewhat tedious splitting into the cases when $x$, $y$ and $z$ do or do not have the same parity, $(ii)$ the base cases and the verification that the claimed formula satisfies the recurrence could be handled essentially with no calculations, just looking at the relevant figures, and $(iii)$ we were able to present the full details of the calculations in the relatively short sections 7 and 9.




\end{document}